\documentclass[10pt,bookmarks]{article}

\usepackage{import}
\usepackage{commands}
\usepackage{constants}
\usepackage{float}

\newconstantfamily{regularity}{
  symbol=\varrho,
  format=\arabic
}

\newconstantfamily{dfan}{
  symbol=C,
}

\newconstantfamily{coercivity}{
  symbol=\gamma,
  format=\arabic
}

\definecolor{IFPENGreenn}{rgb}{0.66,    0.78,    0}

\definecolor{IFPENOrange}{rgb}{1.0, 0.56, 0.0}

\newcommand{\erased}[1]{}

\newcommand{\vertiii}[1]{{\left\vert\kern-0.25ex\left\vert\kern-0.25ex\left\vert #1 
    \right\vert\kern-0.25ex\right\vert\kern-0.25ex\right\vert}}
\newcommand{\matM}{\boldsymbol{M}}
\newcommand{\matA}{\boldsymbol{A}}
\newcommand{\vecp}{\boldsymbol{p}}
\newcommand{\vecb}{\boldsymbol{b}}
\newcommand{\vece}{\boldsymbol{e}}
\newcommand{\vecr}{\boldsymbol{r}}

\newcommand{\vecepsi}{\boldsymbol{\varepsilon}}
\newcommand{\vecrho}{\boldsymbol{\varrho}}
\newcommand{\scas}{s}

\newcommand{\vecw}{\boldsymbol{w}}

\newcommand{\R}{\mathbb{R}}

\newcommand{\WellIndex}{\mathtt{WI}}
\newcommand{\EnergyPercent}{\mathtt{E}}
\newcommand{\COO}{CO${}_2$}

\newcommand{\calN}{\mathcal{N}}

\newcommand{\sfN}{\mathsf{N}}

\newcommand{\calE}{\mathcal{E}}

\newcommand{\calL}{\mathcal{L}}
\newcommand{\calP}{\mathcal{P}}

\newcommand{\calT}{\mathcal{T}}

\newcommand{\calS}{\mathcal{S}}
\newcommand{\calM}{\mathcal{M}}
\newcommand{\calR}{\mathcal{R}}

\newcommand{\dsp}{\displaystyle}
\newcommand{\n}{\mathbf{n}}
\newcommand{\bLambda}{\mathbf{\Lambda}}
\newcommand{\vectmat}[1]{\mathbf{#1}}

\theoremstyle{plain}

\newtheorem{lemma}{Lemma}

\theoremstyle{remark}

\newcommand{\mean}[2]{\langle #1\rangle_{#2}}

\newcommand{\st}{s.t.\@\xspace}

\newcommand{\eqbydef}{\stackrel{\mathrm{def}}{=}}

\DeclareMathOperator{\opmeas}{m}

\DeclareMathOperator{\opint}{int}
\DeclareMathOperator{\opext}{ext}
\DeclareMathOperator{\opd}{d}

\newcommand{\meas}[1]{\opmeas_{#1}}

\newcommand{\closure}[1]{\overline{#1}}

\newcommand{\ud}{{\,\opd}}

\numberwithin{equation}{section}

\newcommand{\myred}[1]{\textcolor{red}{#1}}

\newcommand{\myblue}[1]{}

\newcommand{\orange}[1]{{#1}}

\begin{document}

\title{Reduced Basis method for finite volume simulations of parabolic PDEs \\
applied to porous media flows}

\author{
Sébastien \textsc{Boyaval}\thanks{LHSV, EDF R\&D, ENPC, Institut Polytechnique de Paris, Chatou, France. \texttt{sebastien.boyaval@enpc.fr}}  \thanks{Inria, 48 rue Barrault, 75013 Paris, France. \texttt{baptiste.guilbery@inria.fr}}  
\and
Guillaume \textsc{Ench\'ery}\thanks{IFP Energies nouvelles, 1 et 4 avenue de Bois Pr\'eau, 92852 Rueil-Malmaison Cedex, France. \texttt{guillaume.enchery@ifpen.fr}, \texttt{quang-huy.tran@ifpen.fr}} 
\and
Baptiste \textsc{Guilbery}${}^\dagger$
\and
Jana \textsc{Tarhini}${}^\ddagger$
\and
Quang-Huy \textsc{Tran}${}^\ddagger$
}
\date{\today}

\maketitle

\begin{abstract}
In the context of underground gas storage (of {\COO} typically) in saline aquifers, numerical simulation is a highly valuable tool to evaluate the impact of the uncertainties of various model parameters and to optimize, e.g., injection-production scenarios.
For this purpose, finite volume approximations of Darcy's parabolic model for 
flows in porous media are usually run many times, for many values of parameters like permeability and porosity, at costly computational efforts. 

We study the relevance of reduced basis methods as a way to lower the overall simulation cost of multi-point
finite volume approximations to Darcy's parabolic model for flows in porous media  
for different values of the parameters such as permeability. \orange{We aim to quickly evaluate any linear form (or quantity of interest) of the pressure for many parameter values.} To this end, we construct reduced bases 
by a standard POD-Greedy algorithm. 
Our POD-Greedy algorithm uses two new goal-oriented error estimators designed 
from a discrete space-time energy norm independent of the parameter.
We provide some numerical experiments that validate the efficiency of the proposed estimators.

\end{abstract}
{\small
\begin{center}
\textbf{Keywords}\medskip\\
single-phase flow, porous media, finite volumes, reduced basis, goal-oriented error estimate
\end{center}
\begin{center}
\textbf{Mathematics subject classification}\medskip\\
35J50, 65M08, 65N15, 76S05
\end{center}
} 

\section{Introduction}

\subsection{Context and motivation}
Geological storage of gases such as {\COO} requires computational models capable of describing fluid flows in complex geological formations. In particular, models of slightly compressible \emph{single-phase} Darcy flow are useful for optimizing injection strategies, assessing long-term behavior and quantifying uncertainties of storage sites. {\COO} can be injected into several types of underground formations, including depleted oil and gas reservoirs or saline aquifers in sedimentary basins. In this work, we focus on the latter.

In saline aquifers, large-scale simulations of single-phase Darcy flows are of crucial interest.
Indeed, the injection of {\COO} induces the displacement of brine in the surrounding aquifer, and the resulting pressure perturbation may propagate over very large distances. In risk-assessment studies, one needs to evaluate the pressure field in the surrounding aquifer, potentially along faults located far from the storage area. Such simulations have to be carried out \emph{many times}, for \emph{many values} of the uncertain geological and physical parameters. Moreover, one may be interested not only in the solution itself but in a specific quantity of interest (QOI), such as a time-dependent flux through a boundary of the storage domain, in order to optimize the injection process.

For a given set of parameters, the flow simulator has to solve a large linear system at every time step. Repeating such three-dimensional, time-dependent simulations for many parameter values leads to a substantial computational cost, especially for basin-scale domains.
For both purposes and given values of the model parameters, the flow simulator computes the solution of a large linear system many times.
This multi-query setting motivates the use of model-reduction techniques capable of replacing the repeated solution of large systems by a much less expensive surrogate model while providing a rigorous certfication of its accuracy.

\subsection{Reduced basis methods and previous works}
Several approaches have been developed to reduce the computational cost of large-scale basin and reservoir simulations. Spatial adaptivity, for instance, can reduce the number of degrees of freedom by locally refining the computational mesh in regions where the solution exhibits significant variations.
Adaptive Mesh Refinement (AMR) methods \cite{verfurth1994posteriori,berger1984adaptive} similarly construct and modify a hierarchy of grids according to prescribed error indicators or estimates as
the simulation proceeds. These techniques primarily aim to lower the cost of individual high-fidelity (HF) simulations.

In contrast, the Reduced Basis (RB) philosophy \cite{grepl,M2AN_2005__39_1_157_0,grepl2007efficient,Haasdonk2008,casenave2012accurate} is particularly well suited to multi-query problems. It replaces repeated calls of a parametrized HF model by the solution of a much smaller, therefore less expensive, low-fidelity (LF) surrogate model. Following a standard two-stage procedure \cite{grepl,Haasdonk2008}, the reduced-order model is obtained by projecting the HF problem onto a low-dimensional approximation space generated from selected HF solutions, or \emph{snapshots}, precomputed in the costly \emph{offline} stage. Once this space has been built, the computation of solutions and QOIs for previously unexplored parameter values becomes much more affordable in the \emph{online} stage.

%

Previous works such as \cite{grepl,M2AN_2005__39_1_157_0,grepl2007efficient,Haasdonk2008} have already considered the RB method in a parametrized parabolic context similar to our problem \eqref{SF}. Most notably,
\cite{Haasdonk2008,MR3518747} have applied the RB methodology with a HF simulator based on a cell-centered finite-volume (FV) two-point flux approximation (TPFA). 
These contributions demonstrate the potential of RB techniques for transient porous-media flow. However, to the best of our knowledge, the combination considered here, namely, a three-dimensional transient porous-media flow problem, a cell-centered finite-volume discretization using the average multi-point flux approximation (AvgMPFA), and \emph{goal-oriented}\footnote{i.e., with a quantity of interest which is linear with respect to the parametrized solution} RB reduction with rigorous a posteriori certification, has not previously been addressed in the literature. 

\subsection{Numerical and mathematical challenges}
Our problem presents three features that make both the model reduction and its certification substantially more challenging than in the settings usually considered in the RB literature.

First, we consider a three-dimensional time-dependent problem. 
The HF model consists of a sequence of large linear systems over the whole time interval, and the reduced approximation must accurately represent the evolution of the pressure field. 

Second, the HF model is based on the average multi-point flux approximation (AvgMPFA) finite-volume scheme \cite{schneider2017convergence} rather than the classical TPFA scheme. Thanks to their local conservation properties and their natural treatment of fluxes, finite-volume methods \cite{eymard2000finite} are widely used in industrial porous-media simulations and sometimes even preferred to finite-difference \cite{sammon1988analysis, brenier1991upstream} and finite-element \cite{chavent1986mathematical, chen2001degenerate} methods. TPFA is particularly attractive because of its simplicity and, on suitable grids, its stability and monotonicity properties. However, its consistency requires the computational grid to satisfy an appropriate orthogonality condition. On general non-$\bLambda$-orthogonal grids, this condition is not satisfied, which motivates the use of more general schemes such as MPFA \cite{aavatsmark1994discretization, schneider2017convergence}, mimetic finite differences \cite{brezzi2005family}, virtual-element methods \cite{BeiraoEtAl2013BasicsOfVEM}, and hybrid \cite{sushi} or vertex-centered FV \cite{eymard2012vertex}. We also mention non-linear schemes \cite{le2005schema,lipnikov2009interpolation,schneider2017convergence} that were designed in order to obtain monotone approximations and properties such as the positivity of the solutions or the maximum principle on these grids. 

In this work, we consider the average multi-point flux approximation finite-volume (AvgMPFA) method, as studied, for example, in \cite{schneider2017convergence}. Although this scheme does not in general preserve positivity or a discrete maximum principle, it is consistent on grids that are not $\bLambda$-orthogonal and provides a more accurate discretization for the resulting HF model. The latter is therefore representative of a more general and demanding industrial setting than the TPFA-based models considered in several previous RB studies. Moreover, the discrete operator is non-symmetric and depends non-affinely on the permeability parameter, which introduces additional difficulties in the construction of an efficient online reduced model.

Third, and most importantly, we consider a goal-oriented reduction problem. In many applications, the quantity of practical interest is not the entire solution field but rather a specific output derived from it, for example, a flux through a boundary. A reduced model may consequently provide a sufficiently accurate value of the QOI even when its global solution error is relatively large, or conversely may require a much larger basis if it is certified only through a global energy-norm estimate. This motivates error estimation directly in terms of the QOI. \orange{Applications to uncertainty quantification and optimization with respect to parameter variations motivate a uniform control of the RB approximation with a parameter-independent norm.}

The development of such goal-oriented a posteriori error estimates is the main mathematical contribution of this work. For a linear QOI, we introduce a dual problem and derive computable estimators that provide rigorous control of the output error. The estimators are designed specifically for the transient FV setting considered here and complement the standard energy-norm estimates used to control the primal and dual solutions.

\subsection{Contributions of this work}
The objective of this work is therefore to develop and assess a certified goal-oriented reduced basis method for a three-dimensional, transient, parametrized single-phase Darcy flow model discretized by the AvgMPFA finite-volume scheme. Our main contributions are as follows.
\begin{enumerate}
\item \emph{New a posteriori error estimates}. We derive new rigorous a posteriori estimators for the primal and dual problems. In contrast with standard RB literature, we use a discrete space-time energy norm of $L^2([0,T];H^1(\Omega))$-type 
that is \emph{independent \orange{of} the parameter}. 
We also provide goal-oriented estimators for a linear QOI, 
like in \cite{grepl,M2AN_2005__39_1_157_0} (where a parametrized HF model based on finite element approximations is reduced using another algorithm than POD-Greedy to construct the LF model).
\item \emph{Efficient offline/online decomposition}. To recover an online computational complexity essentially independent of the number $\mathcal{N}$ of HF degrees of freedom, we employ the Empirical Interpolation Method (EIM) \cite{Maday2016ConvergenceAO,chaturantabut2010nonlinear} to handle the non-affine parameter dependence. A lower bound for the coercivity constant, required in the rigorous error estimation, is obtained using the Successive Constraint Method (SCM) \cite{CHEN20081295,HUYNH2007473}.
\item \emph{Numerically robust residual evaluation}. For accurate a posteriori error estimation taking care of machine precision, 
we numerically compute quadratic-in-the-parameter forms (in the dual norm of the residual) \orange{as in \cite{casenave2012accurate}}, using a specific orthonormal basis to represent the residual, while a traditional offline/online decomposition \cite{canuto2009posteriori} leads to numerical precision issues.
\end{enumerate}
The reduced basis is constructed using the standard offline/online paradigm. In the offline stage, HF simulations are performed for selected parameter values and the resulting snapshots are used to construct a low-dimensional approximation space, for instance through a POD-Greedy procedure driven by the proposed error estimates \cite{Haasdonk2008}. In the online stage, the reduced model is then used to compute the solution and the QOI for new parameter values at a cost that is largely independent of $\mathcal{N}$.

The numerical experiments assess both the reliability of the proposed a posteriori estimates and the efficiency of the resulting certified reduced model. In particular, we compare solution-error estimates with the proposed goal-oriented estimates and demonstrate that, for a prescribed QOI, the latter can provide accurate output predictions with substantially smaller reduced spaces. This illustrates the principal advantage of goal-oriented reduction: computational resources are concentrated on accurately predicting the quantity that matters for the application rather than uniformly approximating the entire solution.

\subsection{Outline of the paper}
The remainder of the paper is organized as follows. Section \ref{refScheme} introduces the single-phase flow (SPF) model and its discretization by the average multi-point flux approximation, thereby defining our HF numerical model. 
Section \ref{RB} presents the goal-oriented reduced basis framework. We derive the a posteriori error estimation for the primal and dual problems as well as for the output model, and describe the POD-Greedy algorithm to construct the reduced basis. Additionally, in Section \ref{CompAspect}, we recommend an appropriate numerical evaluation of the residual dual norm in order to avoid the impact of round-off errors on the error bound. Section \ref{NumRes} is devoted to numerical experiments assessing the reliability and efficiency (i.e. effectivity) of the proposed estimates. Finally, concluding remarks and perspectives for future work are given in Section \ref{Conclusion}.

\section{A parametrized High-Fidelity model 
}\label{refScheme}

In this section, we introduce the PDE modelling of single phase porous media flows, and its discretization by a Finite Volume (FV) method which defines our parametrized HF simulator in the sequel.

\subsection{A Darcy model of single phase porous media flows 
}\label{sec:darcy}

We consider the flow of a slightly compressible fluid saturating a porous rock within a connected and bounded polygonal domain $\Omega$ of $\mathbb{R}^3$ at times $t\in(0,T)$. The boundary $\partial\Omega=\Gamma_{\rm D} \cup \Gamma_{\rm N} $ of $\Omega$ is partitioned into a part where Dirichlet boundary conditions are applied, and a part where homogeneous Neumann boundary conditions are used.

The balance of the water volume combined with Darcy's law and with initial and boundary data leads to
\begin{subequations}\label{SF}
\begin{align}
\phi c_t\partial_t p -\nabla \cdot (\bLambda({\nabla} p +\rho g{\nabla} z))& = q, &\mbox{ in } & (0,T)\times \Omega,\\
\bLambda({\nabla} p+\rho g {\nabla} z) \cdot \n & = 0, & \mbox{ on } & (0,T)\times \Gamma_{\mathrm{N}},\\
 p & = p_{\mathrm{D}}, & \mbox{ on } & (0,T)\times \Gamma_{\mathrm{D}}, \label{eq:dirichlet}\\
 p(x,t=0) & =p^0(x), & \mbox{ in } & \Omega,
\end{align}
\end{subequations}
where $p$ denotes the fluid pressure,
$\bLambda=\overline{\kappa}/\mu$ the mobility tensor, $\overline{\kappa}$ the rock permeability tensor, $\mu$ the fluid viscosity, $\phi$ the rock porosity, $c_t$ the total compressibility, $\rho$ the fluid density, $g$ the gravity constant, $z$ the depth and $q(p)$ a well source term to be precised later (in Section~\ref{sec:MPFA} after discretization).  
We designate by $\n$ the unit normal vector outside the domain.

A typical multi-query setting (with parameter variations to be addressed by the RB method), occurs when the permeability tensor $\overline{\kappa}$ is uncertain. We here assume that the domain is made up of two areas, corresponding to two rock types (a reservoir one and a cap rock) each with constant isotropic permeabilities $\overline{\kappa}_1$ or $\overline{\kappa}_2$ (see e.g. Fig.~\ref{fig:domain}) so that
 \[
 \bLambda = 
  \begin{bmatrix}
    \Lambda & 0 & 0\\
    0 &  \Lambda & 0 \\
    0 & 0 & \Lambda
  \end{bmatrix}
\]  
where $\Lambda(x)= \sum_{i=1}^2 \frac{\kappa_i}{\mu} \mathbb{1}_{(i)}(x)$ and $\mathbb{1}_i$ denotes the indicator function of the region $i$.
\begin{figure}[H]
    \centering
    \hspace*{-1.5cm}
    \includegraphics[scale=0.8]{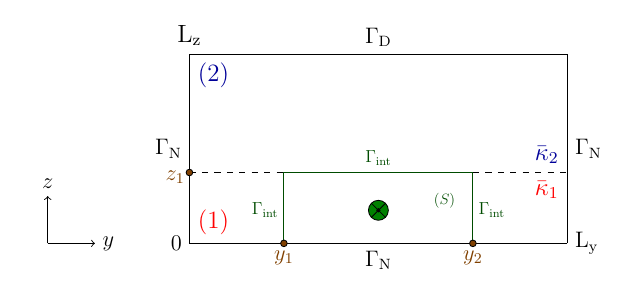}
    \caption{Typical domain configuration.}
    \label{fig:domain}
\end{figure}
Throughout this work, we consider the storage area $(S)$ with boundaries
$$\Gamma_{\rm int}=\{y_1\}\times[0,z_1] \cup\{y_2\}\times[0,z_1]\cup [y_1,y_2]\times\{z_1\}, $$
 where we seek to predict the time evolution of the flux $\scas$ defined  by
\begin{equation}\label{QOI}
\scas= -\int_{\Gamma_{\rm int}} \bLambda({{\nabla}} p +\rho {g} {\nabla} z)\cdot \boldsymbol{\n} \, \mathrm{d}S
\end{equation}
over $\Gamma_{\rm int}$ for many values of $\kappa_1$ and $\kappa_2$.

\subsection{Finite volume discretization}
\label{sec:MPFA}

The single-phase flow equations are first discretized in time using the implicit Euler method, next in space using a FV 
method. Choosing a constant time step
$$ \Delta t = \frac{T}{N}, \quad N \in \mathbb{N}^* \,, $$
we consider at each time iteration $n \in \{0,\ldots,N-1\}$ an approximation $p^{n+1} \approx p(t^{n+1})$ at $t^{n+1}=(n+1)\;\Delta t$ solution to
\begin{subequations}\label{eq::subequs}
\begin{align}
\phi c_t\frac{p^{ n+1}-p^{ n}}{\Delta t}-\nabla \cdot \left( \bLambda(\nabla p^{ n+1} +\rho  g\nabla z)\right)& = q^{n+1}, \label{subequ:SP}\\
\bLambda(\nabla p^{ n+1} +\rho g \nabla z) \cdot \n & = 0, \\
 p^{ n+1} & = p_{\mathrm{D}}.
\end{align}
\end{subequations}

The space discretization is performed using an \emph{admissible mesh} of $\Omega$, defined by a triplet $\mathcal{D}=(\mathcal{T},\mathcal{E},\mathcal{P})$ where : 
\begin{itemize}
\item  $\mathcal{T}$ is a finite set of non empty compact convex polygonal sub-domains of $\Omega$ (the set of cells), called control volumes such that $\overline{\Omega}=\bigcup\limits_{K \in \mathcal{T}}\overline{K}$. For all $K \in \mathcal{T}$, we denote by $\meas{K}>0$ its measure and set $\partial K\eqbydef \closure{K}\setminus K$.
\item $\mathcal{E}$ is a family of subsets of $\overline{\Omega}$ (the set of faces) such that for any $K\in \mathcal{T}$, there exists a subset $\mathcal{E}_K$ of $\mathcal{E}$ where $ \partial K=\bigcup\limits_{\sigma \in \mathcal{E}_K} \sigma$. For any $(K,L)\in\mathcal{T}^2$ with $K \neq L$, either the $(d-1)$ Lebesgue measure of $\overline{K}\cap \overline{L}$ is 0 or $\overline{K}\cap \overline{L}=\overline{\sigma}$ for some $\sigma \in \mathcal{E}$, with $\sigma=K|L$ (an interior face). We denote by $\meas{\sigma}$ the $(d-1)$-dimensional measure of $\sigma$. The sets of inner and boundary faces are denoted by $\calE_{ \opint}$ and $\calE_{\opext}$ respectively.
\item  $\calP=\lbrace \vectmat{x}_K\rbrace_{K \in \calT}$ is a collection of points within   $\Omega$ indexed by $\calT$ (called the \emph{cell centers}, not required to be the barycenters) \st $\vectmat{x}_K\in K$ and $K$ is star-shaped with respect to $\vectmat{x}_K$. 
\end{itemize}
For each cell $K\in \calT$ and face $\sigma \in \calE_K$, $\n_{K,\sigma}$ denotes the unit vector normal to $\sigma$ and pointing outward to $K$.
Additionally, for any cell $K\in\calT$ and any function $\Phi$ belonging to $ L^1(K)$, we define 
$\mean{\Phi}{K}\eqbydef \meas{K}^{-1}\int_K\Phi \ud x$.

Given an admissible mesh, numerically computable approximations $p^n_K\approx \mean{p^n}{K}$ are defined after space discretization by a FV method.
The FV method first integrates \eqref{subequ:SP} over a cell $K$ to obtain
\begin{equation}\label{eq:bilanintK}
    \int_{K}\phi c_t\frac{p^{ n+1}-p^{ n}}{\Delta t}\, \ud x - \int_{K} \nabla \cdot ( \bLambda(\nabla p^{ n+1}+\rho g \nabla z) )\, \ud x=\int_{K} q^{n+1} \, \ud x.
\end{equation}
Applying Green's formula and decomposing the boundary $\partial K$ into faces, the FV method next transforms the first two integrals and recasts \eqref{eq:bilanintK} as
\begin{equation}\label{eq:afterGreen}
 \meas{K} \phi_K c_t  \frac{\mean{p^{n+1}}{K}-\mean{p^n}{K}}{\Delta t} -\sum_{\sigma \in \calE_K} \int_{\sigma}\bLambda(\nabla p^{ n+1}+\rho g\nabla z) \cdot \n_{K,\sigma} \,  \ud S= \int_{K} q^{n+1} \, \ud x,
\end{equation}
which leads one to consider the following numerical scheme
\begin{equation}\label{AssembeledEqu}
 \meas{K}\phi_K c_t  (p_K^{n+1}-p_K^{n} ) +   \Delta t \sum_{\sigma \in \calE_K} F_{K,\sigma}^{n+1}= \Delta t \meas{K} q_K^{n+1}
\end{equation}
with numerical fluxes $F_{K,\sigma}^{n+1}\approx -\int_{\sigma}\bLambda(\nabla p^{ n+1}+\rho g\nabla z) \cdot \n_{K,\sigma}\,\ud S$ and numerical source terms $ q_K^{n+1}$ that allow for the numerical computation of $(p_K^{n+1})_{K\in\calT}$ given $(p_K^{n})_{K\in\calT}$.

The Peaceman model \cite{Chen2007ReservoirS} is used here for the well source term $q_K^{n+1} = \frac{1}{\meas{K}}\int_{K} q^{n+1} \, \ud x$. We suppose that the well is vertical and the perforations are oriented in the z-direction. The well model is then given by
\begin{equation}
\label{eq:wellFlux}
 q_K^{n+1}=  \frac{\WellIndex_K}{\meas{K}} (p_{bh}-p_K^{n+1}-\rho {g}(z_{bh}-z_K) ),   
\end{equation}
where $p_{bh}$ is the bottom hole pressure (at $z=z_{bh}$) and $\WellIndex_K$ is the Peaceman well index\footnote{
We adopt here a usual definition of the well index \cite{Chen2007ReservoirS} 
\begin{equation*}\label{eq:wellindex}
     \WellIndex_K=\frac{2\pi  h_{K, z} \sqrt{\bLambda_{K, xx} \bLambda_{K, yy}}}{\text{ln}(r_{e, K}/r_w)+s_d }
\end{equation*}
where
$h_{K,z}$ is the perforation height, $r_w$ the well radius, $s_d$ the skin factor i.e. a dimensionless number modeling the formation damage caused by drilling, and $r_{e, K}$ the Peaceman radius defined as 
$$r_e= \frac{0.14 [ (\bLambda_{K,yy}/\bLambda_{K,xx})^{1/2} h_{K,x}^2 + (\bLambda_{K,xx}/\bLambda_{K,yy})^{1/2} h_{K,y}^2  ]^{1/2}}{0.5 [ (\bLambda_{K, yy}/\bLambda_{K, xx})^{1/4} +(\bLambda_{K, xx}/\bLambda_{K, yy})^{1/4} ]}, $$
 where $h_{K,x}$ and $h_{K,y}$ are the grid sizes in $x$ and $y$ directions.
} in a perforated cell $K$. 

The flux -$\int_{\sigma}\bLambda(\nabla p^{ n+1}+\rho g\nabla z) \cdot \n_{K,\sigma} \,\ud S$ is numerically approximated using the average multi-point flux scheme studied in \cite{schneider2017convergence}.
For each interior edge $\sigma\in {\calE}_{\rm int}$, with $\calT_\sigma=\{K,L\}$ the approximated flux $F_{K,\sigma}^{n+1}$ is defined as a convex combination of two linear fluxes $\tilde{F}_{K,\sigma}^{n+1}$ and $\tilde{F}_{L,\sigma}^{n+1}$ such that
\begin{equation}
\label{eq:fluxApprox}
F_{K,\sigma}^{n+1}=  \mu_{K,\sigma} \tilde{F}_{K,\sigma}^{n+1}-\mu_{L,\sigma} \tilde{F}_{L,\sigma}^{n+1}, \qquad
\text{ with }\;\; \mu_{K,\sigma}\geq 0, \quad \mu_{L,\sigma}\geq 0, \quad \mu_{K,\sigma}+\mu_{L,\sigma}=1.
\end{equation}
A numerical flux formula as given by \eqref{eq:fluxApprox} is clearly conservative, i.e, 
\begin{equation}
F_{K,\sigma}^{n+1}+F_{L,\sigma}^{n+1}=0.
\label{conservation_flux_nonlinear}
\end{equation}

To build the linear fluxes $ \tilde{F}_{K,\sigma}^{n+1}$ in \eqref{eq:fluxApprox}, we approximate the pressure gradient $\nabla p$ in the direction of the conormal vector $\mean{\bLambda}{K}\n_{K, \sigma}$ after expressing the conormal as a linear combination of the vectors $(\mathbf{x}_{\sigma^{\prime}} - \mathbf{x}_{K} )_{\{\sigma'\in\calS_{K, \sigma}\}}$
\begin{equation}
\mean{\bLambda}{K}\n_{K, \sigma}\approx\sum_{\sigma^\prime \in \calS_{K, \sigma}} \alpha_{K,\sigma\sigma^\prime} (\mathbf{x}_{\sigma^{\prime}} - \mathbf{x}_{K})
.
\label{eq:finalConormalDec}
\end{equation}
The decomposition \eqref{eq:finalConormalDec} is achieved numerically by means of an optimization procedure which aims at reducing the sum of the coefficients $\alpha_{K,\sigma\sigma^\prime}$ and the size of the stencil $\calS_{K, \sigma} \eqbydef \lbrace \sigma^\prime\in\calE_K \, \alpha_{K,\sigma\sigma^\prime} \not = 0 \rbrace$ within $\calE_K$  \cite{schneider2018monotone}.
In \eqref{eq:finalConormalDec}, $\mathbf{x}_{\sigma}$ is a harmonic averaging interpolation point on face $\sigma$
(not necessarily the face center), that is,
\begin{equation}
\label{eq:xsigma}
    \dsp \mathbf{x}_{\sigma} = \omega_{K,{\sigma}} \mathbf{y}_K + \omega_{L,{\sigma}} \mathbf{y}_L  +  \frac{d_{K,\sigma} d_{L,\sigma}}{d_{L,\sigma} \tau_{K,\sigma} + d_{K,\sigma} \tau_{L,\sigma}}(\boldsymbol{\tau}_K^{\sigma}-\boldsymbol{\tau}_L^{\sigma})
\end{equation}
where 
\begin{subequations}\label{omega}  
\begin{alignat}{2}
 \omega_{K,{\sigma}} & = \frac{d_{L,\sigma} \tau_{K,\sigma}}{d_{L,\sigma} \tau_{K,\sigma} + d_{K,\sigma} \tau_{L,\sigma}}, & \quad
\dsp \omega_{L,{\sigma}} & = \frac{d_{K,\sigma} \tau_{L,\sigma}}{d_{L,\sigma} \tau_{K,\sigma} + d_{K,\sigma} \tau_{L,\sigma}},\\
\tau_{K,\sigma} & = \n_{K, \sigma}. \mean{\bLambda}{K}\n_{K, \sigma}, &\quad
\tau_{L,\sigma} & = \n_{L, \sigma}. \mean{\bLambda}{L}\n_{L, \sigma}, \\
\dsp \boldsymbol{\tau}_K^{\sigma} & =(\bLambda_K-\tau_{K,\sigma}\rm{Id})\n_{K,\sigma}, &\quad \boldsymbol{\tau}_L^{\sigma} &=(\bLambda_L-\tau_{L,\sigma}\rm{Id})\n_{L,\sigma},
\end{alignat}
\end{subequations}
In \eqref{omega}, $d_{K,\sigma}$ and $d_{L,\sigma}$ are the distances of the cell centers to $\sigma$, while $\vectmat{y}_K$ and $\vectmat{y}_L$ are their projections on $\sigma$ defined by (see Figure \ref{fig:averagingPoint})
$$\vectmat{y}_K=\vectmat{x}_K +d_{K,\sigma} \n_{K,\sigma}, \quad \vectmat{y}_L=\vectmat{x}_L+d_{L,\sigma} \n_{L,\sigma}.$$
\begin{figure}[H]
    \centering
    \includegraphics[scale=0.7]{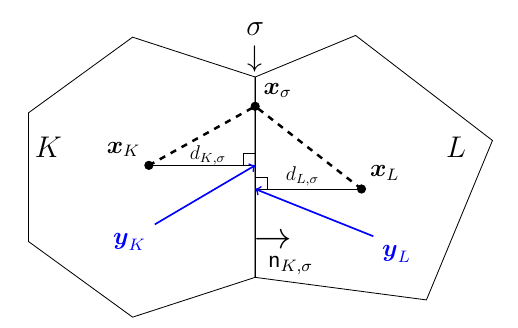}
    \caption{Harmonic averaging point.}
    \label{fig:averagingPoint}
\end{figure}
\noindent The pressure trace at $\sigma\in\calE_K$ is consistently reconstructed as 
\begin{equation}
I_{\sigma}p = \sum_{M \in \{K,L\}} \!\omega_{M,\sigma} p_M, 
\text{ where }\sum_{M \in \{K,L\}} \!\omega_{M,\sigma}  = 1,\qquad \omega_{M,\sigma} \geq 0,
\label{eq:InterpolatedSol}
\end{equation}
using the same weights $\omega_{M,\sigma}$ as in \eqref{omega} (for more details see \cite{agelas2009nine}). With the previous approximations, the linear fluxes read
\begin{equation}\label{eq:linearFlux}
\tilde{F}_{K,\sigma}^{n+1} = \meas{\sigma} \sum_{\sigma^\prime \in \calS_{K, \sigma}} \alpha_{K,\sigma\sigma^\prime} [ (p^{n+1}+\rho g z)_{K,\sigma^\prime}-(p_K^{n+1}+\rho g z_K) ]   ,
\end{equation}
where we used the notation
\begin{equation}\label{U_ksigma}
    (p)_{K, \sigma} = 
\begin{cases}
    I_{\sigma}p & \mbox{ if } \sigma=K|L, \\
    p_\sigma & \mbox{ otherwise.}
\end{cases}
\end{equation}

Finally, to define the numerical flux $F_{K,\sigma}^{n+1}$, the weights $\mu_{K,\sigma}$ and $\mu_{L,\sigma}$ in \eqref{eq:fluxApprox} can be chosen in different ways. In this work, we set $\mu_{K,\sigma}=\mu_{L,\sigma}=\frac{1}{2}$ if $\sigma=K|L$ and $\mu_{K,\sigma}=1$ if $\sigma \subset \partial \Omega$.
Let us remark that in \eqref{eq:linearFlux}, the average value of the boundary condition \eqref{eq:dirichlet}
is used on the face $\sigma^\prime$ if $\sigma^\prime \subset \Gamma_{\rm D}$. An additional unknown $p_{\sigma^\prime}$ is added if $\sigma^\prime \subset \Gamma_{\rm N}$ which can be computed by adding the homogeneous Neumann condition to the discrete system, namely:
\begin{equation}
\label{eq:NeumanFlux}
F_{K,\sigma}^{n+1}= 0.
\end{equation}

To enforce \eqref{eq:NeumanFlux}, we choose $\mathbf{x}_{\sigma}$ anywhere in the hyperplane containing $\sigma$ and $\alpha_{K,\sigma \sigma}$ such that:
\begin{equation}
    \mean{\bLambda}{K}\n_{K, \sigma}= \alpha_{K,\sigma \sigma} (\mathbf{x}_{\sigma} - \mathbf{x}_{K}),
\end{equation}
which amounts to a two-point flux approximation on this type of face.
Substituting this into \eqref{eq:NeumanFlux} leads to
\begin{equation}
    p^{n+1}_{\sigma}=p_K^{n+1}+\rho g (z_K-z_{\sigma}).
\end{equation}

Our numerical output quantity of interest (QOI) $\scas$ is then defined as
 \begin{equation}
 \label{eq:qoi}
 \scas^{n+1}=\sum_{\substack{\sigma =K|L \\ \sigma \subset \Gamma_{\rm int}}} F_{K,\sigma}^{n+1}.
\end{equation}

Our HF model consists in the solutions of the $N$ discrete linear systems obtained after assembling equations \eqref{AssembeledEqu} for all $K \in \calT$, $n \in \{0,\ldots,N-1\}$
\begin{equation}\label{fullOrderSystem}
(\matM +\Delta t \matA ) {\vecp}_\calM^{n+1}= \matM {\vecp}_\calM^{n}+ \Delta t \, {\vecb},
\end{equation}
see also \eqref{discreteFormParab} below, where
\begin{itemize}
\item $\matA \in \mathbb{R}^{\mathcal{N}\times\mathcal{N}} $ is a matrix containing the flux terms, with $\mathcal{N}$ the number of cells,
\item $\matM \in \mathbb{R}^{\mathcal{N}\times\mathcal{N}} $ is a diagonal matrix containing the quantities $\meas{K}\phi_K c_t$,
\item ${\vecp}_\calM^{ n+1}\in \mathbb{R}^{\mathcal{N}}$ is a vector composed of the values of the pressure $p^{n+1}$ in each element $K\in \calT$,
\item ${\vecb}\in \mathbb{R}^{\mathcal{N}}$ contains the Dirichlet condition values as well as the the gravity and source terms.
\end{itemize}
The QOI \eqref{eq:qoi} can be rewritten in
 \begin{equation}\label{DiscreteOutput}
 \scas^{n+1}=\boldsymbol{l}^T {\vecp}_\calM^{n+1}+ c,
\end{equation} 
with $\boldsymbol{l} \in \R^{\calN}$. We want to reduce the computational cost associated with multiple solutions of equations \eqref{AssembeledEqu} and \eqref{DiscreteOutput} 
when changing the permeability values $\kappa_1$ and $\kappa_2$ and assembling the corresponding  values of $\matA$, $\vecb$ and $\boldsymbol{l}$.

\section{Reduced Basis technique in 
time-dependent setting}\label{RB}

We now develop a goal-oriented RB procedure for
HF model \eqref{fullOrderSystem} and the QOI \eqref{DiscreteOutput}. 

In Section \ref{RBGeneral}, we define a LF model based on Galerkin projection. Then, in Section \ref{RBOffline}, we explain how to compute a Galerkin projection subspace by a POD-Greedy approach. The crux of our POD-Greedy approach is the minimization of an a posteriori error estimator of the QOI 
that is described in Section \ref{EstimatorsQ}.

\subsection{Low-fidelity model}\label{RBGeneral}

We consider the space $\R^{\calN}$ along with the inner product $\left\langle \cdot,\cdot \right\rangle_{\orange{\boldsymbol{A}^*}}$ and the corresponding norm $\|\cdot\|_{\orange{\boldsymbol{A}^*}}$, where $\orange{\boldsymbol{A}^*}$ is a symmetric positive definite matrix in $\R^{\calN\times\calN}$ that will be defined in the sequel (see Proposition \ref{Prop1}).

We consider the discrete space-time problem 
\begin{equation}\label{discreteFormParab}
    \begin{bmatrix}
\matM+ \Delta t \matA & 0    & 0     &  \ldots      & \ldots    & 0  \\
 - \matM & \matM + \Delta t \matA & \ddots & \ddots &        & \vdots\\
0 & - \matM & \ddots & \ddots & \ddots &  \vdots  \\
 \vdots  & \ddots & \ddots & \ddots & \ddots & 0  \\
\vdots &        & \ddots & \ddots & \ddots &  0 \\
0 & \ldots      &     \ldots    & 0    & - \matM       & \matM + \Delta t \matA
\end{bmatrix}
\begin{bmatrix}
\vecp_\calM^{1} \\
\vecp_\calM^{2} \\
\vdots\\
\vdots \\
\vecp_\calM^{N-1}\\
\vecp_\calM^{N}
\end{bmatrix}= \Delta t
\begin{bmatrix}
\, \vecb + \matM\vecp_\calM^{0} \\
\;\, \\
\, \vecb \;\, \\
\, \vdots \;\,\\
\, \vdots \;\, \\
\, \vecb \;\,\\
\, \vecb \;\,
\end{bmatrix},
\end{equation}
in a multi-query setting where \eqref{discreteFormParab} has to be solved for many values of a parameter contained in $\matA$ and $\vecb$. 
In the sequel, we generically denote $\xi$ that parameter which is simply $(\kappa_1,\kappa_2)\in\mathbb R^2$ in our application.

Numerical reduction (i.e. reduction of the computational cost without 
loss of accuracy) can be achieved by replacing \eqref{discreteFormParab}, for many parameter values, by its 
Galerkin projection onto a 
linear subspace spanned by a reduced basis $\boldsymbol{Z}_{\rm pr}\in \R^{\calN\times\sfN_{\rm pr}}$, $\sfN_{\rm pr} \ll \calN$.
The reduced solution $\vecp^{\mathsf{N}_{\rm pr},n+1}$ at time iteration $n+1$ is defined as 
\begin{equation}
\label{eq:reducedsolution}
 \vecp^{\mathsf{N}_{\rm pr},n+1}= \boldsymbol{Z}_{\rm pr}\; {\widetilde{\vecp}}^{n+1},
\end{equation}
where $ \widetilde{\vecp}^{n+1}$ is the solution to
\begin{equation}\label{primalReducedEq}
\big(\boldsymbol{Z}_{\rm pr}^T\;\matM \boldsymbol{Z}_{\rm pr} +\Delta t  \;\boldsymbol{Z}_{\rm pr}^T\;\matA \boldsymbol{Z}_{\rm pr} \big) {\widetilde{\vecp}}^{ n+1}=\boldsymbol{Z}_{\rm pr}^T\;\matM \boldsymbol{Z}_{\rm pr}\; {\widetilde{\vecp}}^{n}+ \Delta t\; \boldsymbol{Z}_{\rm pr}^T \; \vecb
\end{equation}
in the so-called \textbf{online phase}, once $\boldsymbol{Z}_{\rm pr}$ has been computed. To that aim, a reduced basis $\boldsymbol{Z}_{\rm pr}$ can first be computed in a so-called \textbf{offline phase}, e.g. using a POD-Greedy method.

\subsection{POD-Greedy method}\label{RBOffline}

The aim of a POD-Greedy algorithm is to iteratively construct matrices $\boldsymbol{Z}^{N_1}_{\rm pr}\in\mathbb R^{\calN\times N_1}$, $\boldsymbol{Z}^{N_2}_{\rm pr}\in\mathbb R^{\calN\times N_2}$, $\ldots$,
of rank $N_1<N_2<\ldots$ such that $\mathop{\rm Span}\boldsymbol{Z}^{N_i}_{\rm pr}$ is a 
subspace of the vector space $\mathop{\rm Span}\boldsymbol{Z}^{N_j}_{\rm pr}$ spanned by the column vectors of $\boldsymbol{Z}^{N_j}_{\rm pr}$ as soon as $N_i\le N_j$. At the end of the algorithm, $\boldsymbol{Z}^{\sfN_{\rm pr}}_{\rm pr}\equiv \boldsymbol{Z}_{\rm pr}$ is the so-called reduced basis useful for Galerkin projection in \eqref{primalReducedEq}, of dimension $\sfN_{\rm pr} \ll \calN$. 

The final iteration is reached when some projection error is under a fixed tolerance $\epsilon_{\rm tol} >0$, for instance 
$$ \vertiii{ \vece^n}_{\rm pr}(\xi) \leq \epsilon_{\rm tol} \quad \forall\xi \in \Xi, $$
where $\vertiii{\cdot}_{\rm pr}$ is a norm on the HF space $\R^{\calN}$ (see proposition \ref{Prop1}), 
\begin{equation}\label{primalRedError}
   \vece^n(\xi)=\vecp_\calM^{n}(\xi)-\vecp^{\mathsf{N}_{\rm pr},n}(\xi)
\end{equation}
is the "primal" reduction error at parameter value $\xi$, and $\Xi:=\{\xi_1,\ldots,\xi_{\calL}\}$ is a training set of parameter values. 
Between two iterations, a scalar $ric \in [0,1]$ controls the increase $N_{i+1}-N_i$  
consisting in adding to $\boldsymbol{Z}^{N_i}_{\rm pr}\in\mathbb R^{\calN\times N_i}$ the largest 
$N_{i+1}-N_i$ POD modes of a "snapshot" matrix (collecting the time evolution of the projection error at a new selected parameter $\xi_{\ell}$) (see Algorithm \ref{alg:PrimalAlgoP}). 
This type of basis-increase between two iterations 
has been introduced 
in \cite{Haasdonk2008} and has remained a standard since, when one iteratively constructs a reduced basis with a greedy-type algorithm that iteratively selects parameter values $\xi_{\ell_1},\xi_{\ell_2},\ldots\in \Xi$ so as to increase $\boldsymbol{Z}^{N_i}_{\rm pr}\in\mathbb R^{\calN\times N_i}$ using the time trajectory $\{\vecp_\calM^{n+1}(\xi_{\ell_i})\}_{n=0}^{N-1}$ at iteration $i$.

The quality of the POD-Greedy selection (i.e. the accuracy reached by the approximation $\vecp^{\mathsf{N}_{\rm pr},n}(\xi)\approx\vecp_\calM^{n}(\xi)$ 
for all $n=1,\ldots,N$ and $\xi\in\Xi$) strongly depends on the parameter $\xi_{\ell_i}$ selected at iteration $i$. 
To that aim, we propose 
to choose $\xi_{\ell_i}$ as a maximizer of a new 
a posteriori estimator $\Delta_{\rm{pr}}^N(\xi)$ of the reduction error 
$$\vertiii{\vecp_\calM^{N}(\xi)-\vecp^{\sfN_{\rm pr},N}(\xi)}_{\rm pr} \leq \Delta_{\rm{pr}}^N(\xi) $$
among all training parameter values $\xi\in\Xi$.

\begin{algorithm}[H]
 \begin{algorithmic}[1]
  \caption{POD-greedy algorithm using $\Delta_{\rm{pr}}^N$}
 \State \textbf{Procedure} $\boldsymbol{Z}_{\rm{pr}}=\textbf{POD-Greedy}(\mathsf{N}_{ \max},\epsilon_{\rm tol},\Xi, ric).$
 \State$ \mathsf{N}_{\rm pr}=1$, $\delta^{ \mathsf{N}_{\rm pr}}=\epsilon_{\rm tol}+1.$
 \State  Take $\xi_1 \in \Xi$, $\ell=1$ and set $\Xi^{\ell}=\{\xi_1\}.$
 \While{$\delta^{ \mathsf{N}_{\rm pr}} > \epsilon_{\rm tol}$ and $\rm \mathsf{N}_{\rm pr} <  \mathsf{N}_{\max}$}.
 \State Compute $\vecp_\calM^{n}(\xi_{\ell})$ for $0\leq n \leq N.$
\If{$\ell=1$}
 \State Set $\mathbf{S}_{\rm{pr}} := \big[\vecp_\calM^{0}(\xi_{1}) \big| \ldots \big|\vecp_\calM^{N}(\xi_{1})\big]$.
 \Else
 \State Set $\mathbf{S}_{\rm{pr}} := \big[\vecp_\calM^{0}(\xi_{\ell})- {\rm{Proj}}_{\boldsymbol{Z}_{\rm {pr}}}(\vecp_\calM^{0}(\xi_{\ell})) \big| \ldots \big|\vecp_\calM^{N}(\xi_{\ell})- {\rm{Proj}}_{\boldsymbol{Z}_{\rm{ pr}}} (\vecp_\calM^{N}(\xi_{\ell}) )\big]$.
 \EndIf
 \State  Compute $\big [{\boldsymbol{z}}_1|\ldots|\boldsymbol{z}_{\delta \sfN_{\rm pr}} \big ] =\textbf{POD}(\mathbf{S}_{\rm{pr}}, ric) $ using Algorithm \ref{algo:podric}.
  \State Define $\boldsymbol{Z}^{\mathsf{N}_{\rm pr} +\delta \mathsf{N}_{\rm pr} }_{\rm{pr}}:=orthonormalize (\boldsymbol{Z}^{\mathsf{N}_{\rm pr}}_{\rm{pr}}\cup
  \big [{\boldsymbol{z}}_1|\ldots|\boldsymbol{z}_{\delta \sfN_{\rm pr}} \big ])$ using Algorithm \ref{algo:GramSchmidt}.
 \State Compute $\delta^{\mathsf{N}_{\rm pr}}= \underset{\xi \in \Xi}{\max} \; \Delta_{\rm{pr}}^N.$
 \State Set $\xi_{\ell+1}=\arg\underset{\xi \in \Xi}{\max}\;\Delta_{\rm{pr}}^N.$
 \State  $\Xi^{ \ell+1} \gets \Xi^{\ell}\cup \{\xi_{\ell+1}\}.$
 \State  ${ \mathsf{N}_{\rm pr}}\gets {\mathsf{N}_{\rm pr}+\delta\mathsf{N}_{\rm pr}}.$
 \State $\ell \gets \ell+1.$
 \EndWhile
 \label{alg:PrimalAlgoP}
 \end{algorithmic}
\end{algorithm}
 
\begin{algorithm}[H]
\hspace*{\algorithmicindent} \textbf{Input:} $\mathbf{S}_{\rm{pr}}\in\mathbb R^{\calN\times N}$, $ric\in(0,1)$ \\
\hspace*{\algorithmicindent} \textbf{Output:} $\boldsymbol{A}^*$-orthonormal $\big [{\boldsymbol{z}}_1|\ldots|\boldsymbol{z}_{\delta \sfN_{\rm pr}} \big ]$.  
\begin{algorithmic}[1]
\caption{POD with ric}
\For{$i \leftarrow 1,\ldots, N$}
\State $\boldsymbol{V}_i,\sigma_i$ $i$ largest vector and eigenvalue pair of 
$\boldsymbol{C}=\boldsymbol{S}_{\rm pr}^T \orange{\boldsymbol{A}^*} \boldsymbol{S}_{\rm pr}$
\EndFor
\State $\delta \sfN_{\rm pr} \leftarrow 1$
\While{$\EnergyPercent_{\delta \sfN_{\rm pr}}=\frac{\sum\limits_{n=1}^{\delta \sfN_{\rm pr}} \orange{\sigma_n}}{\sum\limits_{n=1}^{N} \orange{\sigma_n}} < ric$}
\State $\delta \sfN_{\rm pr} \leftarrow 1+\delta \sfN_{\rm pr}$
\State $\boldsymbol{z}_{\delta \sfN_{\rm pr}}=\frac{1}{\sqrt{\sigma_{\delta \sfN_{\rm pr}}}}\boldsymbol{S}_{\rm pr}\widetilde{\boldsymbol{V}}_{\delta \sfN_{\rm pr}}$ 
\EndWhile
\label{algo:podric}
\end{algorithmic}
\end{algorithm}

\begin{algorithm}[H]
\hspace*{\algorithmicindent} \textbf{Input:} vectors $\boldsymbol{v}_i$, $i\in 1,\ldots, \mathsf{N}_{\rm pr}$.\\
 \hspace*{\algorithmicindent} \textbf{Output:} orthonormal vectors $\boldsymbol{v}_i$.  
\begin{algorithmic}[1]
\caption{Gram-Schmidt with re-iteration}
\For{$\ell=1,2$}
\For{$i \leftarrow 1,\ldots, \mathsf{N}_{\rm pr}$}
\For{$j \leftarrow 1,\ldots,(i-1)$}
\State $\boldsymbol{v}_i \leftarrow \boldsymbol{v}_i - \langle \boldsymbol{v}_i,\boldsymbol{v}_j \rangle_{\boldsymbol{A}^*}\;\boldsymbol{v}_j.$
\EndFor
\EndFor
\State $\boldsymbol{v}_i=\boldsymbol{v}_i/\|\boldsymbol{v}_i\|_{\boldsymbol{A}^*}.$
\EndFor
\label{algo:GramSchmidt}
\end{algorithmic}
\end{algorithm}

\subsection{A posteriori estimation of the primal error}\label{EstimatorsP}

We define the residual of $\vecp^{\mathsf{N}_{\rm pr},n+1}$ as the following linear form
\begin{equation}\label{primalResidual}
\left\langle{\vecr}(\vecp^{\mathsf{N}_{\rm pr},n+1}),\boldsymbol{v} \right\rangle = \frac{1}{\Delta t} \left\langle (\matM +\Delta t \matA ) {\vecp}^{\mathsf{N}_{\rm pr}, n+1} - \matM {\vecp}^{\mathsf{N}_{\rm pr},n} - \Delta t {\vecb} , \, \boldsymbol{v} \right\rangle\,
\forall \boldsymbol{v}\in \R^{\calN}
\end{equation}
which we also denote 
$\vecr(\vecp^{\mathsf{N}_{\rm pr},n+1})={\vecr}^{n+1}$ for the sake of simplicity. The residual dual norm is next defined as
 \begin{equation}\label{ResidualNormPrimal}
 \|\vecr ^ {n+1}\|_{\matA^{*^{-1}}}= \underset{\boldsymbol{v}\in \mathbb{R}^{\mathcal{N}} }{\sup} \frac{\langle \vecr^{n+1},\boldsymbol{v}\rangle}{\|\boldsymbol{v}\|_{\matA^*}}
 \end{equation}
and we now propose to evaluate the primal reduction error $\vece^N$ \emph{a posteriori} (employing either of the two computable estimators $\Delta_{\rm pr}^{N}$ and $\overline{\Delta}_{\rm pr}^{N}$) using a new space-time energy norm $\vertiii{\cdot}_{\rm pr}$ independent from the parameter $\xi$.

\begin{prop}[Energy a posteriori error estimate for the primal problem]\label{Prop1}
We denote $\matA= \frac{1}{2} (\matA+\matA^T )+\frac{1}{2} (\matA-\matA^T) := \matA_{\rm sym}+\matA_{\rm skew}$ the symmetric and skew-symmetric parts of matrix $\matA$.
For any $\xi$, given lower bounds 
\begin{subequations}
\begin{align}
      \orange{\alpha_{\matA}(\xi)} & \orange{\le  \inf_{\boldsymbol{v} \in \mathbb{R}^{\mathcal{N}}  }\frac{\boldsymbol{v}^T \matA_{\rm sym} \boldsymbol{v}}{\|\boldsymbol{v}\|_{\boldsymbol{A}^*}^2 } , \label{alphaE}}\\
      \orange{\alpha(\xi)} & \orange{\le \inf_{\boldsymbol{v} \in \mathbb{R}^{\mathcal{N}}  }\frac{\boldsymbol{v}^T(\matM +\Delta t\matA_{\rm sym})\boldsymbol{v}}{\|\boldsymbol{v}\|_{\boldsymbol{A}^*}^2 } ,\label{alpha}}
\end{align}
\end{subequations}
there holds
\begin{equation}
\vertiii{\vece^N}_{\rm pr}:= \orange{\bigg(\Delta t\sum_{m=1}^{N}  \langle \vece^{m},\matA^*\vece^{m} \rangle  \bigg)^{\!1/2} }\leq \min(\Delta_{\rm pr}^{N},\overline \Delta_{\rm pr}^{N})
\end{equation}
where $\matA^*$ is the symmetric part of $\matA(\xi^*)$ for a specific parameter $\xi^*$, and
\begin{subequations}
\begin{align}
\Delta_{\rm pr}^{ N} & := \bigg(\frac{\Delta t^2}{ \orange{\alpha \;\alpha_{\matA} }}\sum_{m=1}^{N}(N+1-m)\| {\vecr}^{m}\|_{\matA^{*^{-1}}}^2  \bigg)^{\!1/2}, \label{PrimalEstimator}\\
\overline\Delta_{\rm pr}^{ N} & := \bigg(\frac{\Delta t}{ \orange{\alpha^2_{\matA}}}\sum_{m=1}^{N}\| {\vecr}^{m}\|_{\matA^{*^{-1}}}^2  \bigg)^{\!1/2}. \label{PrimalEstimatorbis}
\end{align}
\end{subequations}
\end{prop}
\begin{proof}
    See Appendix \ref{Aproof1}.
\end{proof}

\begin{remark}
 The 
 error norm $\vertiii{\vece^N}_{\rm pr}$ 
 is a discrete version of the \orange{equivalent}
 $L^2([0,T];H^1(\Omega))$ (semi-)norm
    \begin{equation}
        \int_0^T(\bLambda\nabla e).\nabla e dt
    \end{equation}
\end{remark}


Let us remark that the matrix $\matM$ is symmetric positive definite. Thus,
\begin{equation}\label{alphaM}
    \alpha_{\matM}:= \inf_{\boldsymbol{v} \in \mathbb{R}^{\mathcal{N}}}
 \frac{\boldsymbol{v}^T\matM\boldsymbol{v}}{\|\boldsymbol{v}\|_{\boldsymbol{A}^*}^2 }>0,
\end{equation}

and we have
\begin{equation}\label{CHECK}
    \frac{\alpha(\xi)}{\Delta t}\ge  \alpha_{\matA}(\xi) +\frac{\alpha_{\matM}}{\Delta t}. 
\end{equation}

\begin{remark}   
The estimator \eqref{PrimalEstimatorbis} is derived from the parameter-dependent norm error estimate (4.66) established in\cite{grepl}. By contrast with the latter, the estimators proposed herein use a parameter-independent norm. We give here two error bounds because there is no criterion that predicts which bound will be sharper; their relative sharpness is observed empirically and can vary with the test case and parameters. For this reason both estimates are stated, and their comparative performance is illustrated in section \ref{NumRes}. 
\end{remark}

\subsection{A posteriori estimation of the 
QOI error}\label{EstimatorsQ}

In our goal-oriented setting, one is mostly interested in the 
QOI \eqref{DiscreteOutput} for many values of $\xi$
\begin{equation}\label{output}
\langle \boldsymbol{l},\vecp_\calM^{n} \rangle = \scas^n.
\end{equation}

Then, a reduced basis can in fact be constructed after modifying the POD-Greedy Algorithm \ref{alg:PrimalAlgoP} based on $\Delta_{\rm pr}^N$ into a POD-Greedy based on an a posteriori estimator of the QOI reduction error, see Algorithm \ref{alg:OutputPAlgo}. 

To define an a posteriori estimator for a QOI linear in the primal problem, let us now introduce for all $n=1, \ldots, N$ a dual problem which evolves backward in time
\begin{subequations}\label{dualPb}
\begin{align}
\matM \boldsymbol{\psi}_{\calM,n}^{n} & = -\boldsymbol{l},\label{dualPbInit}
\\
(\matM +\Delta t \matA^T)\boldsymbol{\psi}_{\calM,n}^m & = \matM \boldsymbol{\psi}_{\calM,n}^{m+1}\qquad m= 0,\ldots,n-1.
\end{align}
\end{subequations}
Since $\matM$, $\matA^T$ and $\boldsymbol{l}$ do not depend on time, we only solve once the following problem:
 \begin{subequations}\label{NewDualPb}
\begin{align}
\matM \boldsymbol{\Psi}_{\calM}^{N} & = -\boldsymbol{l},
\\
(\matM +\Delta t \matA^T)\boldsymbol{\Psi}_\calM^n & = \matM \boldsymbol{\Psi}_\calM^{n+1}\qquad n= 0,\ldots,N-1. \label{NewDualPbb}
\end{align}
\end{subequations}
Then we appropriately shift the results by defining
\begin{equation}
\boldsymbol{\boldsymbol{\psi}}_{\calM,n}^m=\boldsymbol{\Psi}_\calM^{N-n+m} \qquad  m= 0,\ldots,n.
\end{equation}

Let us also introduce a reduced basis $\boldsymbol{Z}_{\rm{du}}$ for the dual problem, such that at time $t=n$ one computes an approximation of the dual solution in \eqref{NewDualPb} as 
\begin{equation}\label{dualReducedSol}
    \boldsymbol{\Psi}^{\mathsf{N}_{\rm du},n}=\boldsymbol{Z}_{\rm{du}}\; \widetilde{\boldsymbol{\Psi}}{}^{n} \qquad n= 0,\ldots,N,
\end{equation}
by a Galerkin projection, with $\widetilde{\boldsymbol{\Psi}}{}^{n}$ solution to 

\begin{align}
\big(\boldsymbol{Z}_{\rm{du}}^T\;\matM \boldsymbol{Z}_{\rm{du}} +\Delta t \;\boldsymbol{Z}_{\rm{du}}^T \;\matA^T \boldsymbol{Z}_{\rm{du}} \big) \widetilde{\boldsymbol{\Psi}}{}^n &=\boldsymbol{Z}_{\rm{du}}^T\;\matM\boldsymbol{Z}_{\rm{du}} \;\widetilde{\boldsymbol{\Psi}}{}^{n+1} \qquad n= 0,\ldots,N-1,\label{dualReducedEq}
\\
\boldsymbol{Z}_{\rm{du}}^T\;\matM \boldsymbol{Z}_{\rm{du}}\widetilde{\boldsymbol{\Psi}}{}^N&=- \boldsymbol{Z}_{\rm{du}}^T\; \boldsymbol{l}.\label{dualReducedEqFinal}
\end{align}

We denote by $\vecepsi^{n}=\boldsymbol{\Psi}_\calM^n-\boldsymbol{\Psi}^{\mathsf{N}_{\rm du},n}$ the dual reduction error at time iteration $n$ and $\vecrho^{n}$
the residual associated with 
the dual problem 
\begin{align}
    \left \langle{\vecrho}^{n} ,\boldsymbol{v} \right\rangle&= \frac{1}{\Delta t}  \big\langle (\matM +\Delta t\; \matA^T)\boldsymbol{\Psi}^{\mathsf{N}_{\rm du},n} - \matM \boldsymbol{\Psi}^{\mathsf{N}_{\rm du},n+1} , \, \boldsymbol{v} \big\rangle \qquad n= 0,\ldots,N-1, \label{eq:redDualPb}
    \\
    \left \langle{\vecrho}^{N} ,\boldsymbol{v} \right\rangle&=\langle - \boldsymbol{l} - \matM \Psi^{\sfN_{\rm{du}},N},\boldsymbol{v}\rangle,\label{eq:finaldualResidual}
\end{align}
with a residual dual norm taken as 
\begin{equation}\label{ResidualNormDual}
\|\vecrho ^ {n}\|_{\matA^{*^{-1}}}= \underset{\boldsymbol{v}\in \mathbb{R}^{\mathcal{N}}}{\sup} \frac{\langle \vecrho^{n},\boldsymbol{v}\rangle}{\|\boldsymbol{v}\|_{\matA^*}}  \qquad n= 0,\ldots,N.
\end{equation}

\begin{prop}[Energy a posteriori error estimate for the dual problem]\label{prop2}
Given lower bound
\begin{equation}\label{beta_M_A}
   \orange{ \alpha_{\matM}\leq \inf_{\boldsymbol{v}\in \mathbb{R}^n}\frac{\boldsymbol{v}^T\matM\boldsymbol{v}}{\|\boldsymbol{v}\|_{\matA^*}^2 }},
\end{equation}
and the data introduced in Prop.~\ref{Prop1}, there holds for all $\xi$
\begin{equation}
\vertiii{\vecepsi^{N}}_{\rm{du}}:= \orange{\bigg ( \Delta t \sum_{m=0}^{N-1} \langle \vecepsi^{m},\matA^*\vecepsi^{m} \rangle  \bigg )^{\!1/2}}  \leq \min(\Delta_{\rm du}^N,\overline\Delta_{\rm du}^N)
\end{equation}
in the space-time energy norm $\vertiii{\cdot}_{\rm du}$ independent of the parameter $\xi$, with
\begin{subequations}
\begin{align}
\Delta_{\rm du}^N & := \orange{\frac{1}{\alpha^{\!1/2}}\bigg ( \frac{\Delta t^2}{ \alpha_{\matA}} \sum_{m=0}^{N-1}(m+1) \|\vecrho^m\|_{\matA^{*^{-1}}}^2+  T\frac{\|\vecrho^N\|_{\matA^{*^{-1}}}^2}{\alpha_{\matM} }  \bigg )^{\!1/2},} \label{DualEstimator}\\
\overline\Delta_{\rm du}^N & :=  \orange{\frac{1}{\alpha_{\matA}^{\!1/2}}\bigg ( \frac{\Delta t}{ \alpha_{\matA} } \sum_{m=0}^{N-1} \|\vecrho^m\|_{\matA^{*^{-1}}}^2+  \frac{\|\vecrho^N\|_{\matA^{*^{-1}}}^2}{\alpha_{\matM}} \bigg )^{\!1/2}. }\label{DualEstimatorbis}
\end{align}
\end{subequations}
\end{prop}


\begin{proof}
    See Appendix \ref{Aproof2}.
\end{proof}

\begin{remark}
 The 
 error norm $\vertiii{\vecepsi^N}_{\rm du}$ 
 is also a discrete version of a (semi-)norm \orange{equivalent} to
 $L^2([0,T];H^1(\Omega))$ .
\end{remark}

\begin{remark}
    Similar to the primal error estimates, the error estimate \eqref{DualEstimatorbis} is derived from the parameter-dependent norm error estimate (4.81) established in \cite{grepl}.
\end{remark}

\begin{prop}[Output error evaluation]\label{prop3}

Given the data introduced in Prop.~\ref{Prop1}, one can define two reduced outputs:
\begin{itemize}
\item either
\begin{equation}\label{reducedOutput}
 {\scas}^{\mathsf{N}_{\rm s},n}= \left\langle{\boldsymbol{l}}, {\vecp}^{\mathsf{N}_{\rm pr},n}\right\rangle  + \Delta t\; \sum_{n'=0}^{n-1} \big\langle {\vecr}^{n'+1},\boldsymbol{\Psi}^{\mathsf{N}_{\rm du},N-n+n'} \big\rangle
\end{equation}
with approximation error bounded as
\begin{equation}
    |{\scas}^{N}-  {\scas}^{\mathsf{N}_{\rm \Tilde{s}},N}| \leq \min(\Delta_{{s}}^N,\overline{\Delta}_{{s}}^N,\hat{\Delta}_{{s}}^N)
\end{equation}
and with the error estimates defined by
\begin{subequations}
\begin{align}
    \Delta_{{s}}^N & :=  \orange{ \Big(\frac{1}{\alpha_{\matA}} \sum_{n=1}^N   \|\vecr^{n}\|_{\matA^{*^{-1}}}^2\Big)^{\!1/2} \left ( \frac{\Delta t^3}{ \alpha} \sum_{m=0}^{N-1}(m+1) \|\vecrho^m\|_{\matA^{*^{-1}}}^2 +T \frac{\|\vecrho^N\|_{\matA^{*^{-1}}}^2}{\alpha_{\matM}}\right )^{\!1/2}},
    \label{QOIEst}\\
  \overline{\Delta}_{{s}}^N & := \orange{\Big(\frac{1}{\alpha_{\matA}} \sum_{n=1}^N   \|\vecr^{n}\|_{\matA^{*^{-1}}}^2\Big)^{\!1/2} \left (\frac{\Delta t^2}{\alpha_{\matA}}\sum_{n=1}^N   \|\vecrho^{n}\|_{\matA^{*^{-1}}}^2+ \Delta t\frac{\| \vecrho^N\|_{\matA^{*^{-1}}}^2 }{\alpha_{\matM}}\right )^{\!1/2}},
  \label{QOIEstbis}\\
\hat{\Delta}_s^N & :=  \orange{\frac{\Delta t}{\alpha^{\!1/2}}\sum_{k=0}^{N-1} \| \vecr^{k+1}\|_{\matA^{*^{-1}}} \left (\frac{\Delta t}{ \alpha_{\matA}} \sum_{m=k}^{N-1} \|\vecrho^m\|_{\matA^{*^{-1}}}^2 + \frac{\| \vecrho^N\|_{\matA^{*^{-1}}}^2 }{\alpha_{\matM}}\right)^{1/2}} ;
  \label{QOIEstter}
\end{align}
\end{subequations}
\item or
\begin{equation}\label{reducedOutput2} \widetilde{\scas}^{\mathsf{N}_{\rm s},n}= \big\langle{\boldsymbol{l}}, {\vecp}^{\mathsf{N}_{\rm pr},n}\big\rangle
\end{equation}
with approximation error bounded as
\begin{equation}
    |\scas^{N}-  \Tilde{\scas}^{\mathsf{N}_{\rm \Tilde{s}},N}| \leq \min(\Delta_{\Tilde{s}}^N,\overline{\Delta}_{\Tilde{s}}^N,\hat{\Delta}_{\Tilde{s}}^N)
\end{equation}
and with the error estimates defined by
\begin{subequations}
\begin{align}
    \Delta_{\Tilde{s}}^N & := \Delta_s^N +\Delta t \sum_{n=0}^{N-1}| \langle \vecr^{n+1},\boldsymbol{\Psi}^{\sfN_{\rm du},n}\rangle |,\label{QOIEstWC}\\
\overline{\Delta}_{\Tilde{s}}^N&:=\overline{\Delta}_s^N+\Delta t \sum_{n=0}^{N-1}| \langle \vecr^{n+1},\boldsymbol{\Psi}^{\sfN_{\rm du},n}\rangle|, \label{QOIEstWCbis}\\
         \hat{\Delta}_{\Tilde{s}}^N & := \hat{\Delta}_s^N +\Delta t \sum_{n=0}^{N-1}| \langle \vecr^{n+1},\boldsymbol{\Psi}^{\sfN_{\rm du},n}\rangle|. \label{QOIEstWCter}
\end{align}
\end{subequations}
\end{itemize}
\end{prop}
\begin{proof}
    See Appendix \ref{Aproof3}.
\end{proof} 

\begin{remark}
    The output error estimate \eqref{QOIEstbis} has been already proposed in\cite{grepl}.
\end{remark}
The optimal selection between these two definitions, based on their accuracy and efficiency, will be elucidated in Section \ref{NumRes}. There, we present a comparative analysis of the numerical results obtained by POD-Greedy algorithms with $\Delta_s^N$, $\overline{\Delta}_s^N$, $\hat{\Delta}_s^N$, $\Delta_{\Tilde{s}}^N$, $\overline{\Delta}_{\Tilde{s}}^N$ and $\hat{\Delta}_{\Tilde{s}}^N$ , where the construction of $\boldsymbol{Z}_{\rm du}$ is simultaneous to that of $\boldsymbol{Z}_{\rm pr}$, see e.g. Algorithm \ref{alg:OutputPAlgo}. \orange{Both basis are constructed independently and can have different sizes $N_{\rm{pr}}$ and $N_{\rm{du}}$}

\begin{algorithm}[htb]
 \begin{algorithmic}[1]
  \caption{POD-Greedy Algorithm with $\Delta_s^N$}
 \State \textbf{Procedure} $[\boldsymbol{Z}_{\rm{pr}},\boldsymbol{Z}_{\rm{du}}]=\textbf{POD-Greedy}(\mathsf{N}_{ \max},\epsilon_{\rm tol},\Xi, ric).$
  \State$ \mathsf{N}_{\rm pr}=0$, $ \mathsf{N}_{\rm du}=0.$ 
 \State $\delta^{ \mathsf{N}_{\rm s}}=\epsilon_{\rm tol}+1.$
 \State  Take $\xi_1 \in \Xi$, $\ell=1$ and set $\Xi^{\ell}=\{\xi_1\}.$
 \While{$\delta^{ \mathsf{N}_{\rm s}} > \epsilon_{\rm tol}$ and $\rm \mathsf{N}_{\rm pr} <  \mathsf{N}_{\max}$}.
 \State Compute $\vecp_\calM^{n}(\xi_{\ell})$ for $0\leq n \leq N.$
 \State Compute $\boldsymbol{\Psi}_\calM^{n}(\xi_{\ell})$ for $0\leq n \leq N-1.$
 \If{$\ell=1$}
 \State Set $\mathbf{S}_{\rm{pr}} := \big[\vecp_\calM^{0}(\xi_{1}) \big| \ldots \big|\vecp_\calM^{N}(\xi_{1})\big]$.
  \State Set $\mathbf{S}_{\rm{du}} := \big[\Psi_\calM^{0}(\xi_{1}) \big| \ldots \big|\Psi_\calM^{N}(\xi_{1})\big]$.
  \Else
  \State Set $\mathbf{S}_{\rm{pr}} := \big[\vecp_\calM^{0}(\xi_{\ell})- {\rm{Proj}}_{\boldsymbol{Z}_{\rm {pr}}}\big(\vecp_\calM^{0}(\xi_{\ell})\big) \big| \ldots \big|\vecp_\calM^{N}(\xi_{\ell})- {\rm{Proj}}_{\boldsymbol{Z}_{\rm {pr}}}\big(\vecp_\calM^{N}(\xi_{\ell})\big) \big]$.
  \State Set $\mathbf{S}_{\rm{du}} :=  \big[\boldsymbol{\Psi}_\calM^{0}(\xi_{\ell})- {\rm{Proj}}_{\boldsymbol{Z}_{\rm {du}}}\big(\boldsymbol{\Psi}_\calM^{0}(\xi_{\ell})\big)  \big| \ldots  \big|\boldsymbol{\Psi}_\calM^{N}(\xi_{\ell})- {\rm{Proj}}_{\boldsymbol{Z}_{\rm {du}}}\big(\boldsymbol{\Psi}_\calM^{N}(\xi_{\ell})\big) \big]$.
 \EndIf
 \State  Compute $\widetilde{\boldsymbol{Z}}{}^{ \delta \mathsf{N}_{\rm pr}}_{\rm{pr}}=\textbf{POD}(\mathbf{S}_{\rm{pr}}, ric) $.
 \State  Compute $\widetilde{\boldsymbol{Z}}{}^{ \delta \mathsf{N}_{\rm du}}_{\rm{du}}=\textbf{POD}(\mathbf{S}_{\rm{du}}, ric) $.
  \State Define $\boldsymbol{Z}^{\mathsf{N}_{\rm pr} +\delta \mathsf{N}_{\rm pr} }_{\rm{pr}} := orthonormalize(\boldsymbol{Z}^{\mathsf{N}_{\rm pr}}_{\rm{pr}}\cup\{\widetilde{\boldsymbol{Z}}{}^{\delta \mathsf{N}_{\rm pr}}_{\rm{pr}} \})$ using Algorithm \ref{algo:GramSchmidt}.
   \State Define $\boldsymbol{Z}^{\mathsf{N}_{\rm du} +\delta \mathsf{N}_{\rm du} }_{\rm{du}} := orthonormalize(\boldsymbol{Z}^{\mathsf{N}_{\rm du}}_{\rm{du}}\cup\{\widetilde{\boldsymbol{Z}}{}^{\delta \mathsf{N}_{\rm du}}_{\rm{du}} \})$ using Algorithm \ref{algo:GramSchmidt}.
    \State Compute $\delta^{\mathsf{N}_{\rm s}}= \underset{\xi \in \Xi}{\max} \;\Delta_s^N.$
 \State Set $\xi_{\ell+1}=\arg\underset{\xi \in \Xi}{\max}\;\Delta_s^N.$
 \State  $\Xi^{ \ell+1} \gets \Xi^{\ell}\cup \{\xi_{\ell+1}\}.$
  \State  ${ \mathsf{N}_{\rm pr}}\gets {\mathsf{N}_{\rm pr}+\delta\mathsf{N}_{\rm pr}}.$
 \State  ${ \mathsf{N}_{\rm du}}\gets {\mathsf{N}_{\rm du}+\delta\mathsf{N}_{\rm du}}.$
 \State $\ell \gets \ell+1.$
 \EndWhile
 \label{alg:OutputPAlgo}
 \end{algorithmic}
 \end{algorithm}

\subsection{Computational aspects}\label{CompAspect}

Having addressed the offline computation of the reduced basis, we notice that for online computations, the reduced systems \eqref{primalReducedEq} and \eqref{dualReducedEq}, required to evaluate  \eqref{QOIEst}, \eqref{QOIEstWC}, depend on  $\calN$. An affine decomposition strategy is a classical way to ensure the rapid assembly of the reduced system in the online stage. In the following, we also discuss the practical computation of the residual dual norm and the coercivity constant. 
 
 \paragraph{Affine decomposition.}
To construct the reduced matrix $\matA^{\sfN_{\rm pr}}\in \R^{\sfN_{\rm pr}\times\sfN_{\rm pr}}$ and reduced vector $\vecb^{\sfN_{\rm pr}}\in \R^{\sfN_{\rm pr}}$ defined as $\matA^{\mathsf{N}_{\rm pr}}=\boldsymbol{Z}_{\rm pr}^T\matA \boldsymbol{Z}_{\rm pr}$ and $\vecb^{\sfN_{\rm pr}}=\boldsymbol{Z}_{\rm pr}^T \vecb$ in equation \eqref{primalReducedEq}, we still need to compute $\matA$ and $\vecb$ depending on $\xi$.  To compute the reduced quantity $\scas^{N_s,n}$, the quantity $\boldsymbol{l} ^TZ_{\rm pr}\Tilde{p}^{n}$ should be evaluated too. These evaluations, which depend on $\mathcal{N}$, are detrimental to the rapid online evaluation of the reduced basis solution when varying the parameter values. To accelerate the construction, we rewrite $ \matA$, $\vecb$ and $\boldsymbol{l}$ as 
\begin{equation}\label{affineDecomp}
\matA=\sum_{d=1}^{\rm{D_A}} \theta_d^{A}(\xi)\matA_d, \qquad \vecb=\sum_{d=1}^{\rm{D_b}} \theta_d^b(\xi)\vecb_d 
, \qquad\boldsymbol{l}=\sum_{d=1}^{\rm{D_{\boldsymbol{l}}}} \theta_d^{\boldsymbol{l}}(\xi)\boldsymbol{l}_d.
\end{equation} 
In \eqref{affineDecomp}, $\matA_d \in \R^{\calN\times\calN}$, $1\leq d \leq {\rm{D_A}}$, $\vecb_d\in \R^{\calN}, 1\leq d \leq {\rm{D_b}}$ and $\boldsymbol{l}\in \R^{\calN}$, $1\leq d \leq {\rm{D_{\boldsymbol{l}}}}$ do not depend on $\xi$ and they are computed and stored once during the whole offline stage. Thus, for each new parameter $\xi$, we only have to compute the three sets of scalars $\{\theta_d^A(\xi)\}_{d=1}^{\rm{D_A}}$, $\{ \theta_d^b(\xi)\}_{d=1}^{\rm{D_b}}$ and $\{ \theta_d^{\boldsymbol{l}}(\xi)\}_{d=1}^{\rm{D_{\boldsymbol{l}}}}$ and assemble $ \matA^{\mathsf{N}_{\rm pr}}$, $ \vecb^{\mathsf{N}_{\rm pr}}$ and $\boldsymbol{l}^{\mathsf{N}_{\rm pr}}$. This operation only depends on the dimension $\mathsf{N}_{\rm pr}$ of the reduced basis. Note that, in our case this affine decomposition does not exist. We therefore use the Empirical Interpolation Method (EIM) \cite{Maday2016ConvergenceAO,grepl2007efficient} (see also Appendix \ref{EIM}) to build such an approximation for $\matA$, $\vecb$ and $\boldsymbol{l}$.

\begin{remark}

The affine decomposition of $\boldsymbol{l}$ enables us to solve \eqref{dualPbInit} exactly. Specifically we have:
$$
 \boldsymbol{\psi}_{\calM,n}^{n} = -\matM^{-1}\boldsymbol{l}=-\sum_{d=1}^{\rm{D_{\boldsymbol{l}}}} \theta_d^{\boldsymbol{l}}(\xi)\matM^{-1}\boldsymbol{l}_d,
$$
where the terms $\matM^{-1}\boldsymbol{l}_d$ are independent of $\xi$ and can be precomputed and added to the basis $\boldsymbol{Z}_{\rm{du}}$. Consequently, the final dual residual vanishes and $\|\vecrho ^ {N}\|_{\matA^{*^{-1}}}=0$.

\end{remark}

\begin{remark}

Another affine decomposition, can be obtained by applying the Empirical Interpolation Method (EIM) to the coefficients $\alpha_{K,\sigma \sigma^\prime}$ defining the linear fluxes. This would allow for a straightforward reconstruction of the matrix $\matA$ and vectors $\vecb, \boldsymbol{l}$. Furthermore, since the positive definiteness of $\matA$ is guaranteed if all coefficients $\alpha_{K,\sigma \sigma^\prime}$ are positive, this method provides a simple criterion to verify the stability of the EIM reconstruction. However, as the Average MPFA scheme does not necessarily yield a positive definite matrix by construction, this property is less critical here. Consequently, we choose to apply the EIM directly to the vectors $\vecb,\boldsymbol{l}$ and to the non-zero entries of $\matA$. This strategy proves more efficient, as it requires fewer interpolation points to decompose the operators compared to the coefficient-based approach. 
\end{remark}

Regarding the online phase, the assembly of the left-hand and right hand sides in \eqref{primalReducedEq} requires $\mathcal{O}((\rm{D_A}+1)\sfN_{\rm pr}^2)$ and $\mathcal{O}(\rm{D_b}\sfN_{\rm pr})$ operations respectively. The reduced system is then solved with a complexity of $\mathcal{O}(\sfN_{\rm pr}^3+N\sfN_{\rm pr}^2)$.

\paragraph{Residual norm evaluation.} Using the affine decomposition and the fact that $\|\vecr^{n+1}\|_{\matA^{*^{-1}}}^2=(\vecr^{n+1})^T(\matA^{*})^{-1}\vecr^{n+1}$ (which results from Cauchy-Schwartz inequality), we can now rewrite the dual norm of the residual as
\begin{align}\label{ResidualTrad}
\|\vecr^{n+1}\|_{\matA^{*^{-1}}}^2&=\sum_{d=1}^{\rm{D}_b}\sum_{d'=1}^{\rm{D}_b}\theta_d^b(\xi)\theta_{d'}^b(\xi)\;\vecb_d^T (\matA^{*})^{-1}\vecb_{d'} 
-2\sum_{d=1}^{\rm{D}_a}\sum_{d'=1}^{\rm{D}_b}\theta_d^A(\xi)\theta_{d'}^b(\xi)\, (\widetilde{\vecp}{}^{n+1})^T\boldsymbol{Z}_{\rm pr}^T\matA_d^T (\matA^{*})^{-1}\vecb_{d'}
\nonumber\\&
 + \sum_{d=1}^{\rm{D}_a}\sum_{d'=1}^{\rm{D}_a}\theta_d^A(\xi)\theta_{d'}^A(\xi)\,(\widetilde{\vecp}{}^{n+1})^T\boldsymbol{Z}_{\rm pr}^T\matA_d^T  (\matA^{*})^{-1}\matA_{d'}\boldsymbol{Z}_{\rm pr}\widetilde{\vecp}^{n+1}
 \nonumber\\&
 -\frac{2}{\Delta t} \sum_{d=1}^{\rm D_b}\theta_d^b(\xi)\, (\widetilde{\vecp}{}^{n+1}-\widetilde{\vecp}{}^{n})^T\boldsymbol{Z}_{\rm pr}^T \matM  (\matA^{*})^{-1}\vecb_{d}
 +\frac{2}{\Delta t} \sum_{d=1}^{\rm D_A}\theta_d^A(\xi)\, (\widetilde{\vecp}{}^{n+1}-\widetilde{\vecp}{}^{n})^T\boldsymbol{Z}_{\rm pr}^T \matM  (\matA^{*})^{-1}\matA_{d}\boldsymbol{Z}_{\rm pr}\widetilde{\vecp}^{n+1}
 \nonumber \\&
 +\frac{1}{\Delta t^2}(\widetilde{\vecp}{}^{n+1}-\widetilde{\vecp}{}^{n})^T\boldsymbol{Z}_{\rm pr}^T\matM  (\matA^{*})^{-1}\matM\boldsymbol{Z}_{\rm pr}(\widetilde{\vecp}{}^{n+1}-\widetilde{\vecp}{}^{n}).
\end{align}

Its evaluation is very sensitive to round-off errors as stressed in \cite{casenave2012accurate}. Hence, a naive implementation of \eqref{ResidualTrad} may suffer from accuracy issues. These can be circumvented by following the method of \cite{Buhr2014ANS}. We first introduce the Riesz's representative of the residual $\calR$ such that
$$ \left\langle \calR(\vecr^{n+1}) ,\boldsymbol{v}\right\rangle_{\matA^*}=\vecr^{n+1}(\boldsymbol{v}),\quad \forall \boldsymbol{v} \in \R^{\calN}. $$
Using the affine decomposition \eqref{affineDecomp}, the Riesz representation of the primal residual is given by
\begin{equation}\label{resAss1}
    \calR(\vecr^{n+1})=\frac{1}{\Delta t}\; (\matA^{*})^{-1} \matM \boldsymbol{Z}_{\rm pr} (\widetilde{\vecp}{}^{n+1}-\widetilde{\vecp}{}^{n} )+ \sum_{d=1}^{\rm{D}_A}\theta_d^A(\xi)(\matA^{*})^{-1}\matA_d\boldsymbol{Z}_{\rm{pr}}\widetilde{\vecp}^{n+1}-\sum_{d=1}^{\rm{D}_b}\theta_d^b(\xi)(\matA^{*})^{-1}\vecb_d.
\end{equation}
Setting ${\rm{D}_r}=\rm{D}_b+\rm{D}_a \mathsf{N}_{\rm pr}+\mathsf{N}_{\rm pr}$, we define the coefficient vector $\widehat{\vecr}^{n+1} \in \mathbb{R}^{\rm{D}_r}$ as 
\begin{equation}
     \widehat{\vecr}^{n+1}=  \bigg(
\frac{1}{\Delta t} (\widetilde{\vecp}^{n+1}-\widetilde{\vecp}^{n})^T, \, \theta_{1}^A(\xi) (\widetilde{\vecp}^{n+1})^T, \, \ldots, \, \theta_{\rm{D}_{A}}^A(\xi) (\widetilde{\vecp}^{n+1})^T, \, -\theta_1^b(\xi), \, \ldots, \, -\theta_{\rm{D}_b}^b(\xi) \bigg)^{\!T},
\end{equation}
and the matrix $\widehat{\boldsymbol{\eta}}\in \mathbb{R}^{\mathcal{N}\times\rm{D}_r}$ as 
\begin{equation}
 \widehat{\boldsymbol{\eta}}=
\big(
(\matA^{*})^{-1}\matM\boldsymbol{Z}_{\rm{pr}}, \,(\matA^{*})^{-1}\matA_1\boldsymbol{Z}_{\rm{pr}}, \, \ldots, \, (\matA^{*})^{-1}\matA_{\rm{D}_a }\boldsymbol{Z}_{\rm{pr}}, \, (\matA^{*})^{-1}\vecb_1, \, \ldots, \, (\matA^{*})^{-1}\vecb_{\rm{D}_{b}}
\big)^T.    
\end{equation}
The Riesz representative is then written as 
\begin{equation}\label{resAss2}
  \mathcal{R}(\vecr^{n+1})=   \sum_{d=1}^{\rm{D_r}}\widehat{\vecr}{}^{n+1}_d\widehat{\boldsymbol{\eta}}_d,
\end{equation}
and the norm is given by
\begin{equation}\label{RieszResidual}
    \|\mathcal{R}(\vecr^{n+1})\|^2_{\matA^*}= \Big \langle\sum_{d=1}^{\rm{D_r}} \widehat{\vecr}{}^{n+1}_d \widehat{\boldsymbol{\eta}}_d, \sum_{d=1}^{\rm{D_r}}\widehat{r}{}^{n+1}_d \widehat{\boldsymbol{\eta}}_d \Big\rangle_{\!\matA^*}.
\end{equation}
The evaluation of \eqref{RieszResidual} is divided into three steps:
\begin{enumerate}
    \item We construct an orthonormal basis $\boldsymbol{\zeta}$ of $ \widehat{\boldsymbol{\eta}}$ by applying a modified Gram-Schmidt algorithm with reorthogonalization (see Algorithm \ref{algo:GramSchmidt}).
    \item For each $d \in {\rm{D}_r}$, we evaluate $\sum \limits_{i=1}^{\rm{D_r}}\overline{\boldsymbol{\eta}}_{d,i}\boldsymbol{\zeta}_i$, with $\overline{\boldsymbol{\eta}}_{d,i}=\left\langle \widehat{\boldsymbol{\eta}}_d,\boldsymbol{\zeta}_i\right\rangle_{\matA^*}$.
    \item We compute \eqref{RieszResidual} using 
    \begin{equation}\label{ResidualConst2}
\begin{split}
     \|\mathcal{R}(\vecr^{n+1})\|_{\matA^*}^2 &=\Big \langle \sum_{d=1}^{\rm{D_r}}  \widehat{{\vecr}}{}_d^{n+1} \Big (\sum_{i=1}^{\rm{D_r}} \overline{\boldsymbol{\eta}}_{d,i} \boldsymbol{\zeta}_i \Big),\sum_{d=1}^{\rm{D_r}}  \widehat{{\vecr}}{}_d^{n+1} \Big( \sum_{i=1}^{\rm{D_r}} \overline{\boldsymbol{\eta}}_{d,i} \boldsymbol{\zeta}_i \Big ) \Big \rangle_{\!\matA^*}\\
     &= \sum_{i=1}^{\rm{D_r}} \sum_{j=1}^{\rm{D_r}}
     \Big(\sum_{d=1}^{\rm{D_r}}  \widehat{{\vecr}}{}_d^{n+1} \overline{\boldsymbol{\eta}}_{d,i} \Big) 
     \Big(\sum_{d=1}^{\rm{D_r}}  \widehat{{\vecr}}{}_d^{n+1} \overline{\boldsymbol{\eta}}_{d,j} \Big) 
     \big\langle \boldsymbol{\zeta}_i, \boldsymbol{\zeta}_j \big\rangle_{\!\matA^*} 
     \\
     &= \sum_{i=1}^{\rm{D_r}} \Big(\sum_{d=1}^{\rm{D_r}}  \widehat{{\vecr}}{}_d^{n+1} \; \overline{\boldsymbol{\eta}}_{d,i} \Big)^{\!2}.
\end{split}         
\end{equation}
\end{enumerate}
Steps 1 and 2 are performed in the offline stage, while steps 3 is completed during the online stage. We follow the same strategy to define the residual dual norm of the dual problem. 

\paragraph{Coercivity constant computation.}
The coercivity constant defined by \eqref{alphaE} is the minimum of the generalized Rayleigh quotient and we have that $\alpha$ is a lower bound of the smallest eigenvalue $\lambda_{min}$ of the following generalized eigenvalue problem
\begin{equation}\label{proofAlpha}
(M+\Delta t\matA_{\rm sym})\boldsymbol{v} = \lambda \, \orange{\boldsymbol{A}^*} \boldsymbol{v}.
\end{equation}

To avoid the resolution of the generalized eigenvalue problem \eqref{proofAlpha} which requires for instance $\mathcal{O}(\calN^{3})$ using a QR algorithm, 
we consider the successive constraint method (SCM) \cite{CHEN20081295,HUYNH2007473} (see also Appendix \ref{SCM}) which, using \eqref{affineDecomp}, provides an upper bound $\alpha_{\rm UB}(\xi)\in \R$ and a lower bound $\alpha_{\rm LB}(\xi)\in \R$ for the coercivity constant such that
$$\alpha(\xi):=\alpha_{\rm LB}(\xi) \leq \lambda_{min}\leq \alpha_{\rm UB}(\xi). $$
The evaluation of these bounds do not depend on $\calN$.
The same approach is used to compute the other coercivity constants \eqref{alphaE} and \eqref{alphaE}.

\section{Numerical results}\label{NumRes}
In this section, we numerically validate the theoretical results obtained for the reduction of problem \eqref{SF}. Our main goals are to study both efficiency and computational cost of the proposed estimators. 

The following parameters are considered:
\begin{alignat*}{4}
\mu & =1.5\times 10^{-5}\;\text{Pa}. \text{s}, \quad & c_t & =1.4\times 10^{-7} \;\text{Pa}^{-1}, \quad & g &=9.81\;\text{m}/\text{s}^2, \quad & \rho &=700\; \text{kg}/\text{m}^3,  \\
 p_{bh} &=4.13\times 10^7\; \text{Pa}, \quad & \phi &=0.2, \quad & r_w &=0.1, \quad & z_{bh} &=0\; \text{m}.
\end{alignat*}
We use $p_{\rm D}=10^{5}$ Pa as Dirichlet boundary condition. The total duration of the simulation is $T=200$ days and the time step is $\Delta t=10$ days. The initial pressure is defined by 
$$ p_K^0=p_{\rm D}-\rho g z_K,$$.

We consider a three-dimensional domain
\[
\Omega=[ -2.686 \cdot 10^{-3} \;\rm{m},1996 \;\rm{m}]\times [6.1 \cdot 10^{-5}\;\rm{m},1996\;\rm{m}]\times [-1000.13\;\rm{m}, 2.686 \cdot 10^{-3}\;\rm{m}]
\]

(see Figure \ref{fig:3DDomain}), where an anticline is located in the middle. In depth, a high permeability zone whose values $\kappa_1$ belong to $[10^{-14},10^{-12}]$ is surrounded by two impermeable over- and under- burdens where the permeability $\kappa_2$ is in the range  $[10^{-17},10^{-15}]$. To represent this geometry, a \emph{Corner Point Grid} (CPG)
with hexahedra and non-planar faces, is used. The number of cells $\calN$ is equal to $15210$. A well is located in the center of $\Omega$ and perforated along $18$ cells whose centers lie within the bounding box $[900,1050]\times[900,1050]\times[-700,-500]$. The boundary $\Gamma_{\rm int}$ is given by
\begin{alignat*}{3}
\Gamma_{\rm int} &=  \{818.87\}  \times[818.87, 1125.95]  \times [-620.324,-490.622]\\
        & \, \cup  \{1125.95\}  \times[818.87, 1125.95]  \times [-620.324,-490.622] \\
        & \, \cup  [818.87, 1125.95] \times \{818.87\}  \times [-620.324,-490.622]\\
        &  \,\cup  [818.87, 1125.95] \times  \{1125.95\}\times [-620.324,-490.622]\\
        &  \,\cup  [818.87,1125.95]  \times [818.87, 1125.95]\times \{-620.324\}\\
        &  \, \cup  [818.87,1125.95] \times [818.87, 1125.95] \times\{-490.622\}.
\end{alignat*}

 We generate training sets $\Xi_{\rm training}=\{\xi_1, \ldots,\xi_{\calL}\}$ by sampling $\xi=\{\kappa_1,\kappa_2\}$ uniformly at random from the ranges specified above. For the EIM and the Greedy algorithm we use $\calL=100$ and $\calL=300$ for the SCM. The sample distribution obtained by the Greedy algorithm is plotted in Figure \ref{fig:permeability}.
An independent test set $\Xi_{\rm test}$ is created for the online validation stage by sampling previously unseen $\xi$ from the same parameter ranges.

\begin{figure}[H]
    \centering
    \includegraphics[scale=0.2]{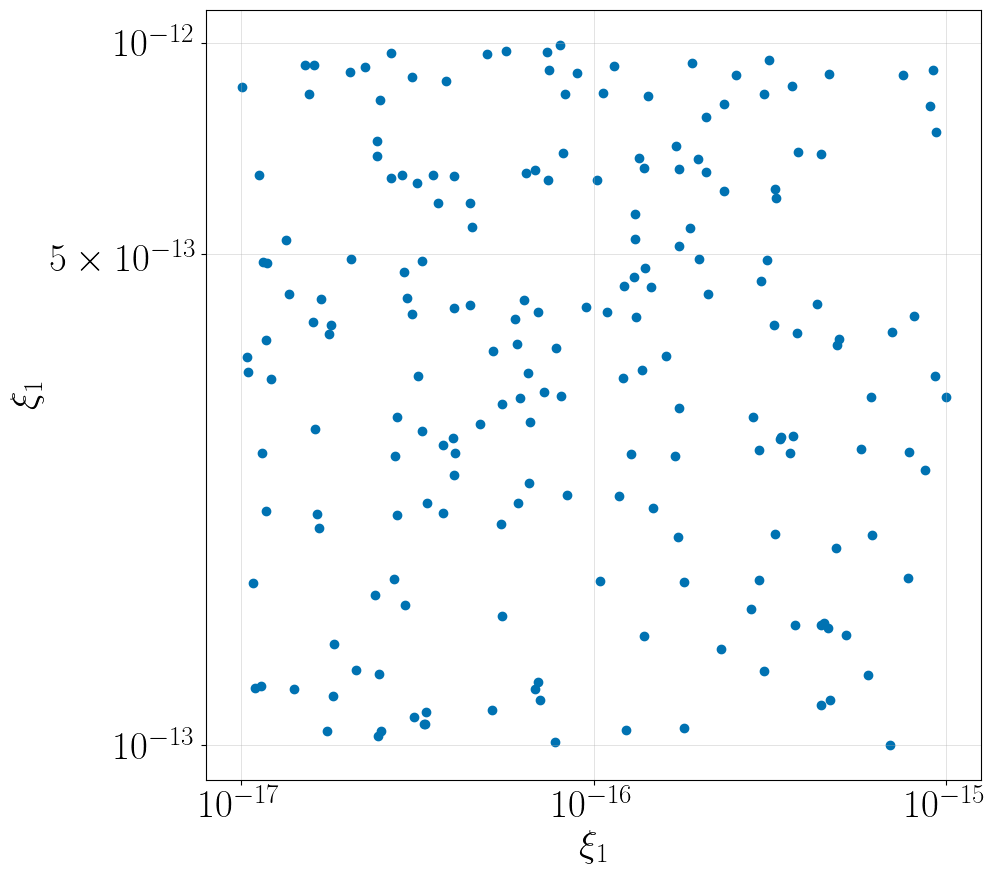}
    \caption{Permeabilities distribution used for the offline stage.}
    \label{fig:permeability}
\end{figure}
\begin{center}
    \begin{figure}[H]
   \begin{subfigure}[b]{0.49\linewidth}
    \centering
     \captionsetup{justification=raggedright,margin={0.8cm,0cm}}
    \includegraphics[width=0.71\textwidth]{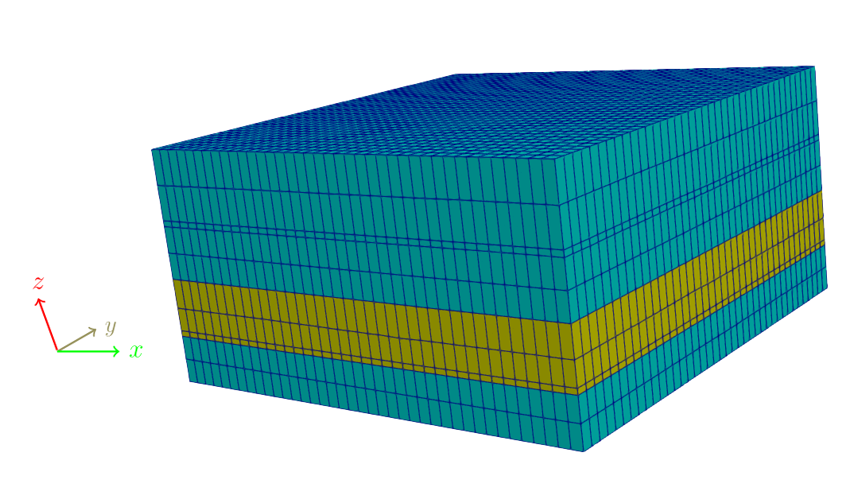} 
  \caption{3D domain} 
   \end{subfigure}
   \hspace{-2.5cm}
     \begin{subfigure}[b]{0.58\linewidth}
    \centering
     \captionsetup{justification=raggedleft,margin={0cm,0.8cm}}
    \includegraphics[width=0.58\textwidth]{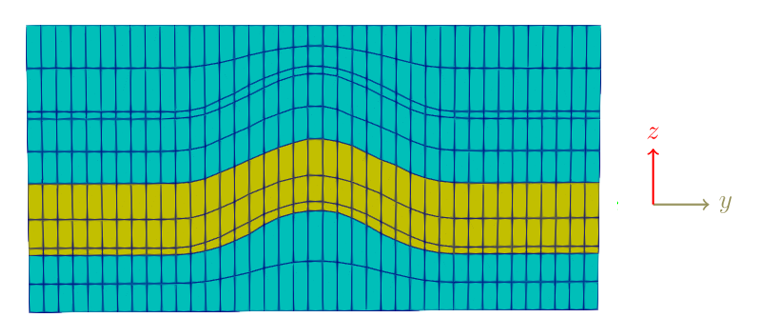} 
  \caption{2D slice} 
   \end{subfigure}
   \caption{Spatial repartition of the permeabilities within $\Omega$: the yellow zone includes cells having the high permeability value $\kappa_1$ and the low permeability value $\kappa_2$ is used in the blue zone.}
       \label{fig:3DDomain}
\end{figure}
\end{center}

\subsection{Affine decomposition of the scheme coefficients}
To construct the reduced model, we start by applying the EIM as discussed in Section \ref{CompAspect}. With respect to the number of interpolation points, we plot in Figure \ref{fig:EIMErrors} the evolution of the interpolation errors defined as

\begin{equation}
\label{maxInterpError}
     e_{\rm D_t,\max}^{\infty}=\max_{\xi \in \Xi_{\rm training,EIM}} \frac{\left \| \vectmat{t}(\xi)- \mathcal{I}_{\rm D_t}[\vectmat{t}(\xi)]\right\|_{L^{\infty}} }{\|\vectmat{t}(\xi)\|_{L^{\infty}}}     
\end{equation}
where $\vectmat{t}$ stands for one of the linearized terms $\matA$, $\vecb$, $\boldsymbol{l}$
and $D_t$ for the corresponding number of interpolation points. 
All the linear decompositions reached a relative error of $10^{-14}$ with 6 interpolations points.

\begin{figure}[H]
     \centering
     \begin{subfigure}[b]{0.32\textwidth}
         \centering
         \includegraphics[width=\textwidth]{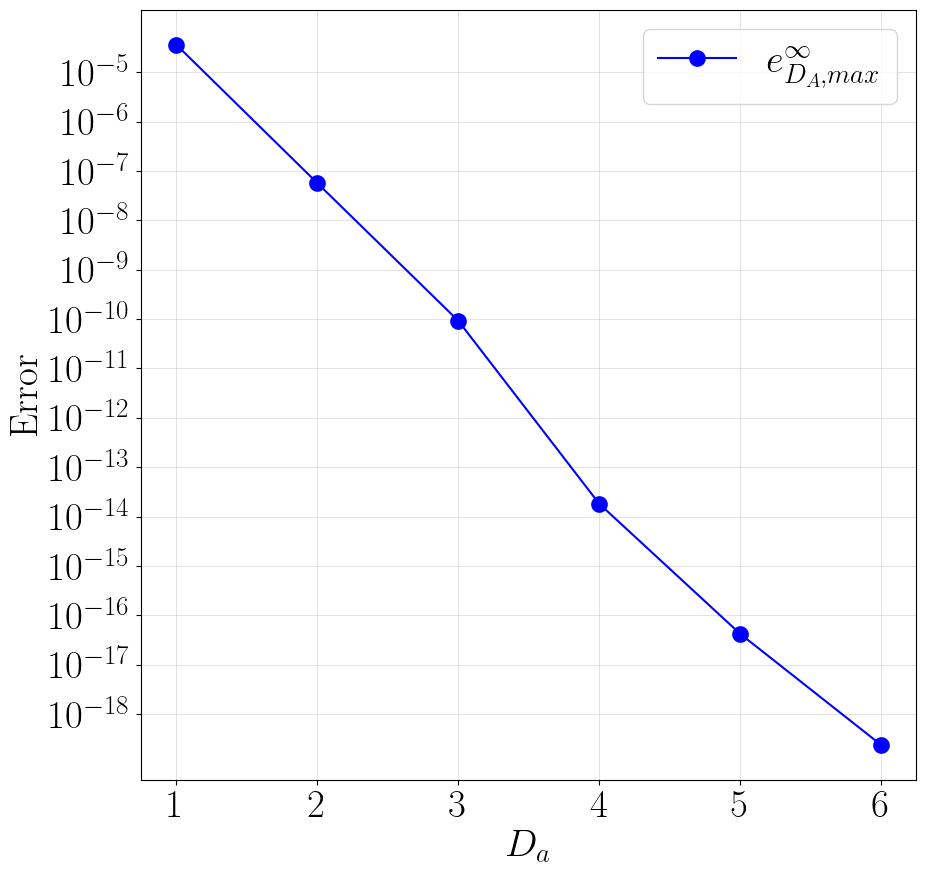}
         \caption{Evolution of the matrix $\matA$ interpolation error \eqref{maxInterpError} with respect to the number of selected parameters $\rm D_A$} 
         \label{fig:ErrorPOfflineMatrix}
     \end{subfigure}
     \hfill
     \begin{subfigure}[b]{0.315\textwidth}
         \centering
         \includegraphics[width=\textwidth]{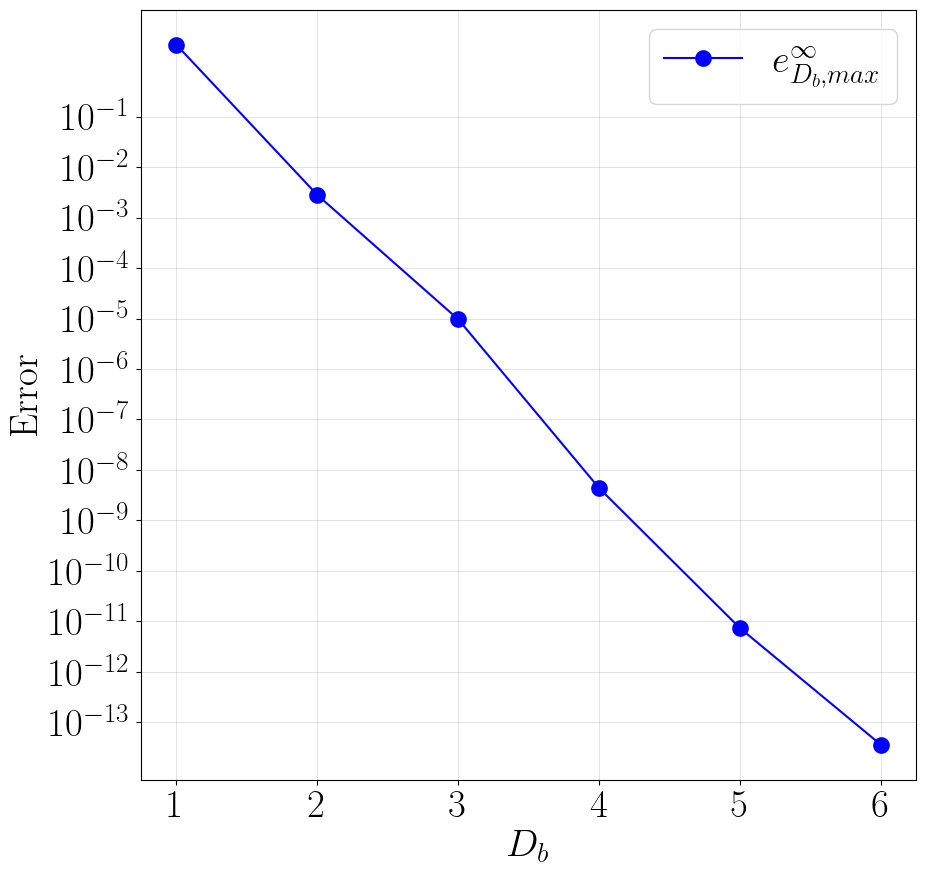}
         \caption{Evolution of the vector $\vecb$ interpolation error \eqref{maxInterpError} with respect to the number of selected parameters $\rm D_b$} 
         \label{fig:ErrorPOfflineVectorB}
     \end{subfigure}
     \begin{subfigure}[b]{0.32\textwidth}
         \centering
         \includegraphics[width=\textwidth]{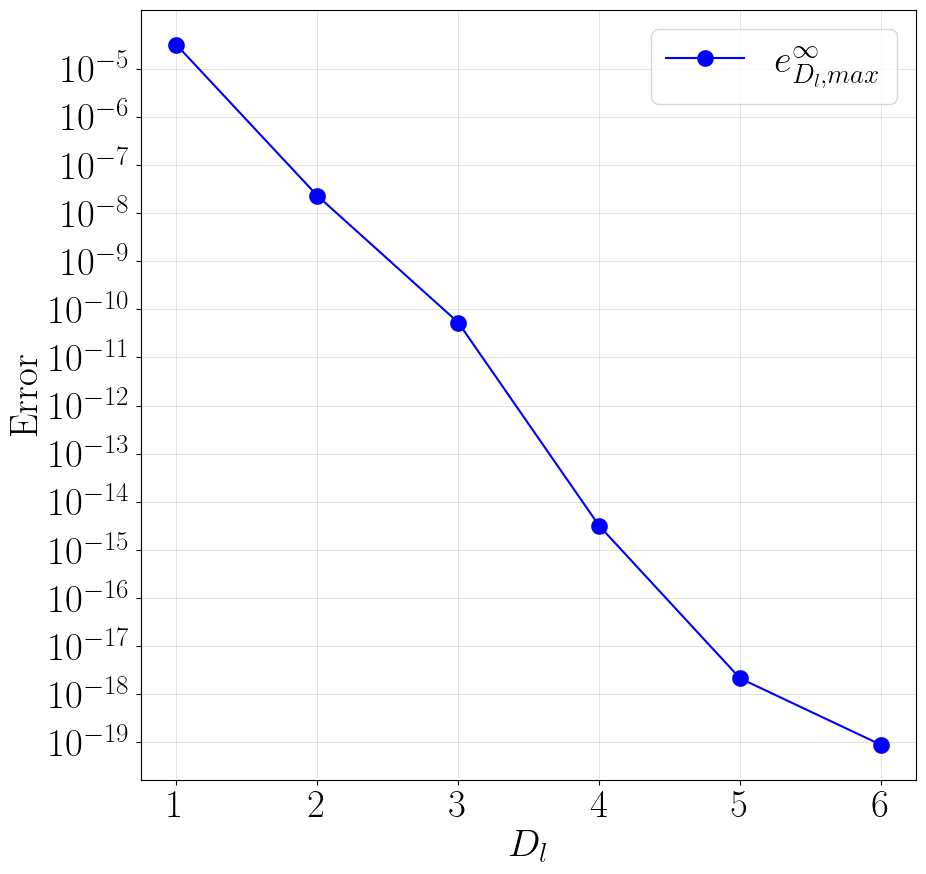}
         \caption{Evolution of the vector $\boldsymbol{l}$ interpolation error \eqref{maxInterpError} with respect to the number of selected parameters $\rm D_{\boldsymbol{l}}$} 
         \label{fig:ErrorPOfflineVEctorL}
     \end{subfigure}
     \caption{Evolution of the interpolation errors.}
     \label{fig:EIMErrors}
\end{figure}

\subsection{Bounds on the coercivity constants} 
We consider the successive constraint method (SCM) to compute a lower bound of the coercivity constant $\alpha_{\matA}$ defined in \eqref{alphaE}. As mentioned above, we use a different training set $\Xi_{\rm training,SCM}$. Referring to Appendix \ref{SCM} for notations and details, we set $\rm M_1=M_2=20$ and $\rm tol=1e-04$. Since $D_A=6$, we must solve 12 eigenvalue problems to define $\boldsymbol{\mathcal{B}}$. The offline greedy Algorithm \ref{alg:SCMAlgo} generates a parameter set $\Xi_{\rm M}$ of dimension $\rm M=25$.
To verify the quality of the SCM result in the online stage, we compute lower and upper bounds for $\alpha_{\matA}$ on $\Xi_{\rm training,SCM}$ and $\Xi_{\rm test}$ and the ratio
$$ r(\xi)=\frac{\alpha_{\rm UB}(\xi)-\alpha_{\rm LB}(\xi)}{\alpha_{\rm UB}(\xi)}.$$
For both samples, $r$ remains in the range $ [0,1.8\times 10^{-5}]$.

\subsection{POD-Greedy algorithm for all-purpose reduced models}
As a first test, we build  a reduced model obtained thanks to a POD-Greedy algorithm driven by either the a posteriori estimate $\Delta_{\rm{pr}}^N$ or $\overline{\Delta}_{\rm{pr}}^N$. We consider the following errors
\begin{equation*}
   \begin{array}{cc}
          \vece_{\rm{pr, \max}}^N = \max_{\xi \in \Xi} \vertiii{\vece^N}_{\rm pr},  \qquad \Delta_{\rm{pr,\max}}^N = \max_{\xi \in \Xi} \;\Delta_{\rm{pr}}^N, \qquad \eta_{\rm{pr,max}}^N =\max_{\xi \in \Xi} \;  \frac{\Delta_{\rm{pr}}^N}{\vertiii{\vece^N}_{\rm pr}}, \\
         \overline{\vece}_{\rm{pr, \max}}^N = \max_{\xi \in \Xi} \vertiii{\overline{\vece}^N}_{\rm pr}, \qquad \overline{\Delta}_{\rm{pr,\max}}^N = \max_{\xi \in \Xi} \;\overline{\Delta}_{\rm{pr}}^N, \qquad \overline{\eta}_{\rm{pr,max}}^N =\max_{\xi \in \Xi} \;  \frac{\overline{\Delta}_{\rm{pr}}^N}{\vertiii{\overline{\vece}^N}_{\rm pr}},
   \end{array}
\end{equation*}
 where $\Delta_{\rm{pr}}^N$ and $\overline{\Delta}_{\rm{pr}}^N$ are given by \eqref{PrimalEstimator} and \eqref{PrimalEstimatorbis} respectively. Here, $\vertiii{\vece^N}_{\rm pr}$ denotes the truth error obtained when the POD-Greedy algorithm is guided by the standard estimator $\Delta_{\rm{pr}}^N$, whereas $\vertiii{\overline{\vece}^N}_{\rm pr}$ represents the truth error when the basis generation is driven by the estimator $\overline{\Delta}_{\rm{pr}}^N$. The POD-Greedy ran for \orange{32} steps with $ric$ set to $0.98$.

We compare in Figure \ref{fig:ErrorsPrimalComp}, the evolution of the maximum error and the a posteriori estimates with respect to the basis dimension $\sfN_{\rm pr}$ using a POD-Greedy algorithm driven by the two estimators \eqref{PrimalEstimator} and \eqref{PrimalEstimatorbis}.
Both estimators lead to similar convergence of the (true) error in the offline and online phases. \orange{
On the other hand, we observe that estimator \eqref{PrimalEstimatorbis} always yields tighter estimates than \eqref{PrimalEstimator}.}
\begin{figure}[H]
    \centering
   \begin{subfigure}[b]{0.35\linewidth}
    \centering
            \captionsetup{justification=raggedright,margin={1.2cm,0cm}}
    \includegraphics[width=\textwidth]{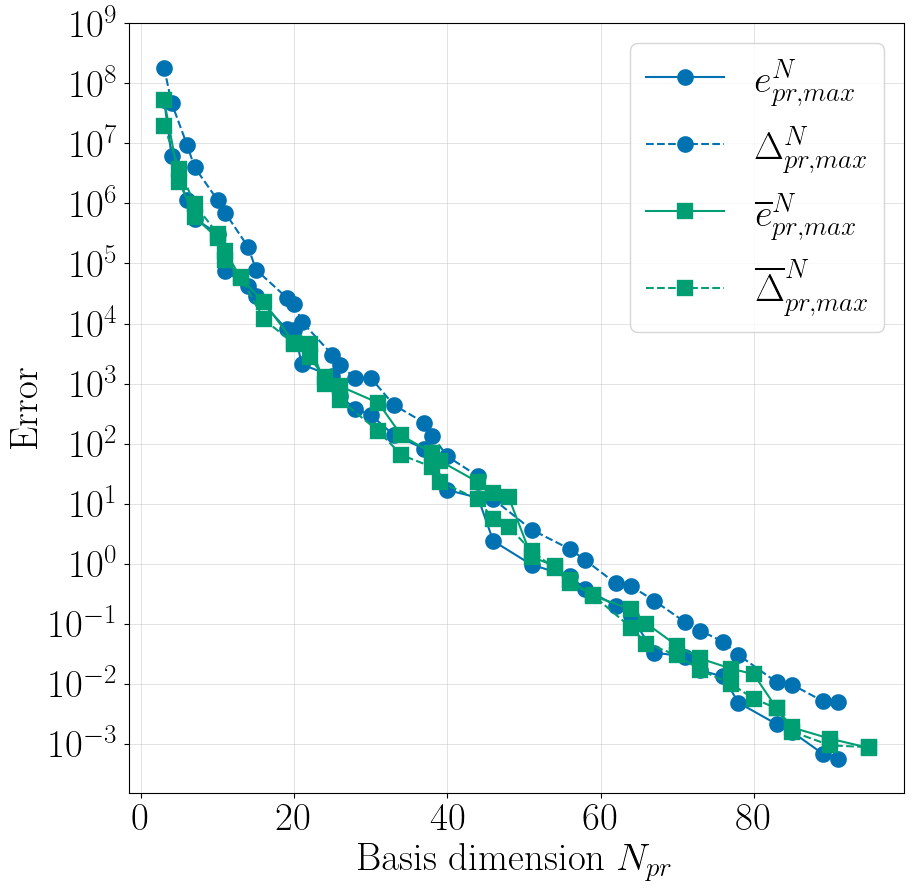} 
 \caption{Offline phase} 
  \label{fig:ErrorsPrimalCompOffline}
   \end{subfigure}
     \begin{subfigure}[b]{0.35\linewidth}
    \centering
            \captionsetup{justification=raggedright,margin={1.2cm,0cm}}
    \includegraphics[width=\textwidth]{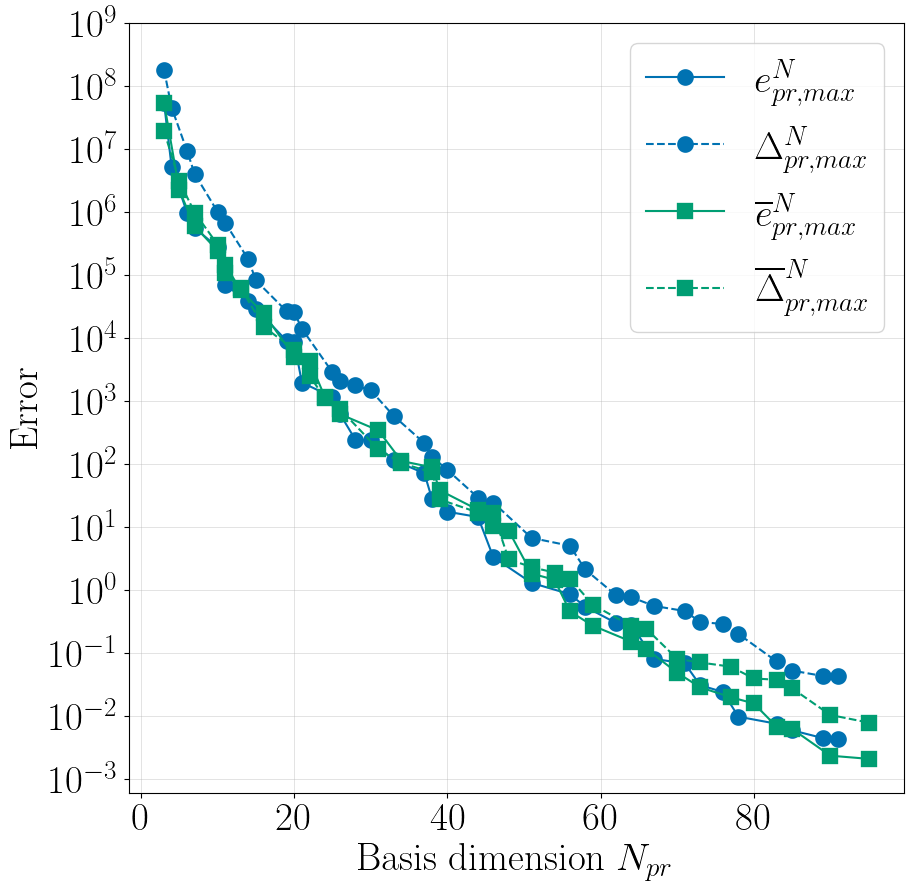} 
  \caption{Online phase} 
  \label{fig:ErrorsPrimalCompOnline}
   \end{subfigure}
    \caption{Maximum truth error and a posteriori estimator as a function of the primal basis dimension using $\Xi_{\rm training} $ on the left and $\Xi_{\rm test}$ on the right for a POD-Greedy algorithm driven by the a posteriori error estimates \eqref{PrimalEstimator} and \eqref{PrimalEstimatorbis}.}
    \label{fig:ErrorsPrimalComp}
\end{figure}

Table \ref{tab:effect_3_primal} shows the efficiency (or effectivity) of the two estimators \eqref{PrimalEstimator} and \eqref{PrimalEstimatorbis}. \orange{ Estimator \eqref{PrimalEstimator} has an 
effectivity index between $20$ and $100$. Estimator \eqref{PrimalEstimatorbis} shows a smaller 
effectivity index ranging from $[4;20]$ for a parameter independent norm.}  
\begin{table}[H]
    \centering

    \begin{subtable}[b]{0.4\textwidth}
        \centering
        \begin{tabular}{c||c|c}
            \hline \hline
$N_{pr}$ & $e^N_{pr,max}$ & $\eta^N_{pr,max}$ \\ \hline
3  &  $5.37\times 10^{7}$  &  20.9  \\ \hline
4  &  $6.12\times 10^{6}$  &  36.7  \\ \hline
6  &  $1.13\times 10^{6}$  &  21.3  \\ \hline
7  &  $5.6\times 10^{5}$  &  76.9  \\ \hline
10  &  $3.12\times 10^{5}$  &  60.4  \\ \hline
11  &  $7.46\times 10^{4}$  &  83.7  \\ \hline
14  &  $4.21\times 10^{4}$  &  33.9  \\ \hline
15  &  $2.88\times 10^{4}$  &  39.9  \\ \hline
19  &  $8.05\times 10^{3}$  &  45.2  \\ \hline
20  &  $7.71\times 10^{3}$  &  53.5  \\ \hline
21  &  $2.08\times 10^{3}$  &  41.3  \\ \hline
25  &  $1.33\times 10^{3}$  &  24.0  \\ \hline
26  &  $6.09\times 10^{2}$  &  25.0  \\ \hline
28  &  $3.77\times 10^{2}$  &  29.1  \\ \hline
30  &  $3.0\times 10^{2}$  &  34.9  \\ \hline
\hline\hline
        \end{tabular}
        \caption{Effectivities of $\Delta_{pr}^N$.}
        \label{tab:EffectPODGreedyPrimal}
    \end{subtable}
    \begin{subtable}[b]{0.4\textwidth}
        \centering
        \begin{tabular}{c||c|c}
            \hline \hline
$N_{pr}$ & $\overline{e}^N_{pr,max}$ & $\overline{\eta}^N_{pr,max}$ \\ \hline
3  &  $5.37\times 10^{7}$  &  4.9  \\ \hline
5  &  $2.28\times 10^{6}$  &  6.4  \\ \hline
7  &  $5.94\times 10^{5}$  &  16.1  \\ \hline
10  &  $2.63\times 10^{5}$  &  7.8  \\ \hline
11  &  $1.62\times 10^{5}$  &  6.4  \\ \hline
13  &  $5.74\times 10^{4}$  &  9.1  \\ \hline
16  &  $2.3\times 10^{4}$  &  5.3  \\ \hline
20  &  $4.75\times 10^{3}$  &  4.8  \\ \hline
22  &  $4.6\times 10^{3}$  &  6.9  \\ \hline
24  &  $9.93\times 10^{2}$  &  6.2  \\ \hline
26  &  $9.01\times 10^{2}$  &  6.9  \\ \hline
31  &  $4.75\times 10^{2}$  &  6.6  \\ \hline
34  &  $1.36\times 10^{2}$  &  6.8  \\ \hline
38  &  $7.0\times 10^{1}$  &  6.1  \\ \hline
39  &  $5.18\times 10^{1}$  &  7.6  \\ \hline
\hline\hline
        \end{tabular}
        \caption{Effectivities of $\overline{\Delta}_{\rm{pr}}^N$.}
       \label{tab:EffectPODGreedyPrimalOscar}
    \end{subtable}
    
    \caption{
    Effectivities of the primal a posteriori error estimate $\Delta_{\rm{pr}}^N$ on the left and $\overline{\Delta}_{\rm{pr}}^N$ on the right with respect to the basis dimension $\mathsf{N}_{\rm pr}$ in the offline stage}
    \label{tab:effect_3_primal}
\end{table}

\subsection{Goal-oriented POD-Greedy algorithm}
Next, we build a goal-oriented reduced model for the output $s^N$ at the final time by first considering the definition \eqref{reducedOutput} of the reduced output. We run the POD-Greedy algorithm for 10 iterations.
We compare, in that case, the evolution of $\vece_{\rm{s, \max}}^N$, $\Delta_{\rm{s, \max}}^N$ , $\overline{\Delta}_{\rm{s, \max}}^N$, $\hat{\Delta}_{\rm{s,\max}}^N$, $\eta_{\rm{s,max}}^N$, $\overline{\eta}_{\rm{s,max}}^N$ and $\hat{\eta}_{\rm{s,max}}^N$ defined as 
\begin{equation*}
    \begin{aligned}
        \mathbf{e}_{\rm{s, \max}}^N &= \max_{\xi \in \Xi} |s^N - s^{N_{\rm s},N}|, 
        & \Delta_{\rm{s,\max}}^N &= \max_{\xi \in \Xi} \;\Delta_{\rm{s}}^N, 
        & \eta_{\rm{s,\max}}^N &= \max_{\xi \in \Xi} \; \frac{\Delta_{\rm{s}}^N}{\mathbf{e}_{\rm s}^N}, \\
        \overline{\mathbf{e}}_{\rm{s, \max}}^N &= \max_{\xi \in \Xi} |\overline{s}^N - \overline{s}^{N_{\rm s},N}|, 
        & \overline{\Delta}_{\rm{s,\max}}^N &= \max_{\xi \in \Xi} \;\overline{\Delta}_{\rm{s}}^N, 
        & \overline{\eta}_{\rm{s,\max}}^N &= \max_{\xi \in \Xi} \; \frac{\overline{\Delta}_{\rm{s}}^N}{\overline{\mathbf{e}}_{\rm s}^N}, \\
        \hat{\mathbf{e}}_{\rm{s, \max}}^N &= \max_{\xi \in \Xi} |\hat{s}^N - \hat{s}^{N_{\rm s},N}|, 
        & \hat{\Delta}_{\rm{s,\max}}^N &= \max_{\xi \in \Xi} \;\hat{\Delta}_{\rm{s}}^N, 
        & \hat{\eta}_{\rm{s,\max}}^N &= \max_{\xi \in \Xi} \; \frac{\hat{\Delta}_{\rm{s}}^N}{\hat{\mathbf{e}}_{\rm s}^N},
    \end{aligned}
\end{equation*}
 with respect to the primal basis dimension $\sfN_{\rm pr}$ using $\Xi_{\rm training}$ and $\Xi_{\rm test}$. 
Here, $s^N $ and $ s^{N_{\rm s},N}$ denote, respectively, the true and reduced quantities of interest obtained when the POD–Greedy algorithm is guided by the estimator $\Delta_{\rm{s}}^N$; $\overline{s}^N$ and $\overline{s}$ denote the corresponding quantities when the algorithm is driven by $\overline{\Delta}_{\rm{s}}^N$; and $\hat{s}^N $ and $\hat{s}^{N_{\rm s},N}$ denote the quantities obtained when basis generation is driven by $\hat{\Delta}_{\rm{s}}^N$.  \\

\begin{figure}[H]
     \centering

     \begin{subfigure}[b]{0.35\textwidth}
         \centering
         \includegraphics[width=\textwidth]{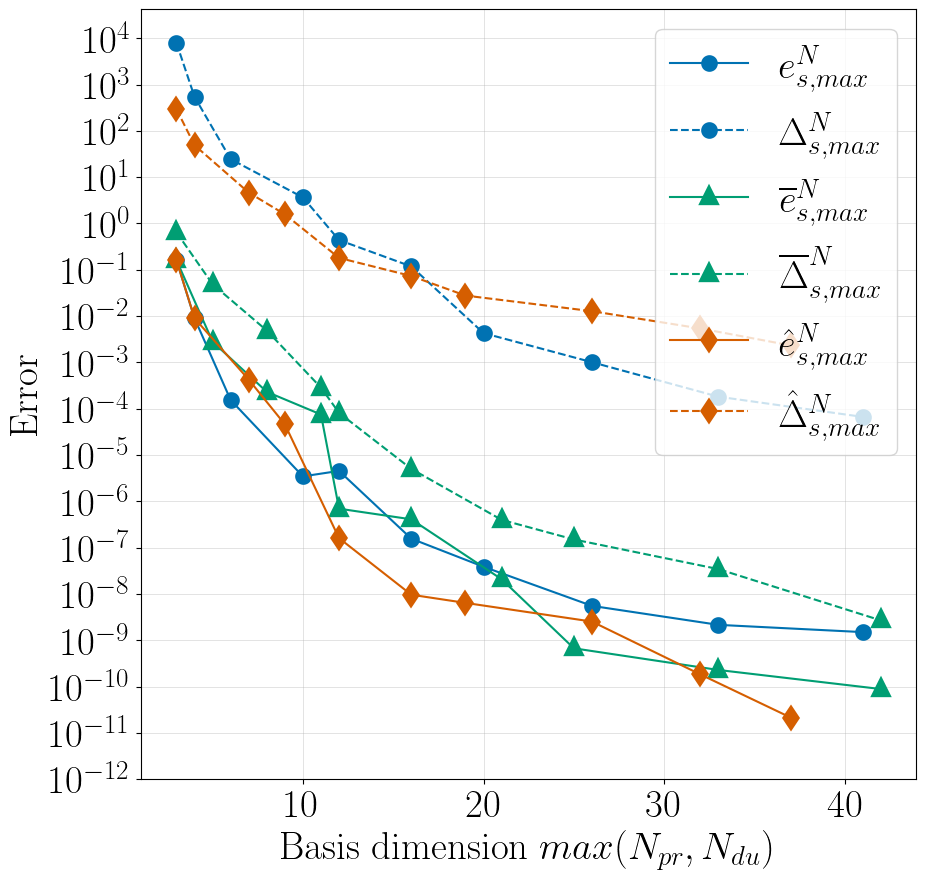}
         \caption{Offline phase} 
    \label{fig:ErrorsQOIOffline}
     \end{subfigure}
     \begin{subfigure}[b]{0.35\textwidth}
         \centering
         \includegraphics[width=\textwidth]{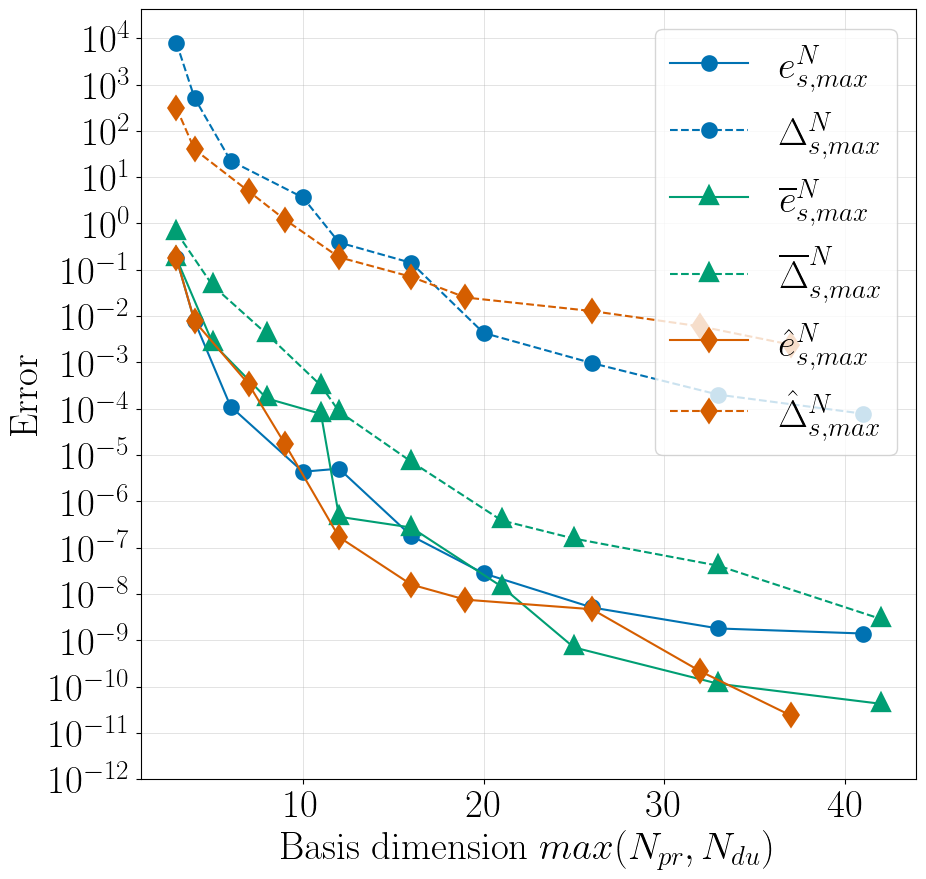}
         \caption{Online phase} 
\label{fig:ErrorsQOIOnline}
     \end{subfigure}
      \caption{
      Maximum truth error and a posteriori estimator of the output \eqref{reducedOutput} as a function of the primal and dual basis dimension using $\Xi_{\rm training}$ on the left and $\Xi_{\rm test}$ on the right for a POD-Greedy algorithm driven by the a posteriori error estimates \eqref{QOIEst}, \eqref{QOIEstbis} and \eqref{QOIEstter}.}
    \label{fig:ErrorsQOI}
\end{figure}
Figure \ref{fig:ErrorsQOI} shows that the error related to each estimator decreases by more than \orange{11} orders of magnitude as the basis dimension grows. \orange{None of the error curves has a superior convergence rate. Estimator \eqref{QOIEstbis} is more efficient than estimators \eqref{QOIEst} and \eqref{QOIEstter}.}  Table \ref{tab:QOIEsteffi} reports the estimators efficiencies. These efficiencies are substantially larger than those observed for the primal problem. \orange{Estimator \eqref{QOIEstter} is the least efficient of the three. It can be explained by the slower convergence of the estimator.}

\orange{The smaller 
effectivity index of estimator \eqref{PrimalEstimatorbis} and \eqref{QOIEstbis} can be explained by looking at the ratio between $\frac{\alpha}{\Delta t}$ and $\alpha_{\matA}$. We numerically observe that}
\begin{equation*}
    \orange{\alpha_{\matA}\leq \frac{\alpha}{\Delta t} \leq \left (1+2\times10^{-5}\right )\alpha_{\matA}}
\end{equation*}
\orange{Both coercivities are almost identical. Taking estimator \eqref{QOIEst} and \eqref{QOIEstbis} and replacing $\frac{\alpha}{\Delta t}$ by $\alpha_{\matA}$ it becomes obvious that \eqref{QOIEstbis} is more effective than \eqref{QOIEst}.}

We now employ the second definition of the output error \eqref{reducedOutput2} and analyze the evolution of the POD-Greedy algorithm in the offline and online stages
with both versions of the error estimates. \orange{We let the Greedy run for 15 iterations and} define:
\begin{equation*}
    \begin{aligned}
        \mathbf{e}_{\Tilde{s},max}^N &= \max_{\xi \in \Xi} |s^N - \Tilde{s}^{N_{\rm s},N}|, 
        & \Delta_{\Tilde{s},max}^N &= \max_{\xi \in \Xi} \;\Delta_{\Tilde{s}}^N, 
        & \eta_{\Tilde{s},max}^N &= \max_{\xi \in \Xi} \; \frac{\Delta_{\Tilde{s}}^N}{\mathbf{e}_{\rm s}^N}, \\
        \overline{\mathbf{e}}_{\Tilde{s},max}^N &= \max_{\xi \in \Xi} |\overline{s}^N - \overline{\Tilde{s}}^{N_{\rm s},N}|, 
        & \overline{\Delta}_{\Tilde{s},max}^N &= \max_{\xi \in \Xi} \;\overline{\Delta}_{\Tilde{s}}^N, 
        & \overline{\eta}_{\Tilde{s},max}^N &= \max_{\xi \in \Xi} \; \frac{\overline{\Delta}_{\Tilde{s}}^N}{\overline{\mathbf{e}}_{\rm s}^N}, \\
        \hat{\mathbf{e}}_{\Tilde{s},max}^N &= \max_{\xi \in \Xi} |\hat{s}^N - \hat{\Tilde{s}}^{N_{\rm s},N}|, 
        & \hat{\Delta}_{\Tilde{s},max}^N &= \max_{\xi \in \Xi} \;\hat{\Delta}_{\Tilde{s}}^N, 
        & \hat{\eta}_{\Tilde{s},max}^N &= \max_{\xi \in \Xi} \; \frac{\hat{\Delta}_{\Tilde{s}}^N}{\hat{\mathbf{e}}_{\rm s}^N}.
    \end{aligned}
\end{equation*}

\begin{table}[t]
    \centering

    \begin{subtable}[b]{0.3\textwidth}
        \centering
        \begin{tabular}{c|c|c}
            \hline \hline
$N_{pr}$ & $N_{du}$ & $\eta^N_{s,max}$ \\ \hline
3  &  2  &  $5.89\times 10^{6}$  \\ \hline
4  &  4  &  $4.58\times 10^{7}$  \\ \hline
6  &  6  &  $7.96\times 10^{8}$  \\ \hline
8  &  10  &  $2.26\times 10^{8}$  \\ \hline
10  &  12  &  $1.22\times 10^{8}$  \\ \hline
12  &  16  &  $1.36\times 10^{9}$  \\ \hline
15  &  20  &  $8.34\times 10^{6}$  \\ \hline
16  &  26  &  $1.19\times 10^{8}$  \\ \hline
20  &  33  &  $1.29\times 10^{8}$  \\ \hline
22  &  41  &  $4.15\times 10^{9}$  \\ \hline
23  &  47  &  $9.11\times 10^{7}$  \\ \hline
27  &  58  &  $7.37\times 10^{7}$  \\ \hline
28  &  67  &  $3.98\times 10^{6}$  \\ \hline
31  &  80  &  $2.71\times 10^{6}$  \\ \hline
34  &  91  &  $7.88\times 10^{7}$  \\ \hline
\hline\hline
        \end{tabular}
        \caption{Effectivities of \eqref{QOIEst}.} 
        \label{tab:QOIEst}
    \end{subtable}
    \hfill 
    \begin{subtable}[b]{0.3\textwidth}
        \centering
        \begin{tabular}{c|c|c}
            \hline \hline
$N_{pr}$ & $N_{du}$ & $\overline{\eta}^N_{s,max}$ \\ \hline
  &  2  &  $1.44\times 10^{3}$  \\ \hline
5  &  4  &  $1.88\times 10^{3}$  \\ \hline
7  &  8  &  $5.97\times 10^{4}$  \\ \hline
10  &  11  &  $1.31\times 10^{4}$  \\ \hline
12  &  12  &  $1.38\times 10^{4}$  \\ \hline
15  &  16  &  $1.18\times 10^{4}$  \\ \hline
18  &  21  &  $1.55\times 10^{5}$  \\ \hline
20  &  25  &  $3.08\times 10^{4}$  \\ \hline
22  &  33  &  $2.12\times 10^{4}$  \\ \hline
25  &  42  &  $1.43\times 10^{4}$  \\ \hline
28  &  52  &  $6.18\times 10^{3}$  \\ \hline
29  &  60  &  $3.35\times 10^{3}$  \\ \hline
31  &  69  &  $5.58\times 10^{2}$  \\ \hline
35  &  82  &  $4.55\times 10^{2}$  \\ \hline
40  &  94  &  $1.42\times 10^{2}$  \\ \hline
\hline\hline
        \end{tabular}
        \caption{Effectivities of \eqref{QOIEstbis}.} 
    \label{tab:QOIEstbis}
    \end{subtable}
\hfill 
\begin{subtable}[b]{0.3\textwidth}
        \centering
        \begin{tabular}{c|c|c}
            \hline \hline
$N_{pr}$ & $N_{du}$ & $\hat{\eta}^N_{s,max}$ \\ \hline
3  &  2  &  $3.59\times 10^{4}$  \\ \hline
4  &  4  &  $9.59\times 10^{5}$  \\ \hline
5  &  7  &  $4.45\times 10^{6}$  \\ \hline
8  &  9  &  $1.97\times 10^{9}$  \\ \hline
10  &  12  &  $3.4\times 10^{9}$  \\ \hline
13  &  16  &  $1.43\times 10^{10}$  \\ \hline
16  &  19  &  $2.15\times 10^{10}$  \\ \hline
17  &  26  &  $1.43\times 10^{10}$  \\ \hline
20  &  32  &  $1.61\times 10^{11}$  \\ \hline
23  &  37  &  $4.62\times 10^{9}$  \\ \hline
25  &  45  &  $7.92\times 10^{9}$  \\ \hline
30  &  52  &  $3.96\times 10^{11}$  \\ \hline
32  &  62  &  $1.78\times 10^{10}$  \\ \hline
35  &  74  &  $5.44\times 10^{10}$  \\ \hline
38  &  81  &  $4.05\times 10^{9}$  \\ \hline
\hline\hline
        \end{tabular}
        \caption{Effectivities of \eqref{QOIEstter}. }
    \label{tab:QOIEstter}
    \end{subtable}

    \caption{
    Effectivities obtained with the posteriori error estimate \eqref{QOIEst} on the left, \eqref{QOIEstbis} on the middle and \eqref{QOIEstter} on the right with respect to the primal and dual basis dimensions $\mathsf{N_{\rm pr}}$ and $\mathsf{N_{\rm du}}$  in the offline stage.}
     \label{tab:QOIEsteffi}
\end{table}

Here, $s^N $ and $ \Tilde{s}^{N_{\Tilde{s}},N}$ denote, respectively, the true and reduced quantities of interest obtained when the POD–Greedy algorithm is guided by the estimator $\Delta_{\Tilde{s}}^N$; $\overline{s}^N$ and $ \overline{\Tilde{s}}^{N_{\rm s},N}$ denote the corresponding quantities when the algorithm is driven by $\overline{\Delta}_{\Tilde{s}}^N$; and  $\hat{s}^N $ and $\hat{\Tilde{s}}^{N_{\Tilde{s}},N}$ denote the quantities obtained when basis generation is driven by $\hat{\Delta}_{\Tilde{s}}^N$.

\begin{figure}[H]
     \centering
     \begin{subfigure}[b]{0.35\textwidth}
         \centering
         \includegraphics[width=\textwidth]{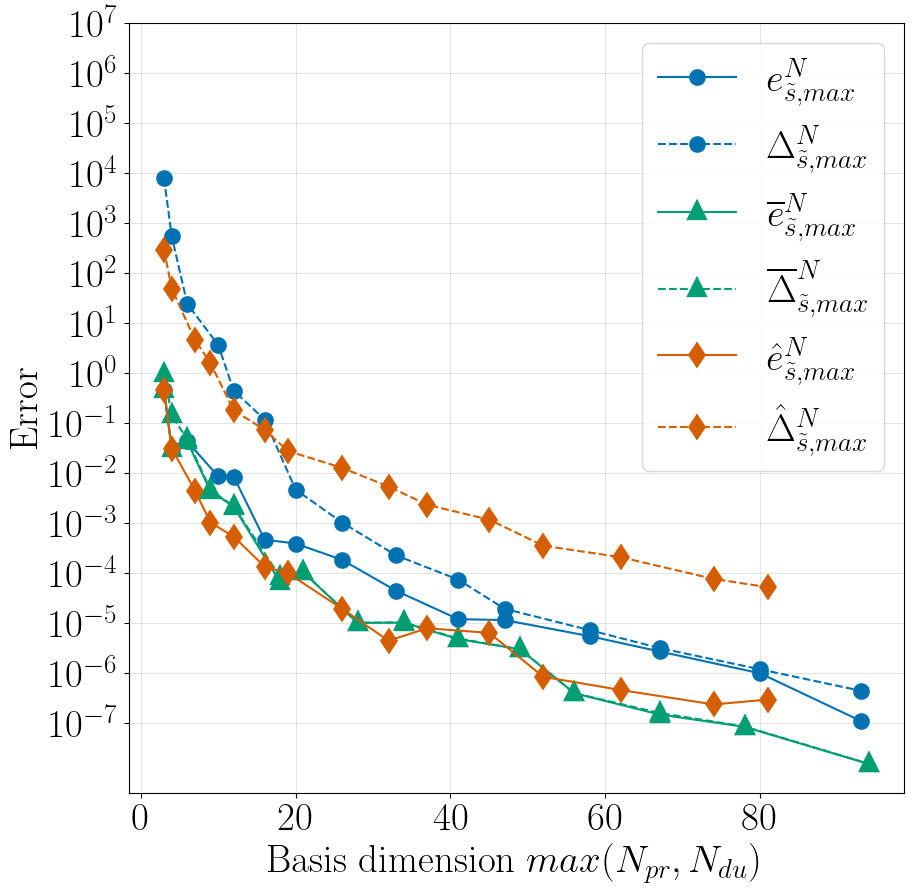}
         \caption{Offline phase} 
    \label{fig:ErrorsQOIOffline2}
     \end{subfigure}
     \begin{subfigure}[b]{0.35\textwidth}
         \centering
         \includegraphics[width=\textwidth]{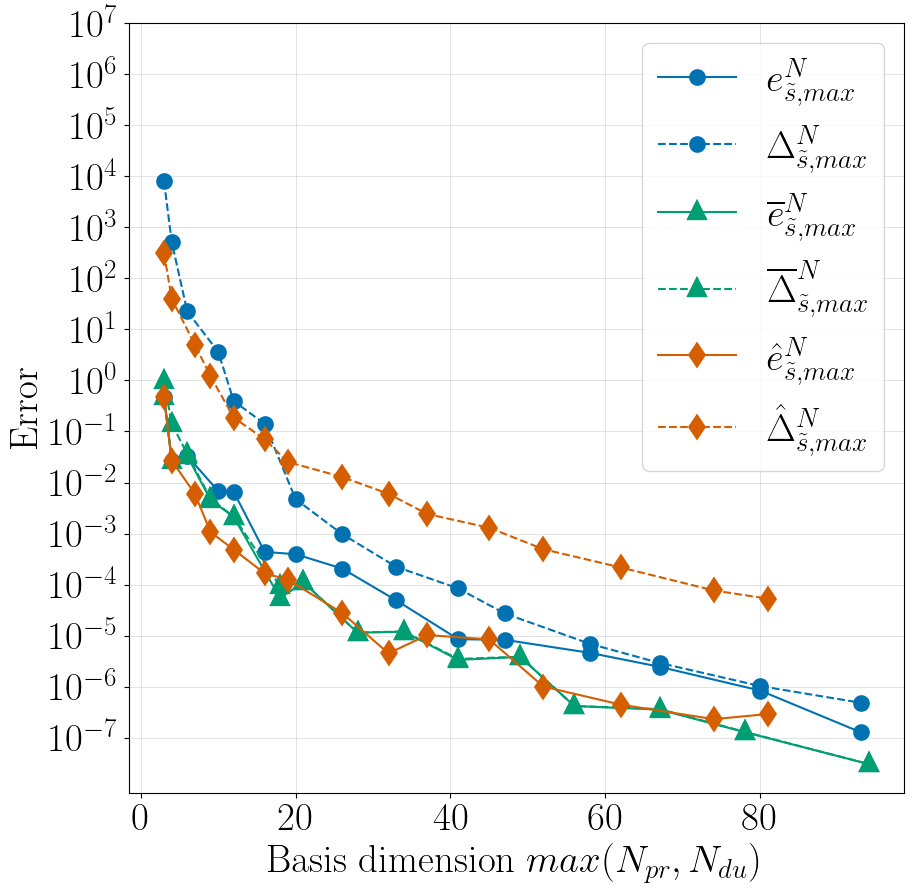}
         \caption{Online phase} 
\label{fig:ErrorsQOIOnline2}
     \end{subfigure}
      \caption{
      Maximum truth error and a posteriori estimator of the output \eqref{reducedOutput2} as a function of the primal and dual basis dimension using $\Xi_{\rm training}$ on the left and $\Xi_{\rm test}$ on the right for a POD-Greedy algorithm driven by the a posteriori error estimates \eqref{QOIEstWC}, \eqref{QOIEstWCbis} and \eqref{QOIEstWCter}.}
    \label{fig:ErrorsQOIWC}
\end{figure}

Figure \ref{fig:ErrorsQOIWC} shows the error evolution for estimators \eqref{QOIEstWC}, \eqref{QOIEstWCbis} and \eqref{QOIEstWCter}. With these estimators, the error decreases by only about six orders of magnitudes, compared with roughly eleven orders of magnitude when using the output definition \eqref{reducedOutput}. Table \ref{tab:QOIEstWC} reports the estimators effectivities: the estimates are not sharp in the early iterations, reaching effectivities of order $10^6$, and become much tighter as the basis grows, falling to order $10^1$ as the error decreases.

\begin{table}[t]
    \centering

    \begin{subtable}[b]{0.31\textwidth}
        \centering
        \begin{tabular}{c|c|c}
            \hline \hline
$N_{pr}$ & $N_{du}$ & $\eta^N_{\Tilde{s},max}$ \\ \hline
3  &  2  &  $2.12\times 10^{8}$  \\ \hline
4  &  4  &  $5.52\times 10^{6}$  \\ \hline
6  &  6  &  $4.26\times 10^{5}$  \\ \hline
8  &  10  &  $6.75\times 10^{5}$  \\ \hline
10  &  12  &  $1.12\times 10^{4}$  \\ \hline
12  &  16  &  $3.19\times 10^{4}$  \\ \hline
15  &  20  &  $6.44\times 10^{3}$  \\ \hline
16  &  26  &  $2.59\times 10^{2}$  \\ \hline
20  &  33  &  $2.53\times 10^{3}$  \\ \hline
22  &  41  &  $1.91\times 10^{3}$  \\ \hline
23  &  47  &  $4.0\times 10^{2}$  \\ \hline
27  &  58  &  $1.9\times 10^{2}$  \\ \hline
28  &  67  &  $1.18\times 10^{2}$  \\ \hline
31  &  80  &  $6.53\times 10^{1}$  \\ \hline
33  &  93  &  $4.2\times 10^{1}$  \\ \hline
\hline\hline
        \end{tabular}
    \caption{Effectivities of \eqref{QOIEstWC}. }
    
    \end{subtable}
    \hfill 
    \begin{subtable}[b]{0.31\textwidth}
        \centering
        \begin{tabular}{c|c|c}
            \hline \hline
$N_{pr}$ & $N_{du}$ & $\overline{\eta}^N_{\Tilde{s},max}$ \\ \hline
3  &  2  &  $5.22\times 10^{4}$  \\ \hline
4  &  4  &  $1.37\times 10^{3}$  \\ \hline
6  &  6  &  $8.56\times 10^{1}$  \\ \hline
8  &  9  &  $6.69\times 10^{1}$  \\ \hline
11  &  12  &  $6.63\times 10^{2}$  \\ \hline
13  &  18  &  $5.07\times 10^{1}$  \\ \hline
16  &  21  &  $2.28\times 10^{0}$  \\ \hline
19  &  28  &  $1.96\times 10^{1}$  \\ \hline
20  &  34  &  $8.55\times 10^{1}$  \\ \hline
21  &  41  &  $1.97\times 10^{1}$  \\ \hline
25  &  49  &  $1.24\times 10^{1}$  \\ \hline
29  &  56  &  $2.15\times 10^{1}$  \\ \hline
32  &  67  &  $8.82\times 10^{0}$  \\ \hline
36  &  78  &  $6.54\times 10^{0}$  \\ \hline
40  &  94  &  $1.59\times 10^{1}$  \\ \hline
\hline\hline
        \end{tabular}
    \caption{Effectivities of \eqref{QOIEstWCbis}. }
    \label{tab:QOIEstWCbis}
    \end{subtable}
    \hfill 
    \begin{subtable}[b]{0.31\textwidth}
        \centering
        \begin{tabular}{c|c|c}
            \hline \hline
$N_{pr}$ & $N_{du}$ & $\hat{\eta}^N_{\Tilde{s},max}$ \\ \hline
3  &  2  &  $2.02\times 10^{6}$  \\ \hline
4  &  4  &  $1.11\times 10^{5}$  \\ \hline
5  &  7  &  $8.2\times 10^{4}$  \\ \hline
8  &  9  &  $2.04\times 10^{5}$  \\ \hline
10  &  12  &  $1.22\times 10^{6}$  \\ \hline
13  &  16  &  $2.31\times 10^{5}$  \\ \hline
16  &  19  &  $1.43\times 10^{4}$  \\ \hline
17  &  26  &  $3.0\times 10^{4}$  \\ \hline
20  &  32  &  $3.3\times 10^{6}$  \\ \hline
23  &  37  &  $1.5\times 10^{5}$  \\ \hline
25  &  45  &  $8.57\times 10^{5}$  \\ \hline
30  &  52  &  $1.33\times 10^{5}$  \\ \hline
32  &  62  &  $3.96\times 10^{6}$  \\ \hline
35  &  74  &  $2.24\times 10^{5}$  \\ \hline
38  &  81  &  $1.44\times 10^{5}$  \\ \hline
\hline\hline
        \end{tabular}
    \caption{Effectivities of \eqref{QOIEstWCter} }
    \label{tab:QOIEstWCter}
    \end{subtable}
    \caption{
    Effectivities obtained with the posteriori error estimate \eqref{QOIEstWC} on the left,\eqref{QOIEstWCbis} on the middle, \eqref{QOIEstWCter} on the right with respect to the primal and dual basis dimensions $\mathsf{N_{\rm pr}}$ and $\mathsf{N_{\rm du}}$ in the offline stage.}
    \label{tab:QOIEstWC}
\end{table}

\paragraph{\orange{Simple QOI estimator}} Using the definition of the coercivity \eqref{alpha} and \eqref{proof7} we get
\begin{equation}\label{eqn:l_estim}
    \orange{\|e^N\|_{\boldsymbol{A}^*}^2\leq \frac{\Delta t^2}{\alpha_{\matA}\alpha}\sum_{m=1}^N\|\boldsymbol{r}^m\|_{\matA^{*^{-1}}}^2,}
\end{equation}
leading to the following error estimate
\begin{equation}\label{eqn:naive_estim_QOI}
    \vece_{\rm{\Tilde{s}}}^N \leq \orange{\sqrt{\boldsymbol{l}^T\boldsymbol{A}^{*^{-1}}\boldsymbol{l}}\left (\frac{\Delta t^2}{\alpha_{\matA}\alpha}\sum_{m=1}^N\|r^m\|_{\matA^{*^{-1}}}^2\right )^{\frac{1}{2}}}=\Delta_{\boldsymbol{l}}^N.
\end{equation}
We thus define
\begin{equation}\label{eqn:error_l_max}
    \vece_{\boldsymbol{l},max}^N=\max_{\xi \in \Xi} \; |\scas^N-\Tilde{\scas}^{\mathsf{N}_{\Tilde{s}},N}|, \Delta_{\boldsymbol{l},max}^N = \max_{\xi \in \Xi} \; \Delta_{\boldsymbol{l}}^N , \eta_{\boldsymbol{l},max}^N =\max_{\xi \in \Xi} \;  \frac{\Delta_{\boldsymbol{l}}^N}{\vece_{\Tilde{s}}^N}.
\end{equation}

Figures \ref{fig:ErrorsQOICompOffline} and  \ref{fig:ErrorsQOICompOnline} compare the decay of the output errors ${\vece}_{\rm s,\max}^N$ (equation \eqref{reducedOutput}), ${\vece}_{\Tilde{s},\max}^N$ (equation \eqref{reducedOutput2}) and ${\vece}_{l,max}^N$ (equation \eqref{eqn:error_l_max})  as a function of the primal-basis dimensions $\sfN_{\rm pr}$. Each curve corresponds to a POD-Greedy algorithm controlled by one of the estimators \eqref{QOIEst}, \eqref{QOIEstbis}, \eqref{QOIEstter}, \eqref{QOIEstWC}, \eqref{QOIEstWCbis},\eqref{QOIEstWCter} and \eqref{eqn:naive_estim_QOI}. We observe that using the definition of the output error \eqref{reducedOutput2} yields a slower convergence than with the definition \eqref{reducedOutput}. Figure \ref{fig:effi_QOI_comp} shows that the estimator \eqref{QOIEstter} is the most efficient one related to this latter reduced output. The estimators associated with \eqref{reducedOutput2} are generally more efficient because their true output error are larger. The natural upper bound \eqref{eqn:naive_estim_QOI} has substantially worse effectivity than the other estimator for the definition \eqref{reducedOutput2}. This behavior supports the use of the dual problem to obtain a tighter error bound.

\begin{figure}[H]
    \centering
   \begin{subfigure}[b]{0.3\linewidth}
    \centering
            \captionsetup{justification=raggedright,margin={1.2cm,0cm}}
    \includegraphics[width=\textwidth]{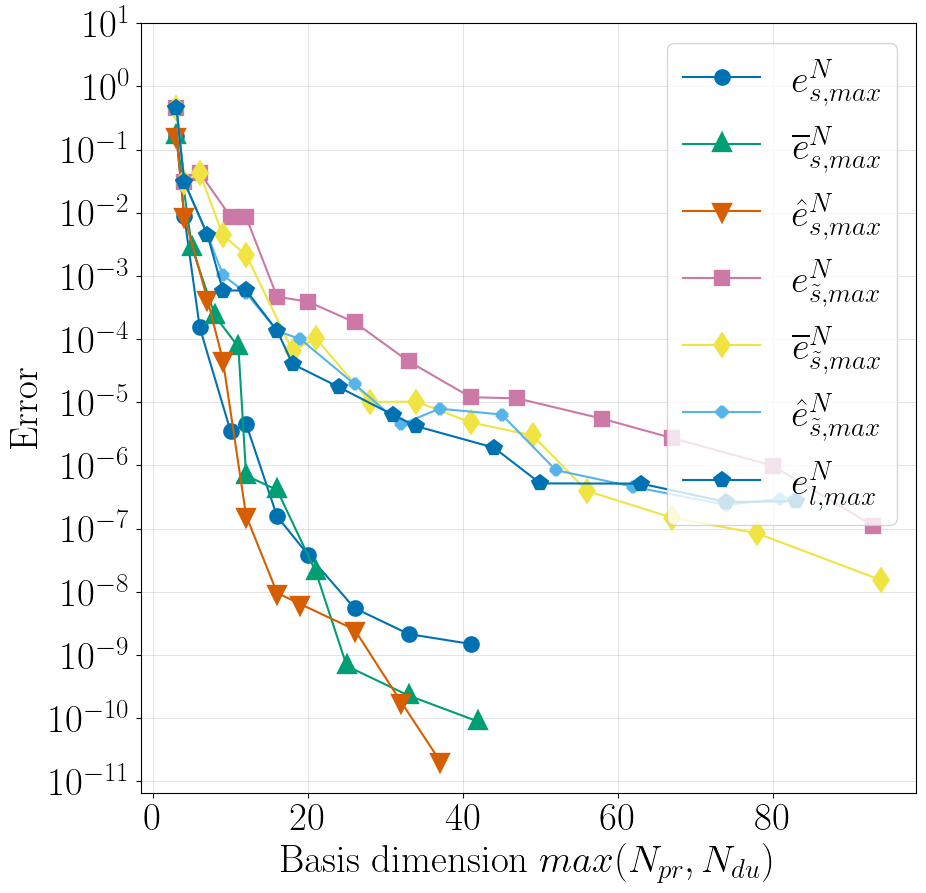} 
 \caption{Offline phase} 
  \label{fig:ErrorsQOICompOffline}
   \end{subfigure}
     \begin{subfigure}[b]{0.3\linewidth}
    \centering
            \captionsetup{justification=raggedright,margin={0.8cm,0cm}}
    \includegraphics[width=\textwidth]{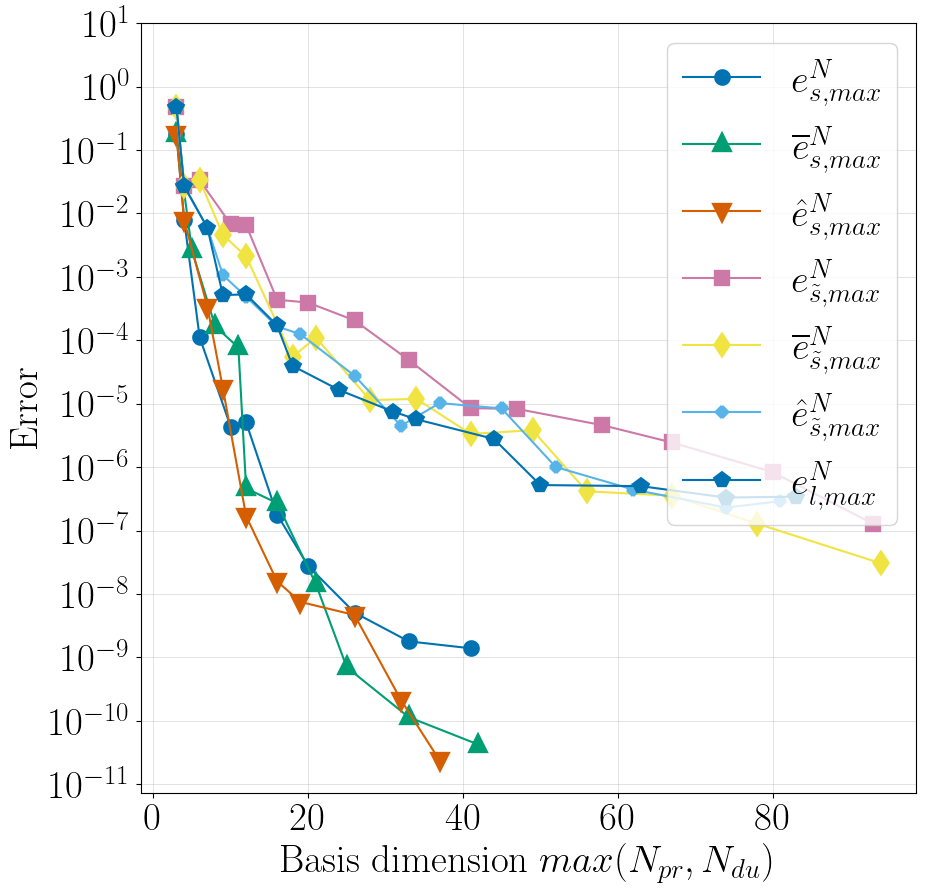} 
  \caption{Online phase} 
  \label{fig:ErrorsQOICompOnline}
   \end{subfigure}
\begin{subfigure}[b]{0.3\linewidth}
    \centering
            \captionsetup{justification=raggedright,margin={0.8cm,0cm}}
    \includegraphics[width=\textwidth]{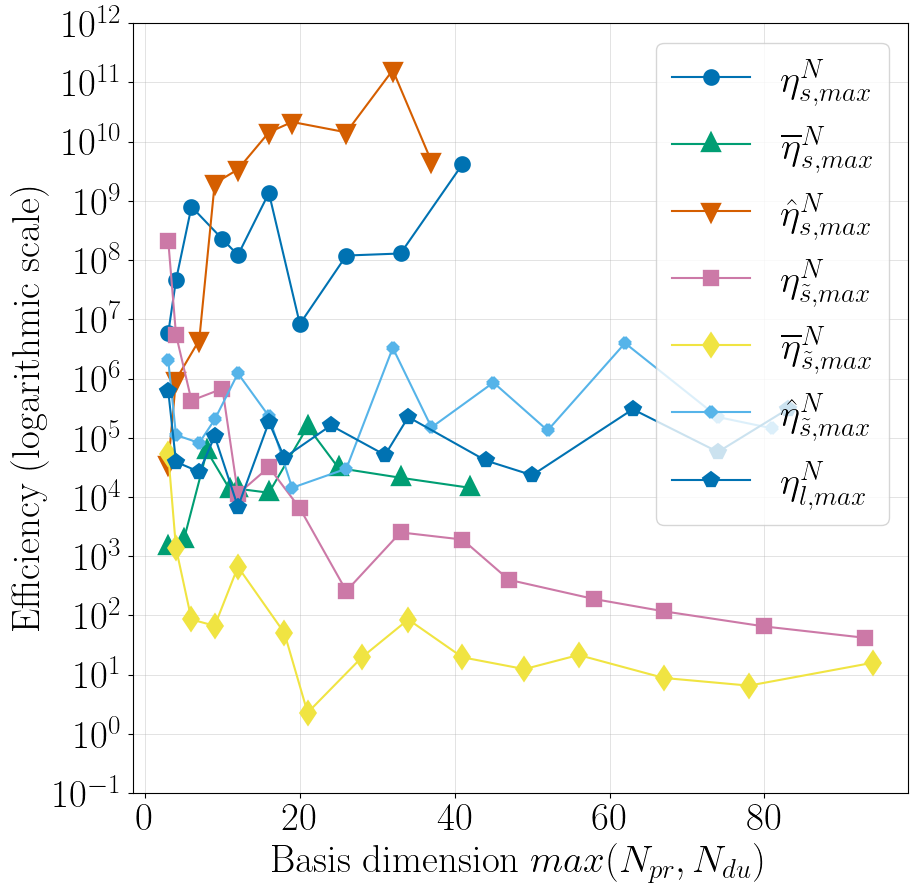} 
  \caption{Offline phase} 
  \label{fig:effi_QOI_comp}
   \end{subfigure}

    \caption{Truth errors for the reduced output \eqref{reducedOutput} and \eqref{reducedOutput2} as a function of the primal basis dimension for a POD-Greedy algorithm driven by \eqref{QOIEst}, \eqref{QOIEstbis}, \eqref{QOIEstter}, \eqref{QOIEstWC}, \eqref{QOIEstWCbis}, \eqref{QOIEstWCter} and \eqref{eqn:naive_estim_QOI} using $\Xi_{\rm training}$ on the left, $\Xi_{\rm test}$ on the middle. The different estimator efficiencies evolutions are represented on the right.}
    \label{fig:ErrorsQOIComp}
\end{figure}
Figure \ref{fig:ErrorsPrimalTime} displays the evolution of the error-estimate efficiency (or effectivity) when a POD-Greedy is driven by  \eqref{PrimalEstimator} and \eqref{PrimalEstimatorbis} for various final times $T$. \orange{The ratio $\eta_{pr,max}^N$ between the estimate \eqref{PrimalEstimator} and the error remains stable as $T$ grows.} For the primal problem \orange{with} $T=200$ days, Figure \ref{fig:evo_jana_effi_N_2} and \ref{fig:evo_oscar_effi_N} show that increasing $N$, and thus refining the time discretization, leads to an increase of both ratios $\eta^N_{pr}$, $\overline{\eta}^N_{pr}$ and \orange{of the basis dimension $N_{pr}$.} 

\begin{figure}[H]
    \centering
   \begin{subfigure}[b]{0.35\linewidth}
    \centering
            \captionsetup{justification=raggedright,margin={1cm,0cm}}
    \includegraphics[width=\textwidth]{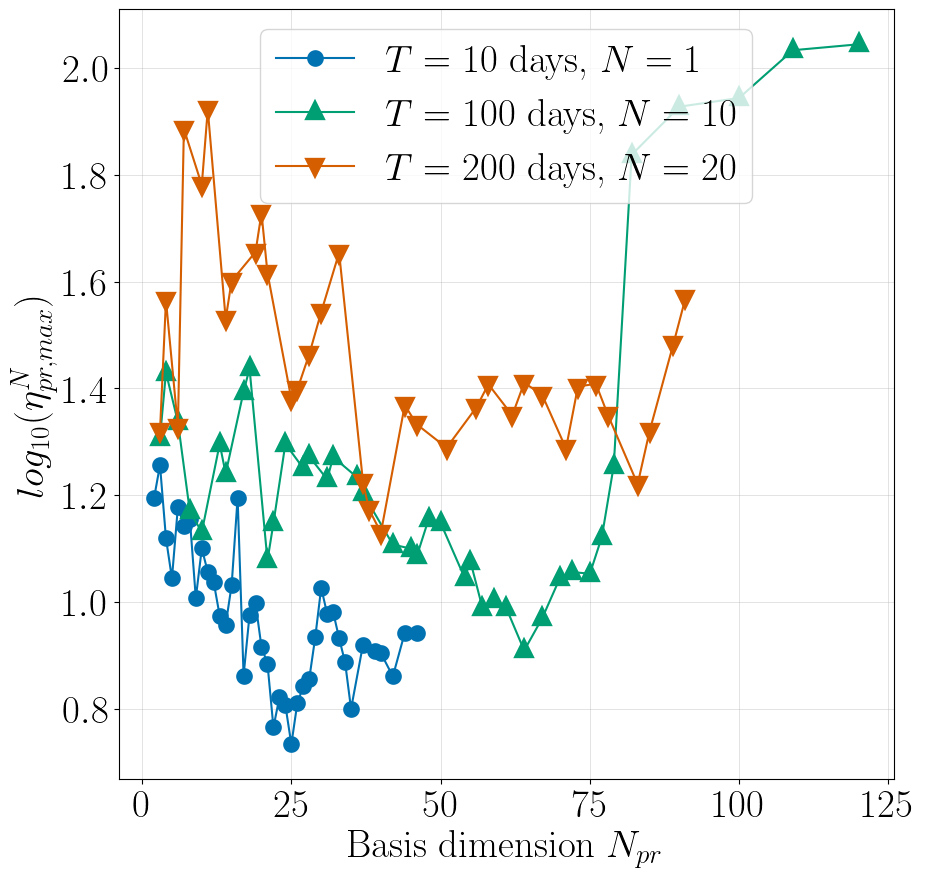} 
    \caption{Evolution of $\eta^N_{pr,max}$} 
  \label{fig:effi_QOI_jana}
   \end{subfigure}
     \begin{subfigure}[b]{0.35\linewidth}
    \centering
            \captionsetup{justification=raggedright,margin={0.6cm,0cm}}
    \includegraphics[width=\textwidth]{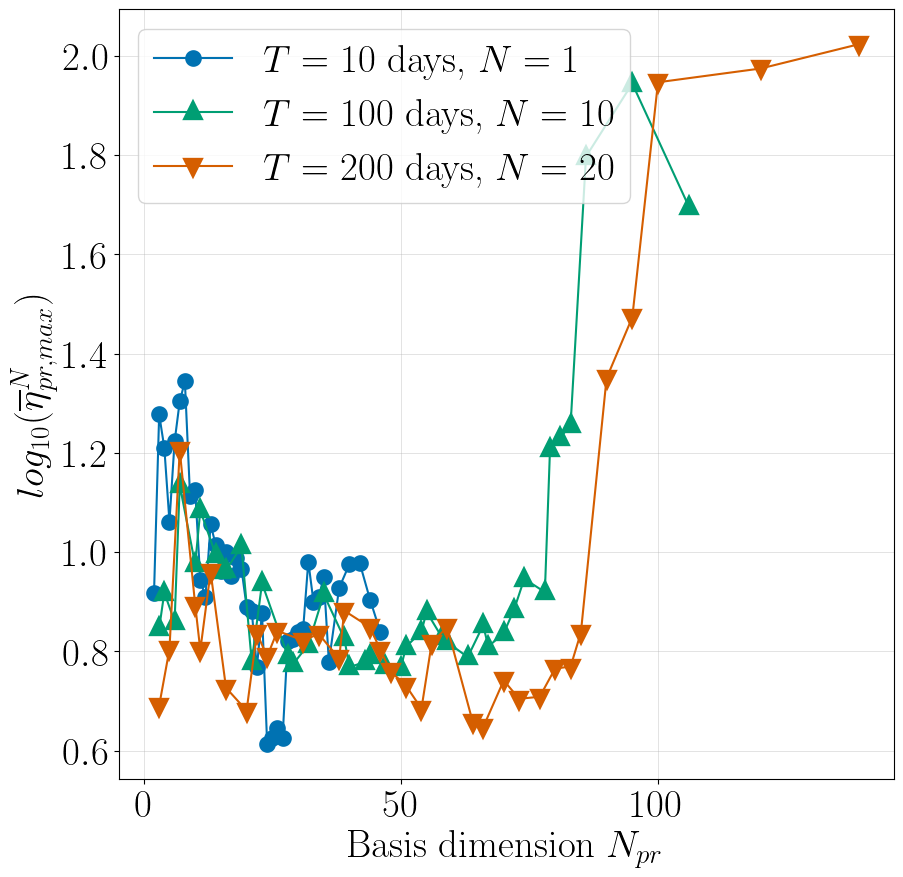} 
  \caption{Evolution of $\overline{\eta}^N_{pr,max}$} 
  \label{fig:effi_QOI_oscar}
   \end{subfigure}
  \caption{Efficiency for the primal estimators as functions of the primal basis dimension with $T=10$ days, $T=100$ days and $T=200$ days and a fixed time-step.}
    \label{fig:ErrorsPrimalTime}
\end{figure}

\begin{figure}[H]
    \centering
   \begin{subfigure}[b]{0.35\linewidth}
    \centering

            \captionsetup{justification=raggedright,margin={1cm,0cm}}
    \includegraphics[width=\textwidth]{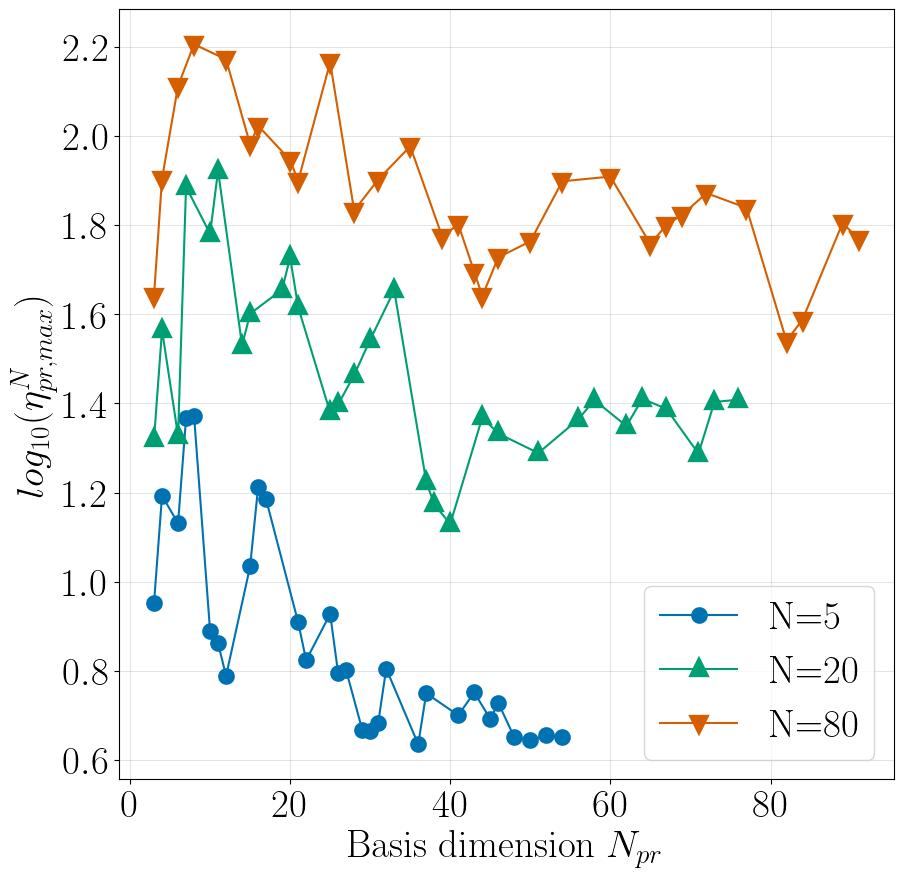} 
    
 \caption{Evolution of $\eta^N_{pr}$} 
    \label{fig:evo_jana_effi_N_2}
   \end{subfigure}
     \begin{subfigure}[b]{0.35\linewidth}
    \centering
    
            \captionsetup{justification=raggedright,margin={0.6cm,0cm}}
    \includegraphics[width=\textwidth]{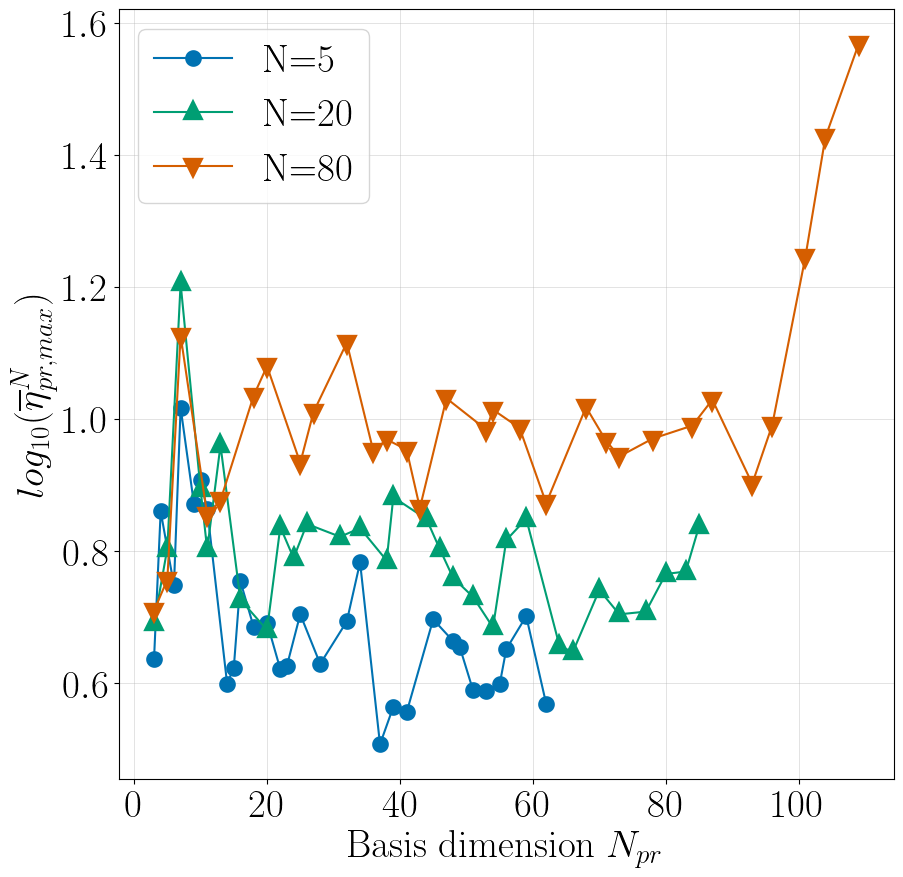} 
  \caption{Evolution of $\overline{\eta}^N_{pr}$}
    \label{fig:evo_oscar_effi_N}
   \end{subfigure}
  \caption{Evolution of the estimator efficiencies  as functions of the primal basis dimension with fixed final time $T$ and various time steps.}
\end{figure}

\subsection{Computational time effort}
Here we give an indication of the time spent at each stage of the construction and use
of the reduced model.

\paragraph{Offline stage.}
Figure \ref{fig:timeCompOff} reports the cumulative computation times of the reduced-basis offline stage as functions of the primal basis dimension $\sfN_{\rm pr}$. All times are normalized by the average time $u$ of a single high-fidelity simulation. EIM and SCM are performed once at the beginning of the offline stage. Their cumulated times thus appear as constant in the plot. These times amount to \orange{$2.8\times u$ and $9.2\times u$} respectively.
 
The label "Greedy" represents the cumulative time spent in the POD-Greedy iterations  (plotted up to $\sfN_{\rm pr}=95$). “UB Computation” denotes the cumulative time required to assemble and solve the reduced systems \eqref{primalReducedEq} and \eqref{dualReducedEq}-\eqref{dualReducedEqFinal} and to compute the residuals \eqref{resAss1}--\eqref{ResidualConst2} over the sampling set. These costs grow with $\sfN_{\rm pr}$ and $\sfN_{\rm du}$ but remain negligible.   "HR Update" denotes the cumulative cost of updating the reduced linear system, the Riesz representative \eqref{resAss1}, and its orthonormal basis $\zeta$ ; this cost is marginal. By far, the largest offline contribution is "HF-POD", representing the cumulative time spent computing exact HF solutions and performing the POD compression.

\begin{figure}[H]
    \centering
   \begin{subfigure}[b]{0.40\linewidth}
    \centering
            \captionsetup{justification=raggedright,margin={1cm,0cm}}
    \includegraphics[width=\textwidth]{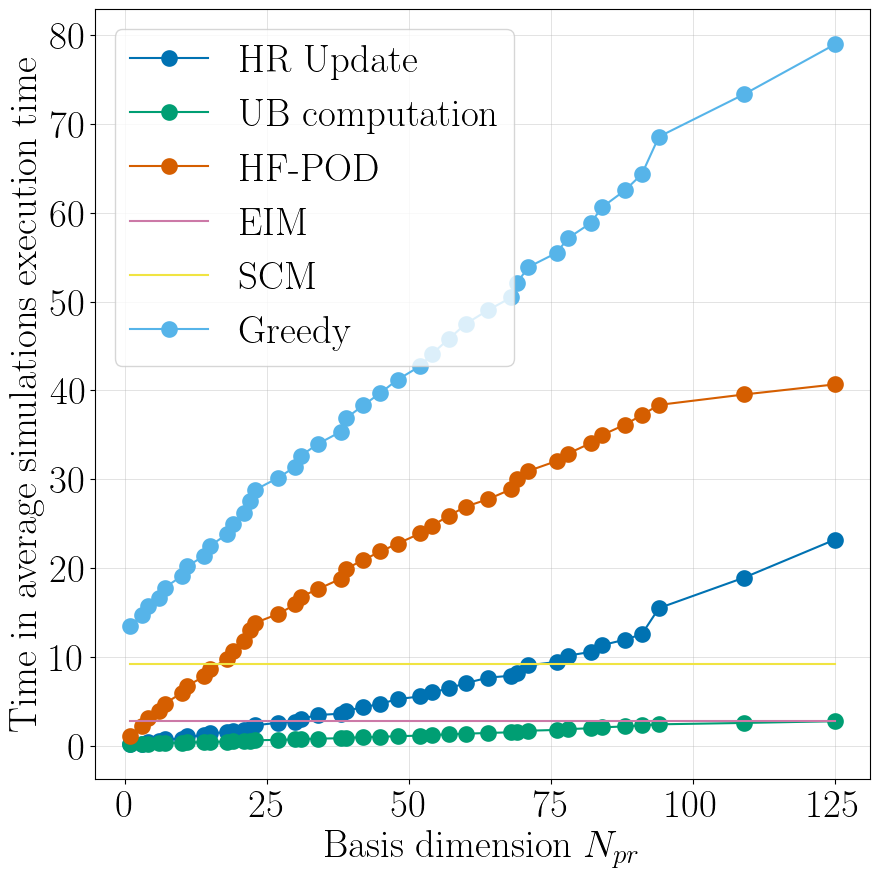}
    \caption{Offline time computational effort.} 
   \end{subfigure}
     \begin{subfigure}[b]{0.40\linewidth}
    \centering
            \captionsetup{justification=raggedright,margin={0.6cm,0cm}}

    \includegraphics[width=\textwidth]{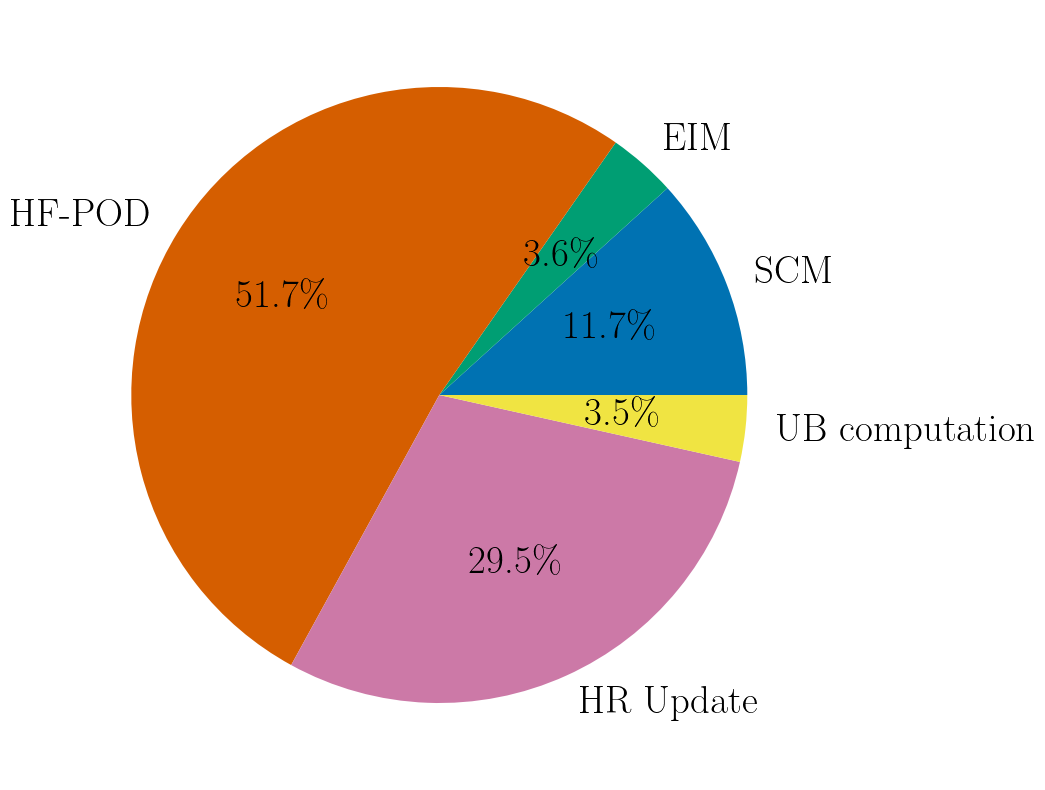}
    \caption{Total time time spent in each part of the offline phase.}
   \end{subfigure}
  \caption{Offline time computational effort.}
  \label{fig:timeCompOff}
\end{figure}

\paragraph{Online stage.}
Given a new parameter value $\xi \in \Xi_{\rm test}$, with our implementation, the online time to compute $\vecp^{\sfN_{\rm pr},N}$ at $T=200$ days \orange{with $\sfN_{\rm pr}=95$ is divided by a factor of $155$ compared to one high-fidelity run needed to obtain $\vecp_\calM^{N}$. The HF model \eqref{SF} is solved using a BiCGStab with a relative tolerance fixed at $10^{-14}$.} The reduced model \eqref{primalReducedEq} is directly solved using an LU-Decomposition.

\section{Conclusion}\label{Conclusion}

In this work, we have discussed a reduced basis method for (finite volume approximations of) parabolic PDEs. We have introduced two new rigorous a posteriori estimates to evaluate the reduction error in a new discrete space-time energy norm independent of the parameter. We have performed numerical simulations in the context of porous media flows (single-phase flows of slightly compressible fluid parametrized by the permeability) to assess the reliability of the a posteriori error bounds and their efficiency at selecting a reduced basis within a POD-Greedy algorithm. None of the proposed error estimates has demonstrated a definitive superior convergence of the true error in all cases. Their sharpness, however, varies considerably. In a certification context, insufficient sharpness can force the use of an unnecessarily large and expensive reduced basis, thereby undermining the expected computational savings.

Our numerical results show that the reduced basis method approach can efficiently reduce the computational cost of engineering studies with parameter values in the context of porous media flows, especially for goal-oriented cases with linear QOIs. The discussed methodology can be considered to estimate different types of linear quantity of interests such as the pressure variation along faults far from the well injection area. Indeed, understanding how injection activities affect pressure distribution in fault networks helps in mitigating risks associated with CO2 migration, fault activation, and potential leakage into overlying aquifers.

\section*{Acknowledgement}
This work was supported by IFPEN, INRIA and \'Ecole nationale des ponts et chauss\'ees. The authors would like to thank Oscar Demont for his contribution in the definition of variants of the a posteriori error estimates.


{\small
\bibliographystyle{siamplain}
\bibliography{biblio}

\begin{thebibliography}{10}

\bibitem{aavatsmark1994discretization}
{\sc I.~Aavatsmark, T.~Barkve, {\O}.~B{\o}e, and T.~Mannseth}, {\em
  Discretization on non-orthogonal, curvilinear grids for multi-phase flow},
  ECMOR IV - 4th European Conference on the Mathematics of Oil Recovery,
  (1994), \url{https://doi.org/https://doi.org/10.3997/2214-4609.201411179}.

\bibitem{agelas2009nine}
{\sc L.~Ag{\'e}las, R.~Eymard, and R.~Herbin}, {\em A nine-point finite volume
  scheme for the simulation of diffusion in heterogeneous media}, Comptes
  Rendus Math{\'e}matique, 347 (2009), pp.~673--676,
  \url{https://doi.org/10.1016/j.crma.2009.03.013}.

\bibitem{BeiraoEtAl2013BasicsOfVEM}
{\sc L.~Beirão Da~Veiga, F.~Brezzi, A.~Cangiani, G.~Manzini, L.~D. Marini, and
  A.~Russo}, {\em Basic principles of virtual element methods}, Mathematical
  Models and Methods in Applied Sciences, 23 (2013), pp.~199--214,
  \url{https://doi.org/10.1142/S0218202512500492}.

\bibitem{berger1984adaptive}
{\sc M.~J. Berger and J.~Oliger}, {\em Adaptive mesh refinement for hyperbolic
  partial differential equations}, J. Comput. Phys., 53 (1984), pp.~484--512,
  \url{https://doi.org/10.1016/0021-9991(84)90073-1}.

\bibitem{brenier1991upstream}
{\sc Y.~Brenier and J.~Jaffr{\'e}}, {\em Upstream differencing for multiphase
  flow in reservoir simulation}, SIAM Journal on Numerical Analysis, 28 (1991),
  pp.~685--696, \url{https://doi.org/10.1137/0728036}.

\bibitem{brezzi2005family}
{\sc F.~Brezzi, K.~Lipnikov, and V.~Simoncini}, {\em A family of mimetic finite
  difference methods on polygonal and polyhedral meshes}, Mathematical Models
  and Methods in Applied Sciences, 15 (2005), pp.~1533--1551,
  \url{https://doi.org/10.1142/S0218202505000832}.

\bibitem{Buhr2014ANS}
{\sc A.~Buhr, C.~Engwer, M.~Ohlberger, and S.~Rave}, {\em A numerically stable
  a posteriori error estimator for reduced basis approximations of elliptic
  equations}, 2014, \url{https://arxiv.org/abs/1407.8005}.

\bibitem{canuto2009posteriori}
{\sc C.~Canuto, T.~Tonn, and K.~Urban}, {\em A posteriori error analysis of the
  reduced basis method for nonaffine parametrized nonlinear pdes}, SIAM Journal
  on Numerical Analysis, 47 (2009), pp.~2001--2022,
  \url{https://doi.org/10.1137/080724812}.

\bibitem{casenave2012accurate}
{\sc F.~Casenave}, {\em Accurate a posteriori error evaluation in the reduced
  basis method}, Comptes Rendus Mathématique, 350 (2012), pp.~539--542,
  \url{https://doi.org/10.1016/j.crma.2012.05.012}.

\bibitem{chaturantabut2010nonlinear}
{\sc S.~Chaturantabut and D.~C. Sorensen}, {\em Nonlinear model reduction via
  discrete empirical interpolation}, SIAM Journal of Scientific Computing, 32
  (2010), pp.~2737--2764, \url{https://doi.org/10.1137/090766498}.

\bibitem{chavent1986mathematical}
{\sc G.~Chavent and J.~Jaffr{\'e}}, {\em Mathematical models and finite
  elements for reservoir simulation: single phase, multiphase and
  multicomponent flows through porous media}, vol.~17, Elsevier, 1986.

\bibitem{CHEN20081295}
{\sc Y.~Chen, J.~S. Hesthaven, Y.~Maday, and J.~Rodríguez}, {\em A monotonic
  evaluation of lower bounds for inf-sup stability constants in the frame of
  reduced basis approximations}, Comptes Rendus Mathématique, 346 (2008),
  pp.~1295--1300, \url{https://doi.org/10.1016/j.crma.2008.10.012}.

\bibitem{Chen2007ReservoirS}
{\sc Z.~Chen}, {\em Reservoir simulation: mathematical techniques in oil
  recovery}, vol.~77 of CBMS-NSF Regional Conference Series in Applied
  Mathematics, SIAM, Philadelphia, 2007.

\bibitem{chen2001degenerate}
{\sc Z.~Chen and R.~E. Ewing}, {\em Degenerate two-phase incompressible flow
  iii}, Numerische Mathematik, 90 (2001), pp.~215--240,
  \url{https://doi.org/https://doi.org/10.1007/s002110100291}.

\bibitem{eymard2000finite}
{\sc R.~Eymard, T.~Gallou{\"e}t, and R.~Herbin}, {\em Finite volume methods},
  Handbook of numerical analysis, 7 (2000), pp.~713--1018,
  \url{https://doi.org/10.1016/S1570-8659(00)07005-8}.

\bibitem{sushi}
{\sc R.~Eymard, T.~Gallou\"et, and R.~Herbin}, {\em Discretisation of
  heterogeneous and anisotropic diffusion problems on general nonconforming
  meshes {SUSHI}: a scheme using stabilisation and hybrid interfaces}, IMA
  Journal of Numerical Analysis, 30 (2009), pp.~1009--1043,
  \url{https://doi.org/10.1093/imanum/drn084}.

\bibitem{eymard2012vertex}
{\sc R.~Eymard, C.~Guichard, R.~Herbin, and R.~Masson}, {\em Vertex-centred
  discretization of multiphase compositional darcy flows on general meshes},
  Computational Geosciences, 16 (2012), pp.~987--1005,
  \url{https://doi.org/10.1007/s10596-012-9299-x}.

\bibitem{grepl}
{\sc M.~A. Grepl}, {\em Reduced basis approximation and a posteriori error
  estimation for parabolic partial differential equations}, PhD thesis, MIT,
  2005, \url{http://hdl.handle.net/1721.1/32387}.

\bibitem{grepl2007efficient}
{\sc M.~A. Grepl, Y.~Maday, N.~C. Nguyen, and A.~T. Patera}, {\em Efficient
  reduced-basis treatment of nonaffine and nonlinear partial differential
  equations}, ESAIM: Mathematical Modelling and Numerical Analysis, 41 (2007),
  pp.~575--605, \url{https://doi.org/DOI: 10.1051/m2an:2007031}.

\bibitem{M2AN_2005__39_1_157_0}
{\sc M.~A. Grepl and A.~T. Patera}, {\em A posteriori error bounds for
  reduced-basis approximations of parametrized parabolic partial differential
  equations}, ESAIM: Mathematical Modelling and Numerical Analysis, 39 (2005),
  pp.~157--181, \url{https://doi.org/10.1051/m2an:2005006}.

\bibitem{Haasdonk2008}
{\sc B.~Haasdonk and M.~Ohlberger}, {\em Reduced basis method for finite volume
  approximations of parametrized linear evolution equations}, ESAIM:
  Mathematical Modelling and Numerical Analysis, 42 (2008), pp.~277--302,
  \url{https://doi.org/10.1051/m2an:2008001}.

\bibitem{HUYNH2007473}
{\sc D.~B.~P. Huynh, G.~Rozza, S.~Sen, and A.~T. Patera}, {\em A successive
  constraint linear optimization method for lower bounds of parametric
  coercivity and inf–sup stability constants}, Comptes Rendus Mathématique,
  345 (2007), pp.~473--478, \url{https://doi.org/10.1016/j.crma.2007.09.019}.

\bibitem{MR3518747}
{\sc L.~Iapichino, S.~Volkwein, and A.~Wesche}, {\em A-posteriori error
  analysis for lithium-ion concentrations in batteries utilizing the
  reduced-basis method}, Math. Comput. Model. Dyn. Syst., 22 (2016),
  pp.~362--379, \url{https://doi.org/10.1080/13873954.2016.1198387},
  \url{https://doi.org/10.1080/13873954.2016.1198387}.

\bibitem{le2005schema}
{\sc C.~Le~Potier}, {\em Sch{\'e}ma volumes finis monotone pour des
  op{\'e}rateurs de diffusion fortement anisotropes sur des maillages de
  triangles non structur{\'e}s}, Comptes Rendus Math{\'e}matique, 341 (2005),
  pp.~787--792, \url{https://doi.org/10.1016/j.crma.2005.10.010}.

\bibitem{lipnikov2009interpolation}
{\sc K.~Lipnikov, D.~Svyatskiy, and Y.~Vassilevski}, {\em Interpolation-free
  monotone finite volume method for diffusion equations on polygonal meshes},
  Journal of Computational Physics, 228 (2009), pp.~703--716,
  \url{https://doi.org/10.1016/j.jcp.2008.09.031}.

\bibitem{Maday2016ConvergenceAO}
{\sc Y.~Maday, O.~Mula, and G.~Turinici}, {\em Convergence analysis of the
  {G}eneralized {E}mpirical {I}nterpolation {M}ethod}, SIAM Journal of
  Numerical Analysis, 54 (2016), pp.~1713--1731,
  \url{https://doi.org/10.1137/140978843}.

\bibitem{sammon1988analysis}
{\sc P.~H. Sammon}, {\em {An Analysis of Upstream Differencing}}, SPE Reservoir
  Engineering, 3 (1988), pp.~1053--1056,
  \url{https://doi.org/10.2118/14045-PA}.

\bibitem{schneider2017convergence}
{\sc M.~Schneider, L.~Ag{\'e}las, G.~Ench{\'e}ry, and B.~Flemisch}, {\em
  Convergence of nonlinear finite volume schemes for heterogeneous anisotropic
  diffusion on general meshes}, Journal of Computational Physics, 351 (2017),
  pp.~80--107, \url{https://doi.org/10.1016/j.jcp.2017.09.003}.

\bibitem{schneider2018monotone}
{\sc M.~Schneider, B.~Flemisch, R.~Helmig, K.~Terekhov, and H.~Tchelepi}, {\em
  Monotone nonlinear finite-volume method for challenging grids}, Computational
  Geosciences, 22 (2018), pp.~565--586,
  \url{https://doi.org/10.1007/s10596-017-9710-8}.

\bibitem{verfurth1994posteriori}
{\sc R.~Verf{\"u}rth}, {\em A posteriori error estimation and adaptive
  mesh-refinement techniques}, J. Comput. Appl. Math., 50 (1994), pp.~67--83,
  \url{https://doi.org/10.1016/0377-0427(94)90290-9}.

\end{thebibliography}
}

\appendix

\section{Proofs of various propositions}
\subsection{Proof of Proposition \ref{Prop1}}\label{Aproof1}
\begin{proof}
We point out the proof starts with a similar approach to the one introduced in the literature \cite{Haasdonk2008}. For each $\boldsymbol{v} \in \R^{\calN}$, we have
\begin{equation}\label{Proof1}
\begin{split}
\left\langle(\matM +\Delta t \matA)\vece^{n},{\boldsymbol{v}} \right\rangle&=
\left\langle\matM\vece^{n-1},{\boldsymbol{v}} \right\rangle -\Delta t \, \langle\vecr^{n},{\boldsymbol{v}}\rangle.
\end{split}    
\end{equation}
Using $\boldsymbol{v}=\vece^n$ in \eqref{Proof1}, Cauchy-Schwarz inequality and \eqref{ResidualNormPrimal}, it follows
\begin{equation}\label{proof2}
\left\langle (\matM +\Delta t \matA)\vece^{n},\vece^n \right\rangle \leq \|\matM^{1/2}\vece^{n-1}\| \|\matM^{1/2}\vece^{n}\|+\Delta t\; \|\vecr^n\|_{\matA^{*^{-1}}}  \|\vece^n\|_{\boldsymbol{A}^*}.
\end{equation}
Now, recalling Young's inequality (for $c \in \mathbb{R }$, $d\in \mathbb{R}$, $\rho \in \mathbb{R}_+$): 
\begin{equation}\label{Young}
 2 |c|\, |d| \leq \frac{1}{\rho^2}c^2 + \rho^2 d^2,   
\end{equation}
and apply it twice: once for $c=\|\matM^{1/2}\vece^{n-1}\|$, $d=\|\matM^{1/2}\vece^{n}\| $ and $\rho=1$ to get 
\begin{equation}\label{proof3}
  2 \|\matM^{1/2}\vece^{n-1}\|\,\|\matM^{1/2}\vece^{n}\|\leq  \langle \matM\vece^{n-1},\vece^{n-1} \rangle +\langle \matM\vece^{n},\vece^n \rangle
\end{equation}
and another time for $c=\|\vecr^n\|_{\matA^{*^{-1}}}$, $d=\|\vece^n\|_{\matA^*}$  and ${\rho=\sqrt{\alpha_{\matA}}}$ to obtain
\begin{equation}\label{proof4}
2\|\vecr^n\|_{\matA^{*^{-1}}} \|\vece^n\|_{\matA^*} \leq {\frac1{\alpha_{\matA}}} \; \|\vecr^n\|_{\matA^{*^{-1}}}^2 + \alpha_{\matA} \;\|\vece^n\|_{\matA^*}^2.
\end{equation}
Now the definition \eqref{alphaE} of the coercivity constant leads to 
\begin{equation}\label{proof5}
2\|\vecr^n\|_{\matA^{*^{-1}}}\;\|\vece^n\|_{\boldsymbol{A}^*} \leq \frac1{\alpha_{\matA}} \;\|\vecr^n\|_{\matA^{*^{-1}}}^2 + \langle  \matA_{\rm sym} \;\vece^n, \vece^n \rangle.
\end{equation}
Combining \eqref{proof2}, \eqref{proof3} and \eqref{proof5} yields 
\begin{equation}
\begin{split}
  \langle \matM \vece^{n},\vece^n \rangle +  \Delta t \; \langle \matA  \vece^n, \vece^n \rangle &\leq \frac{1}{2}\langle \matM\vece^{n-1},\vece^{n-1} \rangle + \frac{1}{2} \langle \matM\vece^{n},\vece^n \rangle\\&+\frac{\Delta t }{2} \langle  \matA_{\rm sym}\; \vece^n, \vece^n \rangle+\frac{{\Delta t} }{2 \alpha_{\matA}} \;\|\vecr^n\|_{\matA^{*^{-1}}}^2.
\end{split}
\end{equation}
Since $\langle \matA_{\rm skew} \;\vece^n,\vece^n \rangle=0 $ and $\langle \matA_{\rm sym} \;\vece^n,\vece^n \rangle=\langle \matA \vece^n,\vece^n \rangle $, we obtain 
\begin{equation}\label{proof6}
\langle \matM \vece^{n},\vece^n \rangle - \langle \matM\vece^{n-1},\vece^{n-1} \rangle +  \Delta t  \;\langle \matA_{\rm sym}\;  \vece^n, \vece^n \rangle \leq \frac{{\Delta t}}{ \alpha_{\matA}} \|\vecr^n\|_{\matA^{*^{-1}}}^2.
\end{equation}
Finally, we sum \eqref{proof6} over $\{1,\ldots,n\}$ and consider that $\vece^0=0$ to get 
\begin{equation}\label{firstErrorInq}
\langle \matM \vece^{n},\vece^n \rangle + \sum_{m=1}^{ n} \Delta t \; \langle \matA_{\rm sym}  \;\vece^m, \vece^m \rangle \leq \sum_{m=1}^{ n} \frac{{\Delta t}}{ \alpha_{\matA}} \|\vecr^m\|_{\matA^{*^{-1}}}^2.
\end{equation}
In \eqref{firstErrorInq}, the matrix $\matA_{\rm sym}$ depends on the parameter, we therefore extend the work of \cite{Haasdonk2008} and introduce an energy error independent of the parameter. From \eqref{firstErrorInq}, we have 
\begin{equation}\label{proof7}
\begin{split}
\langle \matM \vece^{n},\vece^n \rangle + \Delta t \; \langle \matA_{\rm sym}  \;\vece^n, \vece^n \rangle \leq \sum_{m=1}^{ n} \frac{{\Delta t}}{ \alpha_{\matA}} \|\vecr^m\|_{\matA^{*^{-1}}}^2.
\end{split}
\end{equation}
\orange{
Using \eqref{alpha} and summing over $\{1,\ldots,N\}$, we get 
\begin{equation}
\begin{split}
\alpha \Delta t\; \sum_{m=1}^{N} \langle \matA^* \; \vece^m, \vece^m \rangle 
& \leq \Delta t\sum_{m=1}^{ N} [\langle \matM \vece^{m},\vece^m \rangle  +\Delta t\;\langle \matA_{\rm sym}  \;\vece^m, \vece^m \rangle ]\\  &\leq \frac{{\Delta t^2}}{\alpha_{\matA}} \sum_{m=1}^{ N}  (N+1-m)\|\vecr^m\|_{\matA^{*^{-1}}}^2\\
\end{split}
\end{equation}
which is an upper bound of an error that is independent of the parameter $\xi$ as opposed to \cite[Prop. 4.3]{Haasdonk2008}. 
From \eqref{firstErrorInq}, using the definition \eqref{alphaE} of the coercivity constant, we also have
\begin{equation}
    \alpha_{\matA}\Delta t\sum_{m=1}^{N} \langle \matA^* \; \vece^m, \vece^m \rangle \leq \Delta t\sum_{m=1}^{N} \big[\;\langle \matA_{sym} \; \vece^m, \vece^m \rangle \big]\leq  \sum_{m=1}^{ n} \frac{{\Delta t}}{ \alpha_{\matA}} \|\vecr^m\|_{\matA^{*^{-1}}}^2
\end{equation}
which is another upperbound of a parameter-independent error, reminiscent of the one proposed in \cite{grepl}.}

\end{proof}

\subsection{Proof of Proposition \ref{prop2}}\label{Aproof2}
We first estimate the dual error at the final time iteration $t=N$. We propose the following lemma.
\begin{lemma}\label{lemma:dualErrorFinal}
 We have for $n=N$,
 \begin{equation}\label{eq:dualErrorEstFinalT}
     \big\langle {\vecepsi^N},\matM {\vecepsi}^N\big\rangle \leq \frac{\| \vecrho^N\|_{\matA^{*^{-1}}}^2 }{\alpha_{\matM}}.
 \end{equation}
\end{lemma}
\begin{proof}
First, in view of the definition of the output functional $\boldsymbol{l}$, we can write
\begin{equation}
\label{eq:annexDualEst1}
    \begin{split}
        \big\langle {\vecepsi^N},\matM {\vecepsi}^N\big\rangle &= \big\langle \vecrho^N,{\vecepsi}^N\big\rangle 
        \\ & \leq \| \vecrho^N\|_{\matA^{*^{-1}}}  \|\vecepsi^N\|_{\matA^*} \mbox{ according to \eqref{ResidualNormDual}.}
    \end{split}
\end{equation}
On the other hand, using \eqref{alphaM},
\begin{equation}
    \begin{split}
         \|\vecepsi^N\|_{\matA^*}^2 &\leq \frac{\langle \matM \vecepsi^N,\vecepsi^N \rangle}{\alpha_{\matM}}
    \\& \leq \frac{\|\vecrho^N\|_{\matA^{*^{-1}}} \|\vecepsi^N\|_{\matA^*} }{\alpha_{\matM}} \mbox{ according to \eqref{eq:annexDualEst1}}
    \end{split}
\end{equation}
which results in 
\begin{equation}
     \big\langle {\vecepsi^N},\matM {\vecepsi}^N\big\rangle\leq  \frac{\|\vecrho^N\|_{\matA^{*^{-1}}}^2 }{\alpha_{\matM}}.
\end{equation}
This concludes Lemma \ref{lemma:dualErrorFinal} proof.
\end{proof}
\textbf{\orange{Proof of Proposition \ref{prop2}.}}
According to \eqref{NewDualPbb} and \eqref{eq:redDualPb}, we have for $m=0,\ldots,N-1$
$$\big \langle (\matM +\Delta t \matA^T)\vecepsi^m,\boldsymbol{v} \big\rangle= \left \langle \matM \vecepsi^{m+1},\boldsymbol{v}\right\rangle -\Delta t \, \big \langle \vecrho^m,\boldsymbol{v} \big\rangle, \quad \forall \boldsymbol{v} \in \R^{\calN},  $$
 and then take $\boldsymbol{v}=\vecepsi^m$ to obtain
$$ \big\langle (\matM +\Delta t \matA^T)\vecepsi^m, \vecepsi^m \big \rangle= \langle \matM \vecepsi^{m+1},\vecepsi^m \rangle -\Delta t \, \big\langle \vecrho^m,\vecepsi^m\big\rangle.$$
We apply Cauchy-Schwarz inequality and use \eqref{ResidualNormDual}
\begin{equation}\label{DMproof1}
\big\langle (\matM +\Delta t \matA^T)\vecepsi^m, \vecepsi^m\big \rangle \leq  \|\matM^{1/2}\vecepsi^{m+1}\|\|\matM^{1/2}\vecepsi^{m}\|+\Delta t\,\|\vecrho^m\|_{\matA^{*^{-1}}}\;\|\vecepsi^m\|_{\matA^*}.    
\end{equation}
Similarly to the primal problem, we apply inequality \eqref{Young} twice and get 
\begin{align}
2 \|\matM^{1/2}\vecepsi^{m+1}\|\|\matM^{1/2}\vecepsi^{m}\| &\leq \left \langle \matM\vecepsi^{m+1},\vecepsi^{m+1} \right\rangle +\left\langle \matM\vecepsi^{m},\vecepsi^m \right\rangle, \label{DMproof2}\\
2 \|\vecrho^m\|_{\matA^{*^{-1}}}\|\vecepsi^m\|_{\matA^*} & \leq \frac1{\alpha_{\matA}} \|\vecrho^m\|_{\matA^{*^{-1}}}^2 + \alpha_{\matA}\|\vecepsi^m\|_{\matA^*}^2. \label{DMproof3}
\end{align}
Based on the definition of $\alpha_{\matA}$, \eqref{DMproof3} becomes 
\begin{equation}\label{DMproof4}
\begin{split}
2\|\vecrho^m\|_{\matA^{*^{-1}}}\|\vecepsi^m\|_{\matA^*} & \leq \frac1{\alpha_{\matA}} \|\vecrho^m\|_{\matA^{*^{-1}}}^2 + \big\langle  \matA_{\rm sym} \;\vecepsi^m, \vecepsi^m \big\rangle
\\ & =  \frac1{\alpha_{\matA}} \|\vecrho^m\|_{\matA^{*^{-1}}}^2 + \big\langle  \matA \vecepsi^m, \vecepsi^m \big\rangle, \quad \text{since} \; \big\langle  \matA_{\rm skew}\; \vecepsi^m, \vecepsi^m \big\rangle=0
\\ &= \frac1{\alpha_{\matA}} \|\vecrho^m\|_{\matA^{*^{-1}}}^2 + \big\langle   \matA^T\vecepsi^m, \vecepsi^m \big \rangle.
\end{split}
\end{equation}
Now, inequalities \eqref{DMproof1}--\eqref{DMproof4} lead to
\begin{equation}\label{DMproof5}
\big\langle \matM \vecepsi^{m},\vecepsi^m \big \rangle - \big\langle \matM\vecepsi^{m+1},\vecepsi^
{m+1} \big\rangle +  \Delta t \, \big\langle \matA^T  \vecepsi^m, \vecepsi^m \big\rangle \leq  \frac{\Delta t}{ \alpha_{\matA}} \|\vecrho^m\|_{\matA^{*^{-1}}}^2.
\end{equation}
Summing \eqref{DMproof5} over $\{n,\ldots,N-1\}$, we get
\begin{equation}\label{firstInqDualM}
\begin{split}
   \big\langle {\vecepsi^n},\matM {\vecepsi}^n\big\rangle + \Delta t \sum_{m=n}^{ N-1}\big \langle {\vecepsi}^{m}, \matA^T {\vecepsi}^{m} \big\rangle & \leq \frac{\Delta t}{ \alpha_{\matA}} \sum_{m=n}^{N-1} \|\vecrho^m\|_{\matA^{*^{-1}}}^2 + \big\langle {\vecepsi^N},\matM {\vecepsi}^N\big\rangle 
   \\ &  \leq \frac{\Delta t}{ \alpha_{\matA}} \sum_{m=n}^{N-1} \|\vecrho^m\|_{\matA^{*^{-1}}}^2 + \frac{\| \vecrho^N\|_{\matA^{*^{-1}}}^2 }{\alpha_{\matM}} \mbox{ using lemma \ref{lemma:dualErrorFinal}.}
\end{split}
\end{equation}
Inequality \eqref{firstInqDualM} yields for all $n\in \{0,\ldots,N-1\}$
\begin{equation}\label{single_dual_error_approx_non_fixed}
\begin{split} 
\left\langle {\vecepsi^n},\matM {\vecepsi}^n\right\rangle 
+ \Delta t \big\langle {\vecepsi}^{n}, \matA^T {\vecepsi}^{n} \big\rangle 
& \leq \frac{\Delta t}{ \alpha_{\matA}} 
\sum_{m=n}^{N-1} \|\vecrho^m\|_{\matA^{*^{-1}}}^2 + 
\frac{\| \vecrho^N\|_{\matA^{*^{-1}}}^2 }{\alpha_{\matM}}. 
\end{split}
\end{equation}
Therefore, summing over $\{0,\ldots,N-1\}$, we get
\begin{equation}\label{InqDual}
\sum_{m=0}^{N-1} \big[\left\langle {\vecepsi^m},\matM {\vecepsi}^m\right\rangle + \Delta t\, \big\langle {\vecepsi}^{m}, \matA^T {\vecepsi}^{m} \big\rangle \big] \leq \frac{\Delta t}{ \alpha_{\matA}} \sum_{m=0}^{N-1}(m+1) \|\vecrho^m\|_{\matA^{*^{-1}}}^2+\frac{T}{\Delta t} \frac{\|\vecrho^N\|_{\matA^{*^{-1}}}^2}{\alpha_{\matM}}
\end{equation}

Finally, using the definition \eqref{alpha} and the fact that
$$\big \langle {\vecepsi}^{m}, \matA^T {\vecepsi}^{m} \big\rangle = \big\langle \matA {\vecepsi}^{m}, {\vecepsi}^{m}
\big\rangle=\big\langle \matA_{\rm sym}\; {\vecepsi}^{m}, {\vecepsi}^{m}\big \rangle$$ 
lead to
\begin{equation}\label{eqn:proof_jana_dual_estim}
\orange{\Delta t\sum_{m=0}^{N-1}  \big\langle {\vecepsi}^{m}, \matA^*\; {\vecepsi}^{m} \big\rangle  \leq  \alpha^{-1}\left ( \frac{\Delta t^2}{ \alpha_{\matA}} \sum_{m=0}^{N-1}(m+1) \|\vecrho^m\|_{\matA^{*^{-1}}}^2+ T \frac{\|\vecrho^N\|_{\matA^{*^{-1}}}^2}{\alpha_{\matM} } \right ).}
\end{equation}
From \eqref{firstInqDualM}, 
using the definition \eqref{alphaE}, we also have 
\begin{equation}\label{eqn:proof_oscar_dual_estim}
\orange{\Delta t\sum_{m=0}^{N-1}  \big\langle {\vecepsi}^{m}, \matA^*\; {\vecepsi}^{m} \big\rangle  \leq  \alpha_{\matA}^{-1}\left (
\frac{\Delta t}{ \alpha_{\matA}} \sum_{m=0}^{N-1} \|\vecrho^m\|_{\matA^{*^{-1}}}^2 + \frac{\|\vecrho^N\|_{\matA^{*^{-1}}}^2}{\alpha_{\matM}}\right ) .}
\end{equation}

\subsection{Proof of Proposition \ref{prop3}}\label{Aproof3}
\begin{proof}
From \eqref{dualPb}, we have 
$$\big\langle \matM (\boldsymbol{\boldsymbol{\psi}}_{\calM,n}^k-\boldsymbol{\psi}_{\calM,n}^{k+1})+\Delta t \;\matA^T \boldsymbol{\psi}_{\calM,n}^k,\vece^{k+1}\big \rangle =0.   $$
We then sum over $k=0,\ldots,n-1$, to obtain
\begin{equation*}
    \langle \matM (\boldsymbol{\psi}_{\calM,n}^0-\boldsymbol{\psi}_{\calM,n}^1),\vece^1\rangle +\langle \matM (\boldsymbol{\psi}_{\calM,n}^1-\boldsymbol{\psi}_{\calM,n}^2),\vece^2\rangle +\ldots +\langle \matM (\boldsymbol{\psi}_{\calM,n}^{n-1}-\boldsymbol{\psi}_{\calM,n}^n),\vece^{n}\rangle
    +\Delta t\sum_{k=0}^{n-1} \langle \matA^T \boldsymbol{\psi}_{\calM,n}^k,\vece^{k+1} \rangle=0,        
\end{equation*}
which gives us  
\begin{equation*}
    \begin{split}
      \langle \matM \boldsymbol{\psi}_{\calM,n}^0,\vece^1\rangle-&\langle \matM \boldsymbol{\psi}_{\calM,n}^1,\vece^1\rangle +\langle \matM \boldsymbol{\psi}_{\calM,n}^1,\vece^2\rangle-\langle \matM \boldsymbol{\psi}_{\calM,n}^2,\vece^2\rangle +\ldots \\&+\langle \matM \boldsymbol{\psi}_{\calM,n}^{n-1},\vece^{n}\rangle-\langle \matM \boldsymbol{\psi}_{\calM,n}^{n},\vece_n^{n}\rangle+ \Delta t\sum_{k=0}^{n-1} \langle \matA^T \boldsymbol{\psi}_{\calM,n}^k,\vece^{k+1} \rangle=0,  
    \end{split}
\end{equation*}
leading to 
$$ \sum_{k=0}^{n-1} \langle \matM \boldsymbol{\psi}_{\calM,n}^k,\vece^{k+1} \rangle -\sum_{k=1}^{n-1} \langle \matM \boldsymbol{\psi}_{\calM,n}^k,\vece^k \rangle+ \Delta t\sum_{k=0}^{n-1} \langle \matA^T \boldsymbol{\psi}_{\calM,n}^k,\vece^{k+1} \rangle=\langle \matM \boldsymbol{\psi}_{\calM,n}^{n},\vece^{n} \rangle.$$
Since $\langle \matM \boldsymbol{\psi}_{\calM,n}^0,\vece^0 \rangle =0$, the above equation becomes
$$ \sum_{k=0}^{n-1} \langle \matM \boldsymbol{\psi}_{\calM,n}^k,\vece^{k+1} -\vece^k\rangle+ \Delta t\sum_{k=0}^{n-1} \langle \matA^T \boldsymbol{\psi}_{\calM,n}^k,\vece^{k+1} \rangle=\langle \matM \boldsymbol{\psi}_{\calM,n}^{n},\vece^{n}\rangle. $$
Using the final condition of the dual problem \eqref{dualPb}, we can write
\begin{equation}\label{proof}
\langle \matM \boldsymbol{\psi}_{\calM,n}^n,\vece^{n} \rangle=-\langle \boldsymbol{l}, \vecp_\calM^{n}- \vecp^{\mathsf{N}_{\rm pr},n} \rangle  = \sum_{k=0}^{n-1} [ \langle \matM \boldsymbol{\psi}_{\calM,n}^{k},\vece^{k+1}-\vece^{k} \rangle+\Delta t\;\langle \matA^T \boldsymbol{\psi}_{\calM,n}^k,\vece^{k+1} \rangle].
\end{equation}
Equation \eqref{proof} can be rewritten as  
\begin{equation}
\label{eq:diffQOI}
    \begin{split}
     \langle \boldsymbol{l}, \vecp_\calM^{n}- \vecp^{\mathsf{N}_{\rm pr},n} \rangle  &= -\sum_{k=0}^{n-1}  \langle \left (\matM +\Delta t \matA \right)\vecp_\calM^{k+1},\boldsymbol{\psi}_{\calM,n}^{k}\rangle+\sum_{k=0}^{n-1}  \langle \left (\matM +\Delta t \matA \right)\vecp^{\sfN_{\rm pr},k+1},\boldsymbol{\psi}_{\calM,n}^{k}\rangle
     \\ & \quad + \sum_{k=0}^{n-1} \langle \matM \vecp_\calM^{k},\boldsymbol{\psi}_{\calM,n}^k\rangle -\sum_{k=0}^{n-1} \langle \matM \vecp^{\sfN_{\rm pr},k},\boldsymbol{\psi}_{\calM,n}^k\rangle
     \\ &=- \Delta t \; \sum_{k=0}^{n-1} \langle \vecb, \boldsymbol{\psi}_{\calM,n}^k\rangle - \sum_{k=0}^{n-1} \langle \matM \vecp^{\sfN_{\rm pr},k},\boldsymbol{\psi}_{\calM,n}^k\rangle+\sum_{k=0}^{n-1} \langle \left(\matM+\Delta t  \matA\right)\vecp^{\sfN_{\rm pr},k+1},\boldsymbol{\psi}_{\calM,n}^k\rangle
     \\ & \quad + \, \sum_{k=0}^{n-1}\langle \left(\matM+\Delta t  \matA \right) \vecp^{\sfN_{\rm pr},k+1},\boldsymbol{\psi}_{n}^{\sfN_{\rm du},k}\rangle-\sum_{k=0}^{n-1}\langle \left(\matM+\Delta t  \matA \right) \vecp^{\sfN_{\rm pr},k+1},\boldsymbol{\psi}_{n}^{\sfN_{\rm du},k}\rangle
     \\& \quad + \,  \sum_{k=0}^{n-1}\langle \matM \vecp^{\sfN_{\rm pr},k},\boldsymbol{\psi}_{n}^{\sfN_{\rm du},k}\rangle -\sum_{k=0}^{n-1}\langle \matM \vecp^{\sfN_{\rm pr},k},\boldsymbol{\psi}_{n}^{\sfN_{\rm du},k}\rangle+\Delta t \; \sum_{k=0}^{n-1}\langle \vecb,\boldsymbol{\psi}_{n}^{\sfN_{\rm du},k}\rangle - \Delta t \; \sum_{k=0}^{n-1}\langle \vecb,\boldsymbol{\psi}_{n}^{\sfN_{\rm du},k}\rangle
     \\ &= \Delta t \;\sum_{k=0}^{n-1} \langle \vecr^{k+1},\boldsymbol{\psi}_{\calM,n}^k-\boldsymbol{\psi}_n^{\sfN_{\rm du},k}\rangle +  \Delta t \;\sum_{k=0}^{n-1} \langle \vecr^{k+1},\boldsymbol{\psi}_n^{\sfN_{\rm du},k}\rangle
     \\ &=  \Delta t \;\sum_{k=0}^{n-1} \langle \vecr^{k+1},\boldsymbol{\Psi}_\calM^{N-n+k}-\boldsymbol{\Psi}^{\sfN_{\rm du},N-n+k}\rangle +  \Delta t \;\sum_{k=0}^{n-1} \langle \vecr^{k+1},\boldsymbol{\Psi}^{\sfN_{\rm du},N-n+k}\rangle
     \\ &=  \Delta t \;\sum_{k=0}^{n-1} \langle \vecr^{k+1},\vecepsi^{N-n+k}\rangle +  \Delta t \;\sum_{k=0}^{n-1} \langle \vecr^{k+1},\boldsymbol{\Psi}^{\sfN_{\rm du},N-n+k}\rangle.
    \end{split}
\end{equation}

\begin{itemize}
    \item \textbf{First choice.}  The error bound is evaluated according to 
    \begin{equation}
        | {\scas}^{n}-  {\scas}^{\mathsf{N}_{\rm s},n}|= \left | \sum_{k=0}^{n-1}  \Delta t \, \langle \vecr^{k+1},\vecepsi^{N-n+k}\rangle \right |.
    \end{equation}
    We use \eqref{ResidualNormPrimal} and Cauchy Schwarz inequality to obtain 
    \begin{equation}\label{first_QOI_CS}
    | {\scas}^{n}-  {\scas}^{\mathsf{N}_{\rm s},n}|\leq \Delta t\sum_{k=0}^{n-1} \| \vecr^{k+1}\|_{\matA^{*^{-1}}} \| \vecepsi^{N-n+k}\|_{\matA^*}.
    \end{equation}
    For $n=N$, \eqref{first_QOI_CS} gives
    \begin{equation}\label{first_QOI_CS_N}
    | {\scas}^{N}-  {\scas}^{\mathsf{N}_{\rm s},N}|\leq \Delta t\sum_{k=0}^{N-1} \| \vecr^{k+1}\|_{\matA^{*^{-1}}} \| \vecepsi^{k}\|_{\matA^*} .
    \end{equation} 
Applying the Cauchy Schwartz inequality again on \eqref{first_QOI_CS_N} we write
\begin{equation}
   |{\scas}^{N}-  {\scas}^{\mathsf{N}_{\rm s},N}| \leq \Delta t\Big(\sum_{k=0}^{N-1} \,\|\vecr^{k+1}\|_{\matA^{*^{-1}}}^2\Big)^{\!1/2}\Big(\sum_{k=0}^{N-1} \,\|\vecepsi^{k}\|_{\matA^*}^2\Big)^{\!1/2}.
\end{equation}
\orange{Then, one can use \eqref{eqn:proof_jana_dual_estim} to obtain}
\begin{equation}
|{\scas}^{N}-  {\scas}^{\mathsf{N}_{\rm s},N}| \leq \frac{\Delta t}{\orange{\alpha^{\!1/2}}} \Big(\sum_{n=1}^N   \|\vecr^{n}\|_{\matA^{*^{-1}}}^2\Big)^{\!1/2} \left ( \frac{\Delta t}{ \orange{\alpha_{\matA}}} \sum_{m=0}^{N-1}(m+1) \|\vecrho^m\|_{\matA^{*^{-1}}}^2 +\frac{T}{\Delta t} \frac{\|\vecrho^N\|_{\matA^{*^{-1}}}^2}{\orange{\alpha_{\matM}}}\right )^{\!1/2}=: \Delta_s^N
\end{equation}
\orange{or \eqref{eqn:proof_oscar_dual_estim} to obtain}
\begin{equation}
|{\scas}^{N}-  {\scas}^{\mathsf{N}_{\rm s},N}| \leq \frac{\Delta t}{\orange{\alpha_{\matA}^{\!1/2}}}\Big(\sum_{n=1}^N   \|\vecr^{n}\|_{\matA^{*^{-1}}}^2\Big)^{\!1/2} \left (\frac{1}{\orange{\alpha_{\matA}}}\sum_{n=1}^N   \|\vecrho^{n}\|_{\matA^{*^{-1}}}^2+ \frac{1}{\Delta t}\frac{\| \vecrho^N\|_{\matA^{*^{-1}}}^2 }{\orange{\alpha_{\matM}}}\right )^{\!1/2}=: \overline{\Delta}_s^N.
\end{equation}
\orange{Definition \eqref{alpha} and \eqref{single_dual_error_approx_non_fixed} leads to}
    \begin{equation}
       \orange{\alpha}\| \vecepsi^{n}\|_{\matA^*}^2\leq \frac{\Delta t}{ \orange{\alpha_{\matA}}} \sum_{m=n}^{N-1} \|\vecrho^m\|_{\matA^{*^{-1}}}^2 + \frac{\| \vecrho^N\|_{\matA^{*^{-1}}}^2 }{\orange{\alpha_{\matM}}}.
    \end{equation}
Inserting the previous inequality into \eqref{first_QOI_CS_N} we obtain
    \begin{equation}
 |{\scas}^{N}-  {\scas}^{\mathsf{N}_{\rm s},N}| \leq \frac{\Delta t}{\orange{\alpha^{\!1/2}}}\sum_{k=0}^{N-1} \| \vecr^{k+1}\|_{\matA^{*^{-1}}} \left (\frac{\Delta t}{ \orange{\alpha_{\matA}}} \sum_{m=k}^{N-1} \|\vecrho^m\|_{\matA^{*^{-1}}}^2 + \frac{\| \vecrho^N\|_{\matA^{*^{-1}}}^2 }{\orange{\alpha_{\matM}}}\right)^{1/2}=: \hat{\Delta}_s^N.
\end{equation}
\item \textbf{Second choice.} \eqref{QOIEstWC}, \eqref{QOIEstWCbis}, \eqref{QOIEstWCter} follow from \eqref{eq:diffQOI} and the previous definitions.
\end{itemize}

It concludes the proof of proposition \ref{prop3}.
\end{proof}

\section{Successive constraint method}\label{SCM}
Let $\Xi$ be a set of parameter values. For each $\xi\in \Xi$, the successive constraint method (SCM) consists in finding an upper bound $\alpha_{\rm{UB}}(\xi)$ and a lower bound $\alpha_{\rm{LB}}(\xi)$ of the coercivity constant $\alpha(\xi)$ through an offline-online strategy.
The SCM relies on the affine decomposition assumption \eqref{affineDecomp}, which enables us to express $\alpha(\xi)$ as 
\begin{equation}\label{trueAlpha}
\alpha(\xi)= \inf_{\boldsymbol{v} \in \mathbb{R}^{\mathcal{N}}} \sum_{d=1}^{D_A}\theta_d^A(\xi)\frac{\boldsymbol{v}^T\matA_d\boldsymbol{v}}{\|\boldsymbol{v}\|_{\boldsymbol{A}^*}^2}=\inf_{\boldsymbol{v} \in \mathbb{R}^{\mathcal{N}}} \sum_{d=1}^{D_A}\theta_d^A(\xi)w_d.
\end{equation}
To define the lower bound $\alpha(\xi)$, we express \eqref{trueAlpha} as a  minimization problem  
\begin{equation}\label{miniProb}
    \alpha(\xi)=\inf_{\vecw\in \boldsymbol{\mathcal{W}}} \mathcal{J}(\xi,w),
\end{equation}
where the set $\boldsymbol{\mathcal{W}}$ is defined as
$$\boldsymbol{\mathcal{W}}:= \Big\{\vecw=\left(w_1,\ldots,w_{D_A}\right) \in \mathbb{R}^{D_A} \mid \exists \; \boldsymbol{v}\in \mathbb{R}^{\mathcal{N}} \; \text{s.t.}\;\; w_d=\frac{\boldsymbol{v}^T\matA_d\boldsymbol{v}}{\|\boldsymbol{v}\|_{\boldsymbol{A}^*}^2},\; 1\leq d \leq D_A \Big\}, $$
and the objective function is given by
\begin{align*}
  \boldsymbol{\mathcal{J}} \colon \Xi \times \mathbb{R}^{D_A} & \to  \mathbb{R}\\
 (\xi,\vecw) &\mapsto  \mathcal{J}(\xi,\vecw)= \sum_{d=1}^{D_A}\theta_d^A(\xi) w_d.
\end{align*}
The idea of the SCM is based on creating two sets  $\boldsymbol{\mathcal{W}}_{\rm LB}$ and  $\boldsymbol{\mathcal{W}}_{\rm UB}$, such that $\boldsymbol{\mathcal{W}}_{\rm UB} \subset \boldsymbol{\mathcal{W}} \subset \boldsymbol{\mathcal{W}}_{\rm LB}$, where we perform the minimization over these two sets and define
$$ \alpha_{\rm LB}(\xi)=\min_{\vecw\in \boldsymbol{\mathcal{W}}_{\rm LB}} \mathcal{J}(\xi,\vecw) \qquad \text{and} \qquad   \alpha_{\rm UB}(\xi)=\min_{\vecw\in \boldsymbol{\mathcal{W}}_{\rm UB}} \mathcal{J}(\xi,\vecw).$$
\paragraph*{Definition of $\boldsymbol{\mathcal{W}}_{\rm UB}$.}
We introduce the subset of parameter values $\Xi_{\rm M} \subset \Xi $ obtained using a greedy algorithm (see Algorithm \ref{alg:SCMAlgo}). The construction of $\Xi_{\rm M}$ requires a training set $\Xi_{\rm training}$ and a fixed tolerance $0\leq \rm{tol} \leq1$ that controls the relative gap between the lower and upper bounds.
 \begin{algorithm}[ht]
 \textbf{Input}: $\rm{tol},\Xi.$
 \begin{algorithmic}[1]
 \caption{Construction of $\Xi_{\rm M}$}
 \State Choose arbitrary $\xi_1\in \Xi$.
 \State Set $j=1$ and $\Xi_j=\{\xi_1\}$.
 \State Compute $\eta_{j}(\xi)=\frac{\alpha_{\rm{UB}}(\xi)-\alpha_{\rm{LB}}(\xi)}{\alpha_{\rm{UB}}(\xi)}$.
 \While{$\underset{\xi \in \Xi}{\max}\;\eta_{j}(\xi) > \rm{tol}$}
 \State Compute  $\xi_{j+1}=\arg {\underset{\xi \in \Xi}{\max}}\;\eta_{j}(\xi)$.
 \State Set $\Xi_{j+1}=\Xi_{j}\cup \{\xi_{j+1}\}$.
 \State ${j} \gets {j+1}.$
 \State  $\eta_{j}(\xi)=\frac{\alpha_{\rm{UB}}(\xi)-\alpha_{\rm{LB}}(\xi)}{\alpha_{\rm{UB}}(\xi)}$.
 \EndWhile
 \label{alg:SCMAlgo}
 \end{algorithmic}
 \end{algorithm}

For all $1\leq j \leq \rm M$ and for each $\xi_j \in \Xi_{\rm M}$,
\begin{enumerate}
    \item we assemble $\matA(\xi_j)=\sum_{d=1}^{D_A}\theta_d^A(\xi_j) \matA_d$,
    \item we solve the generalized eigenvalue problem
    \begin{equation}\label{eigenValuePb}
        \matA (\xi_j)\boldsymbol{y}=\lambda \boldsymbol{A}^* \boldsymbol{y},
    \end{equation}
    and extract the smallest eigenvalue $\alpha^j$ and its corresponding eigenvector $\boldsymbol{v}^j$,
    \item we compute the vector $\boldsymbol{w}^j \in \mathbb{R}^{D_A}$ such that
    $$(\boldsymbol{w}^j)_d=\frac{(\boldsymbol{v}^j)^T \matA_d \boldsymbol{v}^j}{\|\boldsymbol{v}^j\|_{\boldsymbol{A}^*}^2}$$,
    \item  we define the set
    $$ \boldsymbol{\mathcal{W}}_{\rm UB}=\left\{ \boldsymbol{w}^j \mid 1\leq j \leq \rm M \right\},$$
    and compute the upper bound
    $$\alpha_{\rm UB}(\xi)= \underset{\vecw \in \boldsymbol{\mathcal{W}}_{\rm UB}}{\mathrm{arg\;min}}\mathcal{J}(\xi,\vecw). $$
\end{enumerate}

\paragraph*{Definition of $\boldsymbol{\mathcal{W}}_{\rm LB}$.} First, we need to introduce the constraint interval 
$$ \boldsymbol{\mathcal{B}}= \prod_{d=1}^{D_A}\left[\inf_{\boldsymbol{v} \in \mathbb{R}^{\mathcal{N}}} \frac{\boldsymbol{v}^T\matA_d\boldsymbol{v}}{\|\boldsymbol{v}\|_{\boldsymbol{A}^*}^2}, \sup_{\boldsymbol{v} \in \mathbb{R}^{\mathcal{N}}}\frac{\boldsymbol{v}^T\matA_d\boldsymbol{v}}{\|\boldsymbol{v}\|_{\boldsymbol{A}^*}^2}\right], $$
obtained by computing, once at the beginning of the SCM, the smallest and largest eigenvalues of a problem similar to \eqref{eigenValuePb} and obtained by replacing $\matA (\xi_j)$
by $\matA_d$. 
We define the set 
$$ \begin{aligned}
    \boldsymbol{\mathcal{W}}_{\rm LB}^j(\xi)=\left\{ \vecw \in \boldsymbol{\mathcal{B}} \mid \boldsymbol{\mathcal{J}}(\xi^{'};\vecw) \geq \alpha(\xi^{'}), \quad \forall \xi^{'}\in P_{\rm M_1}(\xi;\Xi_j); \right. \\ \left.  \boldsymbol{\mathcal{J}}(\xi^{'};\vecw) \geq \alpha^{j-1}(\xi^{'}),\quad \forall \xi^{'} \in P_{\rm M_2}(\xi;\Xi \backslash \Xi_j)\right \},
\end{aligned}
$$
where $ P_{\rm M}(\xi;\mathbb{D}):=\left\{ \rm M \text{ closest points to } \xi \text{ in } \mathbb{D} \right\}$.

\section{Empirical interpolation method}\label{EIM}

The efficiency of the RB method relies on the affine decomposition \eqref{affineDecomp} proposed in Section \ref{CompAspect}. However, this decomposition is not always available. But the empirical interpolation method (EIM) can provide one to approximate, in our case,  $\Hat{\boldsymbol{v}}(\xi)$ with an affine sum. Given a family of parameter-dependent vectors $\mathscr{T}=\{ \Hat{\boldsymbol{v}}(\xi)\in \R^F; \xi\in \Xi_{\rm training} \}$, the EIM aims at finding an approximation to the elements of $\mathscr{T}$ through an operator $\mathcal{I}_{\rm M_{ EIM}}$ that interpolates the vector $ \Hat{\boldsymbol{v}}(\xi)$ at some selected points. 
Using a greedy process, we construct the set of vectors $\{\tilde{\boldsymbol{v}}^1,\ldots, \tilde{\boldsymbol{v}}^{\rm M_{EIM}}\}$ and the interpolation points $\{x_1,\ldots,x_{\rm M_{EIM}}\}$ such that 
\begin{subequations}
\begin{align}
   \mathcal{I}_{\rm M_{ EIM}}[ \Hat{\boldsymbol{v}}(\xi)]\approx\sum_{i=1}^{\rm M_{EIM}} \theta_i(\xi)  \tilde{\boldsymbol{v}}^i,\label{subeq2}
\end{align}
\end{subequations}
where $\theta_d(\xi)\in \R$ and $\tilde{\boldsymbol{v}}^i\in \R^F$, $1\leq i \leq \rm M_{EIM}$, do not depend on $\xi$.

To begin the procedure, we randomly choose $\xi_1$ from $\Xi_{\rm training}$ and set $ \Hat{\boldsymbol{v}}^1= \Hat{\boldsymbol{v}}(\xi_1)$. The first interpolation point is chosen such that 
$$ x_1=\arg \max_{1\leq j \leq F } | \Hat{\boldsymbol{v}}_{j}^1|, $$
where $\Hat{\boldsymbol{v}}_{j}^1$ is the $j$-th element of $\Hat{\boldsymbol{v}}^1$.
We then initialize the first basis function as 
$$ \tilde{\boldsymbol{v}}^1=\Hat{\boldsymbol{v}}^1/\Hat{\boldsymbol{v}}_{j_1}^1, $$
with $1\leq j_1\leq F$, the index corresponding to the selected point $x_1$. At the $m$-th step, $m=2,\ldots,\rm{ M_{ EIM}}-1$, given the set of interpolations points $\{x_1,\ldots,x_{\rm M_{EIM}-1}\}$ and the set of basis elements $\{ \tilde{\boldsymbol{v}}^1,\ldots, \tilde{\boldsymbol{v}}^{\rm M_{EIM}-1}\}$, we select the next snapshot as the worst approximated one by the current interpolant. To do so, we first
write the $m$ equations stating the equality between the current EIM approximation and a vector $\Hat{\boldsymbol{v}}(\xi)$ at the current $m$ interpolation points. This leads to the following lower triangular linear system:
\[
\begin{bmatrix}
   1  & 0&\cdots & 0 & 0 \\
\tilde{\boldsymbol{v}}_{j_2}^1  & 1 &\cdots &0 &0 \\
  \vdots &\vdots & & \vdots & \vdots \\
  \tilde{\boldsymbol{v}}_{j_{\rm M_{EIM-1}}}^1  &  \tilde{\boldsymbol{v}}_{j_{\rm M_{EIM-1}}}^2
   &\cdots &1 &0 \\
    \tilde{\boldsymbol{v}}_{j_{\rm M_{EIM}}}^1 & \tilde{\boldsymbol{v}}_{j_{\rm M_{EIM}}}^2 &\cdots& \tilde{\boldsymbol{v}}_{j_{\rm M_{EIM}}}^{\rm M_{EIM}} &1
\end{bmatrix}
\begin{bmatrix}
    \theta_1  \\
    \theta_2    \\
      \vdots \\
       \theta_{\rm{M_{EIM}-1}} \\
        \theta_{\rm{M_{EIM}}} 
\end{bmatrix}
(\xi)=
\begin{bmatrix}       
     \Hat{\boldsymbol{v}}_{j_1}  \\
      \Hat{\boldsymbol{v}}_{j_2}\\
      \vdots \\
      \Hat{\boldsymbol{v}}_{j_{\rm M_{EIM}-1}}\\
  \Hat{\boldsymbol{v}}_{j_{\rm M_{EIM}}}
\end{bmatrix}(\xi).
\]
We choose 
\begin{subequations}
\begin{align}
\xi_{m+1} = \arg \max_{\xi \in \Xi_{\rm training}}\| \Hat{\boldsymbol{v}}(\xi)-\mathcal{I}_{m}[ \Hat{\boldsymbol{v}}(\xi)]\|_{L^{\infty}}.
\end{align}
\end{subequations}
The $(m+1)$-th interpolation point is then defined as 
$$x_{m+1}=\arg \max_{1\leq j \leq F} | r_j^{m+1}|$$ with $\boldsymbol{r}^{m+1}=\Hat{\boldsymbol{v}}(\xi_{m+1})-\mathcal{I}_{m}[ \Hat{\boldsymbol{v}}(\xi_{m+1})]$
and the corresponding basis vector is taken as 
$$ \tilde{\boldsymbol{v}}^{m+1}= \boldsymbol{r}^{m+1}/ r_{j_{m+1}}^{m+1}.$$
We repeat this procedure until a given tolerance $\epsilon_{\rm{EIM}}>0$ is reached, i.e.
$$ \max_{\xi \in \Xi_{\rm training}} \|\Hat{\boldsymbol{v}}(\xi)-\mathcal{I}_{m}[\Hat{\boldsymbol{v}}(\xi)]\|_{L^{\infty}} < \epsilon_{\rm{EIM}}. $$

\end{document}